\documentclass[10pt]{amsart}

\usepackage{amsmath, amsthm, amssymb, graphicx, enumitem, everypage, datetime, getfiledate}
\newcounter{forwardtheorem}

\newcounter{citedtheorems}

\newtheorem{defn}{Definition}[section]
\newtheorem{theorem}[defn]{Theorem}
\newtheorem*{theorem-x}{Theorem}
\newtheorem*{theorem-m}{Main Theorem}
\newtheorem*{theorem-n}{Main Theorem}
\newtheorem*{cor-x}{Corollary}
\newtheorem*{lemma-x}{Lemma}
\newtheorem*{concl-x}{Conclusion}
\newtheorem*{claim-x}{Claim}
\newtheorem*{thm-r}{Theorem \ref{concl:sop2-max}}
\newtheorem*{thm-q}{Theorem \ref{theorem:p-t}}
\newtheorem*{claim-s}{Claim \ref{m1-sat}}

\newtheorem{prob}[defn]{Problem}

\newtheorem{thm-lit}[citedtheorems]{Theorem}
\newtheorem{defn-lit}[citedtheorems]{Definition}
\newtheorem*{ftheorem}{Theorem}
\newtheorem{fact}[defn]{Fact}
\newtheorem{cor}[defn]{Corollary}

\newtheorem{concl}[defn]{Conclusion}
\newtheorem{conv}[defn]{Convention}
\newtheorem{conv-r}[defn]{Conventions and Remarks}
\newtheorem{claim}[defn]{Claim}

\newtheorem{lemma}[defn]{Lemma}
\newtheorem{obs}[defn]{Observation}
\newtheorem{rmk}[defn]{Remark}

\newtheorem{disc}[defn]{Discussion}

\newtheorem{qst}[defn]{Question}

\newtheorem{hyp}[defn]{Hypothesis}

\newcommand{\spec}{\operatorname{ESS}}
\newcommand{\truespec}{\operatorname{spec}}

\newcommand{\xt}{\mathfrak{t}}

\newcommand{\xr}{\operatorname{val}}
\newcommand{\xp}{\mathfrak{p}}

\newcommand{\vp}{\varphi}

\newcommand{\uu}{\mathcal{U}}

\newcommand{\uf}{\operatorname{uf}}

\newcommand{\trg}{{T_{rg}}}
\newcommand{\tr}{\operatorname{Tr}}
\newcommand{\tp}{\operatorname{tp}}

\newcommand{\tlf}{\triangleleft}

\newcommand{\step}{\vspace{3mm}\noindent\emph}
\newcommand{\st}{\mathbf{S}}

\newcommand{\SOP}{\operatorname{SOP}}

\newcommand{\rstr}{\upharpoonright}

\newcommand{\rn}{\operatorname{range}}

\newcommand{\pr}{\operatorname{Pr}}

\newcommand{\os}{\mathcal{S}}
\newcommand{\ord}{\operatorname{Or}}
\newcommand{\psord}{\operatorname{PsOr}}

\newcommand{\ob}{\operatorname{ob}}

\newcommand{\NSOP}{\operatorname{NSOP}}

\newcommand{\ml}{\mathcal{L}}

\newcommand{\mct}{\mathcal{T}}
\newcommand{\mcp}{\mathcal{P}}

\newcommand{\mch}{\mathcal{H}}

\newcommand{\mc}{\mathcal{C}}

\newcommand{\mb}{\mathbf{b}}

\newcommand{\ma}{\mathbf{a}}

\newcommand{\lost}{\L os' }

\newcommand{\lgn}{\operatorname{lg}}
\newcommand{\lgb}{\operatorname{lgd}}
\newcommand{\rep}{\operatorname{rep}_B}

\newcommand{\lcf}{\operatorname{lcf}}

\newcommand{\ii}{\mathbf{i}}

\newcommand{\HF}{\mathbf{HF}}

\newcommand{\gC}{{\mathfrak C}}

\newcommand{\frt}{\operatorname{frt}}

\newcommand{\eff}{\mathcal{F}}

\newcommand{\dom}{\operatorname{dom}}

\newcommand{\Ded}{\operatorname{Ded}}
\newcommand{\de}{\mathcal{D}}

\newcommand{\cZ}{{\mathcal Z}}
\newcommand{\cY}{{\mathcal Y}}
\newcommand{\cx}{\mathcal{X}}
\newcommand{\cX}{{\mathcal X}}
\newcommand{\cW}{{\mathcal W}}
\newcommand{\cU}{{\mathcal U}}

\newcommand{\cts}{\mc^{\mathrm{ct}}(\cs)}
\newcommand{\cT}{{\mathcal T}}
\newcommand{\cs}{\mathbf{s}}

\newcommand{\cl}{\operatorname{cl}}

\newcommand{\cis}{\operatorname{cis}}

\newcommand{\cf}{\operatorname{cf}}

\newcommand{\bx}{\mathbf{x}}

\newcommand{\br}{\vspace{2mm}}

\newcommand{\bc}{\mathbf{c}}

\newcommand{\Av}{\operatorname{Av}}

\newcommand{\AP}{\operatorname{AP}}

\newcommand{\sub}{\operatorname{Sub}}

\title{Model-theoretic applications of cofinality spectrum problems}

\author{M. Malliaris and S. Shelah}
\thanks{\emph{Thanks:}
Malliaris was partially supported by NSF grants 1300634 and by a Sloan fellowship.
Shelah was partially supported by the Israel Science Foundation grants 710/07 and 1053/11 
and European Research Council grant 338821. 
This is manuscript 1051 in Shelah's list of publications.}
\address{Department of Mathematics, University of Chicago, 5734 S. University Avenue, Chicago, IL 60637, USA and 
Mathematical Sciences Research Institute, 17 Gauss Way, Berkeley, CA 94720, USA}
\email{mem@math.uchicago.edu}

\address{Einstein Institute of Mathematics, Edmond J. Safra Campus, Givat Ram, The Hebrew
University of Jerusalem, Jerusalem, 91904, Israel, and Department of Mathematics,
Hill Center - Busch Campus, Rutgers, The State University of New Jersey, 110
Frelinghuysen Road, Piscataway, NJ 08854-8019 USA}
\email{shelah@math.huji.ac.il}
\urladdr{http://shelah.logic.at}

\begin{document}

%\AddEverypageHook{\testhook}

%\AddEverypageHook{\newdatehook}

\begin{abstract}
We apply the recently developed technology of cofinality spectrum problems to prove a range of theorems in model theory.
First, we prove that any model of Peano arithmetic is $\lambda$-saturated iff it has cofinality $\geq \lambda$ and the underlying order has no $(\kappa, \kappa)$-gaps 
for regular $\kappa < \lambda$. 
%We also answer a question about balanced pairs of models of PA. 
Second, assuming instances of GCH, we prove that $SOP_2$ characterizes maximality in the 
interpretability order $\tlf^*$, settling a prior conjecture and proving that $SOP_2$ is a real dividing line. 
Third, %in light of this result, 
we establish the beginnings of a structure theory for $NSOP_2$, proving for instance 
that $NSOP_2$ can be characterized in terms of few inconsistent higher formulas. 
In the course of the paper, we show %we improve our previous theorems on cofinality spectrum problems by showing 
that $\xp_\cs = \xt_\cs$ in \emph{any} weak cofinality spectrum problem closed under exponentiation (naturally defined).  
We also prove that the local versions of these cardinals need not coincide, 
even in cofinality spectrum problems arising from Peano arithmetic.  
\end{abstract}



\maketitle

In a recent paper \cite{MiSh:998}  
we connected and solved two a priori unrelated open questions: 
the question from model theory of whether $SOP_2$ is maximal in Keisler's order, and the question 
from set theory/general topology of whether $\xp = \xt$. This work was described in the research announcement \cite{MiSh:E74} and the commentary \cite{Moore}. 
In order to prove these theorems, we introduced a general framework called {cofinality spectrum problems}, 
reviewed in \S \ref{s:preparation} below. 

In the present paper, we develop and apply cofinality spectrum problems to a range of problems in model theory, primarily on Peano arithmetic and 
around the strong tree property $SOP_2$, also called the $2$-strong order property. 
We prove the following theorems:

\begin{ftheorem}[Theorem \ref{t:pa}]
Let $N$ be a model of Peano arithmetic, or just bounded PA, and $\lambda$ an uncountable cardinal. 
If the reduct of $N$ to the language of order has cofinality $>\kappa$ and no $(\kappa, \kappa)$-cuts for all $\kappa < \lambda$, then $N$ is $\lambda$-saturated. 
\end{ftheorem}

On the earlier history, which involves \cite{pabion} and \cite{Sh:a} VI.2-3, see \S \ref{s:peano-arithmetic}.  
We also address a question about balanced pairs of models of PA. %. in Peano arithmetic of Kossak and Schmerl. 
The proof of Theorem \ref{t:pa} relies on the answer to a question arising from \cite{MiSh:998}, of intrinsic interest: 

\begin{ftheorem}[Theorem {\ref{ps-is-ts}}]
Let $\cs$ be a cofinality spectrum problem which is closed under exponentiation. Then 
$\xp_\cs = \xt_\cs$. 
\end{ftheorem}

As explained in \S \ref{s:preparation}, Theorem \ref{ps-is-ts} complements a main theorem of \cite{MiSh:998}, 
which showed that $\xt_\cs \leq \xp_\cs$ for any cofinality spectrum problem $\cs$. As a consequence, 
% of Theorem $\ref{ps-is-ts}$, 
we are able to characterize $\xt_\cs$ in terms of the first symmetric cut. %, Theorem \ref{ps-is-ts}.  
However, as we show in \S \ref{s:local}, 
the local versions of these cardinals, $\xp_{\cs, \ma}$, $\xt_{\cs, \ma}$ need not coincide unless the underlying model 
is saturated. 

We then turn to the strong tree property $SOP_2$. A major result of \cite{MiSh:998} was that $SOP_2$ suffices 
for being maximal in Keisler's order $\tlf$. It was not proved to be a necessary condition, but we conjectured there that $SOP_2$ 
characterizes maximality in Keisler's order. The difficulty in addressing this question may be in building ultrafilters. However, 
in the present paper, for a related open problem, 
we give a complete answer: 

\begin{ftheorem}[Theorem \ref{tstar-sop2-max}, under relevant instances of GCH]
$T$ is $\tlf^*$-maximal if and only if it has $SOP_2$.
\end{ftheorem}

The ordering $\tlf^*$ %of D\{v}zamonja and Shelah \cite{DzSh:692} 
refines Keisler's order, but is defined not in terms of ultrapowers but 
rather in terms of interpretability. 
Theorem \ref{tstar-sop2-max} answers a very interesting question going back to D\v{z}amonja and Shelah \cite{DzSh:692} and Shelah and Usvyatsov \cite{ShUs:844}.  
We inherit the assumption of a case of GCH from \cite{ShUs:844}, where one direction of the theorem was proved, building on work of 
\cite{DzSh:692}. The direction proved here is in ZFC. 
%A proof eliminating the case of GCH is in preparation. 

Theorem \ref{tstar-sop2-max} gives decisive evidence for $SOP_2$ being a dividing line, by giving the equivalence of a natural inside/syntactic property 
and an outside property. However, this was done without developing a structure theory. So in \S \ref{s:1198}, we develop the beginnings of a structure theory 
for $NSOP_2$. We define a notion of `higher formulas' using ultrafilters  %in terms of consistency along certain average sequences arising from ultrafilters 
and prove, for example, that: 
\begin{ftheorem}[Theorem \ref{a13}]
$T$ is $NSOP_2$ iff for all infinite $A$, the number of %$($so called$)$ 
pairwise $1$-contradictory higher $\vp$-formulas over $A$ is $\leq |A|$. 
\end{ftheorem}

%Here $\HF$ stands for ``higher formulas.'' 
Section \ref{s:1198} contains several other results, notably the ``symmetry lemma'' 
states that $NSOP_3$ can be characterized in terms of symmetry of inconsistency for these higher formulas. 
We also prove that $SOP_2$ is sufficient for a certain exact-saturation spectrum to be empty, 
connecting to work of Shelah \cite{Sh:900} and Kaplan and Shelah \cite{KpSh:F1473}.  
The results seem persuasive and we believe they begin a structure theory for $NSOP_2$ theories. 
The introductions to each section contain further motivation for these results and discussions of prior work. 
%Accounts of \cite{MiSh:998} also appear in Malliaris and Shelah \cite{MiSh:E74}, which gives some model-theoretic motivation,    
%and in Moore \cite{Moore}, which gives context for the general topology result that $\xp=\xt$.

Sections 5 and 8 may be read essentially independently of prior sections. Section 7 uses one earlier theorem; the reader 
interested in the $\tlf_*$ result may wish to begin there and refer back as needed. 

\setcounter{tocdepth}{2}

%\dominitoc
\tableofcontents

\section{Background on cofinality spectrum problems} \label{s:preparation}

Roughly speaking, a cofinality spectrum problem $\cs = (M, M_1, M^+, M^+_1, \Delta)$ is given by:

\begin{itemize}
\item an elementary pair of models $(M, M_1)$ which can be expanded to the elementary pair $(M^+, M^+_1)$, and
\item a set $\Delta$ of formulas in $\tau(M)$ defining discrete linear orders, closed under finite Cartesian products, 
\end{itemize}
such that in the expanded model $M^+_1$,
\begin{itemize}
\item  each instance of a formula in $\Delta$ defines a ``pseudofinite'' linear order, meaning that each of its $M^+_1$-definable subsets has 
a first and last element, 
\item on at least one of the Cartesian products, the ordering is well behaved, e.g. like the G\"odel pairing function,\footnote{G\"odel's pairing function orders pairs of ordinals first by maximum, then by first coordinate, then by second coordinate. If e.g. there is a total linear order on the model, it makes sense to compare (``maximum'') elements in any two such 
ordered sets; if not, this requirement may be satisfied by some $\ma \times \ma$, that is the Cartesian product of a given ordered set with itself.}

and
\item $(M^+, M^+_1, \Delta)$ has ``enough set theory for trees,'' meaning essentially that for each $\vp \in \Delta$ there is 
$\psi \in \tau(M^+_1)$ so that for each linear order defined by an instance of $\vp$, there is a tree defined by an instance of $\psi$:
\begin{itemize}
\item whose 
nodes are functions from that order to itself, of length bounded by a distinguished element $d$ of the order, and 
\item the basic operations on this tree (the partial order 
on nodes given by initial segment, the length of a node i.e. size of its domain, the value of a node at an element of its domain, and concatenation 
of an additional value) are likewise uniformly definable.
\end{itemize} 
\end{itemize}

Natural and motivating examples of cofinality spectrum problems 
may be constructed beginning with pairs of models of Peano arithmetic or pairs of models of set theory for $(M^+, M^+_1)$ 
with $M$ being a reduct to a language containing linear order;  
or beginning with a pair $(M, M_1)$ where $M_1$ is an ultrapower of $M$, using the theorem that ultrapowers commute with reducts. 
See \cite{MiSh:998} \S 2.3 for more details on these examples.  
The fruitful idea was simply that, in some sufficiently rich model, 
one may study the amount of saturation in an underlying linear order in relation to the fullness of the derived trees.   

Here are the formal definitions, originally given in \cite{MiSh:998}.

\begin{defn}[Enough set theory for trees, \cite{MiSh:998} Definition 2.3]  \label{d:estt} 
Let $M_1$ be a model and $\Delta$ a nonempty set of formulas in the language of $M_1$.
We say that $(M_1, \Delta)$ has \emph{enough set theory for trees} when the following conditions are true.

\begin{enumerate}
\item $\Delta$ consists of first-order formulas $\vp(\bar{x},\bar{y};\bar{z})$, with 
$\ell(\bar{x}) = \ell(\bar{y})$. %, not necessarily equal to $1$. 

\item For each $\vp \in \Delta$ and each parameter $\overline{c} \in {^{\ell(\overline{z})}M_1}$, 
$\vp(\bar{x},\bar{y}, \bar{c})$ defines a discrete linear order on $\{ \bar{a} : M_1 \models \vp(\bar{a}, \bar{a}, \bar{c}) \}$ with a first and last element. 

\item The family of all linear orders defined in this way will be denoted $\ord(\Delta, M_1)$. Specifically, each $\ma \in \ord(\Delta, M_1)$ 
is a tuple $(X_\ma, \leq_\ma, \vp_\ma, \overline{c}_\ma, d_\ma)$, where:
\begin{enumerate}
\item $X_\ma$ denotes the underlying set $\{ \bar{a} : M_1 \models \vp_\ma(\bar{a}, \bar{a}, \bar{c}_\ma) \}$
\item $\bar{x} \leq_\ma \bar{y}$ abbreviates the formula $\vp_\ma(\bar{x},\bar{y},\bar{c}_\ma)$ 
\item $d_\ma \in X_\ma$ is a bound for the length of elements in the associated tree; it is often, but not always, $\max X_\ma$. 
If $d_\ma$ is finite, we call $\ma$ trivial. 
\end{enumerate}

\item For each $\ma \in \ord(\ma)$, $(X_\ma, \leq_\ma)$ is pseudofinite, meaning that any bounded, nonempty, $M_1$-definable subset has a $\leq_\ma$-greatest 
and $\leq_\ma$-least element. 

\item For each pair $\ma$ and $\mb$ in $\ord(\Delta, M_1)$, there is $\bc \in \ord(\Delta, M_1)$ such that: 
\begin{enumerate} 
\item there exists an $M_1$-definable bijection $\pr: X_\ma \times X_\mb \rightarrow X_\bc$ such that the coordinate projections are $M_1$-definable. 
\item if $d_\ma$ is not finite in $X_\ma$ and $d_\mb$ is not finite in $X_\mb$, then also $d_\bc$ is not finite in $X_\bc$. 
\end{enumerate} 

\item For some nontrivial $\ma \in \ord(\Delta, M_1)$, there is $\bc \in \ord(\Delta, M_1)$ such that
$X_\bc = \pr(X_\ma \times X_\ma)$ and the ordering $\leq_\bc$ satisfies: 
\[ M_1 \models (\forall x \in X_\ma)(\exists y \in X_\bc)(\forall x_1, x_2 \in X_\ma)(\max \{ x_1, x_2 \}  \leq_\ma x \iff \pr(x_1, x_2) \leq_\bc y) \]

\item To the family of distinguished orders, we associate a family of trees, as follows. 
For each formula $\vp(\bar{x},\bar{y}; \bar{z})$ in $\Delta$ there are formulas $\psi_0, \psi_1, \psi_2$ of the language of $M_1$ such that 
for any $\ma \in \ord(\cs)$ with $\vp_\ma = \vp$:

\begin{enumerate}
\item $\psi_0(\bar{x}, \overline{c}_\ma)$ defines a set, denoted $\mct_\ma$, of functions from $X_\ma$ to $X_\ma$. 
%\item $\psi_1(\bar{x}, \bar{y}, \overline{c}_\ma)$ defines the partial order on $\mct_\ma$ given by initial segment. 
\item $\psi_1(\bar{x}, \bar{y}, \overline{c})$ defines a function $\lgn_\ma : \mct_\ma \rightarrow X_\ma$ satisfying:
	\begin{enumerate}
	\item for all $b \in \mct_\ma$, $\lgn_\ma(b) \leq_\ma d_\ma$. 
	\item for all $b \in \mct_\ma$, $\lgn_\ma(b) = \max(\dom(b))$. %$($Note: for $c \in \mct_\ma$, $\maxdom(c) = \lgn(c)-1$.$)$ 
%	\item for all $b, c \in \mct_\ma$, $b \tlf c \implies \lgn_\ma(b) \leq_\ma \lgn_\ma(c)$. 
\end{enumerate}
\item $\psi_2(\bar{x},\bar{y},\overline{c})$ defines a function from 
$\{ (b,a) : b \in \mct_\ma, a \in X_\ma, a <_\ma \lgn_\ma(b) \}$ 
into $X_\ma$ whose value is called $\xr_\ma(b,c)$, and abbreviated $b(a)$. 
\begin{enumerate}
	\item %$($existence of concatenation below $d_\ma$$)$: 
	if $c \in \mct_\ma$ and $\lgn_\ma(c) < d_\ma$ and $a \in X_\ma$, then $c^\smallfrown \langle a \rangle$ exists, 
	i.e.  there is $c^\prime \in \mct_\ma$ such that $\lgn_\ma(c^\prime) = \lgn_\ma(c)+1$, $c^\prime(\lgn_\ma(c)) = a$, and   
	\[ (\forall a <_\ma \lgn_\ma(c) ) (c(a) = c^\prime(a)) \]
\item $\psi_0(\bar{x}, \overline{c})$ implies that if $b_1 \neq b_2 \in \mct_\ma$, $\lgn_\ma(b_1) = \lgn_\ma(b_2)$ then for some 
$n <_\ma \lgn_\ma(b_1)$, $b_1(n) \neq b_2(n)$.
\end{enumerate}

\item $\psi_3(\bar{x},\bar{y},\overline{c})$ defines the partial order $\tlf_{\ma}$ on $\mct_{\ma}$ given by initial segment, that is, 
such that that $b_1 \tlf_{\ma} b_2$ implies:
\begin{enumerate}
	\item for all $b, c \in \mct_\ma$, $b \tlf c$ implies $\lgn_\ma(b) \leq_\ma \lgn_\ma(c)$. 
	\item $\lgn_\ma(b_1) \leq_\ma \lgn_\ma(b_2)$
	\item $(\forall a <_\ma \lgn_\ma(b_1) ) \left( b_2(a) = b_1(a)\right)$
\end{enumerate}
\end{enumerate}
\end{enumerate}
The family of all $\mct_\ma$ defined this way will be denoted $\tr(\Delta, M_1)$.  We refer to elements of this family as trees. 
\end{defn}

Notice the pairing requirement in condition (6). While we need Cartesian products to exist, it is largely 
unimportant what exactly the order on these products is (as long as it satisfies the other requirements). 
It's sufficient for the results of \cite{MiSh:998} that one such order behave well, like the usual G\"odel pairing function. 
In the course of the present paper, we will consider weakenings of this requirement.

\begin{defn}[Cofinality spectrum problems, \cite{MiSh:998} Definition 2.4] 
\label{cst}  \label{d:csp} 
We call the six-tuple 
\[ \cs = (M, M_1, M^+, M^+_1, T, \Delta)\] a \emph{cofinality spectrum problem} when:
\begin{enumerate}
\item  $M \preceq M_1$.
\item $T \supseteq Th(M)$ is a theory in a possibly larger vocabulary.
\item $\Delta$ is a set of formulas in the language of $M$, i.e., we are interested in studying the orders of $M, M_1$ in the
presence of the additional structure of the expansion. 
\item $M^+$, $M^+_1$ expand $M, M_1$ respectively so that $M^+ \preceq M^+_1 \models T$ 
and $(M^+_1, \Delta)$ has enough set theory for trees. 
\item We may refer to the components of $\cs$ as $M^\cs$, $\Delta^\cs$, etc. for definiteness. When $T = Th(M)$, $M=M^+$,
$M_1 = M^+_1$, or $\Delta$ is the set of all formulas $\vp(x,y,\overline{z})$ in the language of $T$ which satisfy $\ref{d:estt}(2)-(4)$, 
these may be omitted. 
\end{enumerate}
\end{defn}

\begin{rmk}
The identities of $M^+$ and $M_1$ are not essential to many arguments. 
\end{rmk}

\begin{defn}[The cardinals $\xp_\cs$, $\xt_\cs$ and the cut spectrum, \cite{MiSh:998} Definition 2.8] \label{cst:card}
For a cofinality spectrum problem $\cs$ we define the following: %\footnote{``$\ord$'' stands for orders, ``$\tr$'' for trees, ``ct'' for cut, ``ttp'' for treetops.}
\begin{enumerate}
\item $\ord(\cs) = \ord(\Delta^\cs, M^\cs_1)$ % in the sense of $\ref{d:estt}$ %, i.e. 

\br
\item $\cts = $
%\begin{align*}
$\{ (\kappa_1, \kappa_2)   : ~$for some $\ma \in \ord(\cs)$, $(X_\ma, \leq_\ma)$ 
 has a $(\kappa_1, \kappa_2)$-cut$\}$. \\ Note that the $\kappa_\ell$ are infinite. 
%\end{align*}

%In particular, $\ref{mc}$ requires that $\kappa_1, \kappa_2$ are regular. 
\br
\item \label{d:tr-cs} $\tr(\cs) = \{ \mct_\ma : \ma \in \ord(\cs) \} = \tr(\Delta^\cs, M^\cs_{1})$ %in the sense of $\ref{d:estt}$

\br

\item $\mc^{\mathrm{ttp}}(\cs) =$ $\{ \kappa ~:~ \kappa \geq \aleph_0, ~ \ma \in \ord(\cs)$, and there is in the tree $\mct_\ma$ 
 a strictly increasing sequence of cofinality $\kappa$ with no upper bound $\}$

\br
\item \label{card:t}  Let $\xt_\cs$ be $\min \mc^{\mathrm{ttp}}(\cs)$ and let $\xp_\cs$ be
$\min \{ \kappa : (\kappa_1, \kappa_2) \in \cts ~\mbox{and}~ \kappa = \kappa_1 + \kappa_2 \}$.
\end{enumerate}

\br
\emph{A key role will be played by $\mc(\cs, \xt_\cs)$, where this means:}

\br
\begin{enumerate}[resume]
\item For $\lambda$ an infinite cardinal, write
\[ \mc(\cs, \lambda) = \{ (\kappa_1, \kappa_2) ~:~ \kappa_1 + \kappa_2 < \lambda, ~ (\kappa_1, \kappa_2) \in \cts \}. \] 
\end{enumerate}
\end{defn}

By definition, both $\xt_\cs$ and $\xp_\cs$ are regular. The main engine of the paper \cite{MiSh:998} was the following theorem, proved by model-theoretic means. 
Note that it entails $\xt_\cs \leq \xp_\cs$. 

\begin{thm-lit}[\cite{MiSh:998} Theorem 9.1] \label{998-mt}
For any cofinality spectrum problem $\cs$, $\mc(\cs, \xt_\cs) = \emptyset$.
\end{thm-lit}

\begin{disc} \label{on-pairing} \emph{We now resume the discussion following Definition \ref{d:estt}. 
In \cite{MiSh:998}, Cartesian products were used in two ways. 
First, we needed the simple existence of Cartesian products of pairs of elements of $\ord(\cs)$, 
with no restrictions on the ordering of the pairs other than: (i) pseudofiniteness and 
(ii) the property that if $d_\ma$, $d_\mb$ are nonstandard then so is $d_{\ma \times \mb}$ (call such 
a product \emph{nontrivial}). This was needed for the basic arguments connecting behavior across all orders: notably, 
establishing existence of the lower cofinality function.  
To rule out symmetric cuts, we needed only to be able to take the Cartesian product of an order with itself.  
For the main lemma ruling out antisymmetric cuts, we needed (a) the existence of $\ma$ and $\ma^\prime$ such that $\ma$ is coverable 
as a pair by $\ma^\prime$ (see below), and then (b) the existence of the nontrivial product $\ma \times \ma \times \ma \times 
\ma^\prime \times \ma^\prime \times \ma^\prime$. One may take as a definition of ``coverable as a pair'' its key property, 
Corollary 5.7(a) of \cite{MiSh:998} [quoted here with the assumptions of nontriviality made there explicitly added in]:} 
\end{disc}
\begin{itemize}
\item[(**)] there is $a \in X_\ma$, with $a < d_\ma$ and $a$ not a finite successor of $0_\ma$, such that the G\"odel codes for functions from $[0,a]_\ma$ to $[0,a]_\ma \times [0,a]_\ma$
may be definably identified with a definable subset of $X_{{\ma}^\prime}$ whose greatest element is $< d_{\ma^\prime}$. 
\end{itemize}
For property (**), 
it suffices to have $\ma \in \ord(\cs)$ such that the order on $\ma \times \ma$ is given by the G\"odel pairing function and $d_{\ma \times \ma}$ is nonstandard, 
by the arguments of \cite{MiSh:998}. We will give some alternate sufficient conditions in Section \ref{s:tools}. 

%Otherwise: 
%
%\begin{concl}
%To carry out the proof of \cite{MiSh:998} that $\mc(\cs, \xt_\cs) = \emptyset$, it suffices to assume that $\cs$ is a cofinality spectrum problem satisfying \ref{d:estt}
%except that the conditions on closure of $\ord(\cs)$ under Cartesian products are replaced by the   
%For the existence of (**) as just described, it suffices to have [nontrivial] $\ma \in \ord(\cs)$ such that $d_{\ma \times \ma} = \max X_{\ma \times \ma}$. 
%\end{concl}
%
%\begin{proof}
%
%
%\end{proof}
%
%\begin{cor}
%For the main theorems of \cite{MiSh:998} on cofinality spectrum problems, it suffices to assume that:
%%\begin{enumerate}
%%\item 
%for any pair $\ma, \mb$ of elements of $\ord(\cs)$, not necessarily distinct, there is $\mc \in \ord(\cs)$ such that $X_\mc = X_\ma \times X_\mb$ and 
%$d_\mc = \max(X_\mc)$. 
%%\item for some $ $
%%\end{enumerate}
%\end{cor}
%
%\begin{proof}
%
%
%\end{proof}

The following summarizes our conventions. 

\begin{conv}[Key conventions on CSPs] \label{conv:cs}  \emph{ }

\begin{enumerate}[label=\emph{(\alph*)}]
\item Recall that definable means in the sense of $M^+_1 = M^+_1[\cs]$, unless otherwise stated.\footnote{the main exception being 
that the elements of $\ord(\cs)$ are a priori required to be $\Delta$-definable; this is relaxed if $\cs$ is hereditarily closed, 
defined below.}  
\item Recall that for $\mct_\ma \in \tr(\cs)$, $\eta \in \mct_\ma$ implies 
that $\lgn(\eta) \leq d_\ma$, and we have closure under concatenation, i.e. if $\lgn(\eta) < d_\ma$ and $a \in X_\ma$ then $\eta^\smallfrown\langle a \rangle \in \mct_\ma$. 
\item When $\mct$ is a definable subtree of $\mct_\ma$, we will write 
\[ (\mct, \tlf_\ma) \mbox{ or just } (\mct, \tlf) \mbox{ to mean } (\mct, \tlf_\ma\rstr_\mct) \]
\item When $\eta$ is a sequence of $n$-tuples, $t < \lgn(\eta)$, and $k<n$, write $\eta(t,k)$ for the $k$th element of $\eta(t)$. 
%\item In the cases where $\cs$ is \emph{not} weak and so has Cartesian products, write $\mct_{\ma \times \mb}$ for the tree associated to this order. 
\item Generalizing \cite{MiSh:998}, we use ``internal map'' in the present paper for \emph{any} map definable in $M^+_1$, not necessarily an element of some $\mct \in \tr(\cs)$. 
\item We call a definable, discrete order \emph{pseudofinite} when every nonempty, bounded, definable subset has a first and last element. 
\item We call $\ma \in \ord(\cs)$ \emph{nontrivial} when $d_\ma$ is not a finite successor of $0_\ma$, and we call $\cs$ nontrivial when at least one $\ma \in \ord(\cs)$ 
is nontrivial.%\footnote{and thus also that the squares, or products of nontrivial $\ma$, are nontrivial: see \ref{d:estt}$(\ref{x:here})$.}  Note that $d_\ma$ need not be $\max X_\ma$. We assume all CSPs and all $\ma \in \ord(\cs)$ which we work with are nontrivial unless otherwise stated. 
\end{enumerate}
\end{conv}

\subsection{New definitions: weak and hereditary CSPs.} 
We now give some new definitions weakening and extending CSPs looking towards applications.
%, possibly 
%models of set theory, Peano arithmetic 
%and other applications.  

First, as regards applying CSPs it will be useful to keep track of whether or not the assumption on closure under 
Cartesian products is needed. For example, in dealing with a CSP arising from a model $M = M^+ = M_1^+$ whose domain is linearly ordered, 
it may be that we care primarily about cuts occurring in the linear order $\dom(M)$. 
Thus, we introduce `weak' CSPs in Definition \ref{a2x} and in the present paper, use them where possible.  

%For the present paper, we will keep track of when the condition on closure under Cartesian products is needed; we now 
%give this a name and investigate some consequences.   

\begin{defn}[Weak CSP] \label{a2x} \emph{ }
\begin{enumerate}
\item Call $\cs$ a \emph{weak CSP} if it satisfies $\ref{d:csp}$ omitting the requirement that $\ord(\cs)$ be closed under 
Cartesian product, $\ref{d:estt}(5)+(6)$. 

\item For $\cs$ a weak CSP, $\ma \in \ord(\cs)$,  let $Y_\ma = Y_{\ma, \cs} = \{ a \in X_\ma : a <_\ma d_\ma \}$. 
\end{enumerate}
\end{defn}

%Many of the results of the paper are true in the generality of weak CSPs so we use them where possible (recall Discussion \ref{on-pairing}). 

Second, in \cite{MiSh:998} Convention 5.1 we had observed that CSPs have available an internal notion of cardinality.\footnote{which a priori need not satisfy all properties 
of cardinality. There
it was required that the function be an element of the tree $\mct_{\ma \times \mb}$.  Depending on the amount of set theory available in $M^+_1$, 
the present weakening is an a priori loss in definability: when functions are required to belong to a definable tree, we are always able to 
quantify over them, whereas in the definition below the quantification over all such functions may be strictly informal (not first order).} 
In the present paper, we use the more general definition: 

\begin{defn}
Let $\cs$ be a cofinality spectrum problem. %, $\ma, \mb \in \ord(\cs)$. 
Whenever $A, B$ are definable subsets of $M_1$ $($with parameters$)$, we write
\[ \text{``} |A| \leq |B| \text{''} \]
to mean ``there exists an $M^+_1$-definable partial 1-to-1 function $h$ with 
$A \subseteq \dom(h)$ and $\rn(h) \subseteq B$''. Likewise, we write
\[ \text{``} |A| < |B| \text{''} \]
to mean ``$(|A| \leq |B|) \land \neg ( |B| \leq |A| )$,'' i.e. $|A| \leq |B|$ and there does not exist an $M^+_1$-definable 
injection from $B$ into $A$.
\end{defn}

Third, while the definition of CSP carefully lays out the 'canonical' orders and trees, it is natural to add some 
internal closure conditions, stating essentially that objects which are 
internally isomorphic to canonical ones also count. 

\begin{defn} \label{m10}
Let $\cs$ be a CSP or weak CSP. %, Definition $\ref{d:csp}$, or just a weak CSP, Definition $\ref{a2x}$. 

\begin{enumerate}
\item We call $\mb$ a \emph{pseudo-order of $\cs$} and write $\mb \in \psord(\cs)$ when it satisfies all of the requirements on elements of $\ord(\cs)$ from Definition \ref{d:estt} 
in the case where $\Delta$ is the set of all formulas 
in the language. \footnote{In particular, it must have an associated, definable tree, with definable length functions 
and concatenation, and so forth.} $($The point is that possibly the formulas defining it are not from $\Delta_\cs$.$)$ 
\item Given $\mb_2 \in \psord(\cs)$, we say $\mb_1$ is a \emph{suborder} of $\mb_2$ and write $\mb_1 \in \sub(\mb_2)$ when: $\mb_1 \in \psord(\cs)$, $X_{\mb_1} \subseteq X_{\mb_2}$ 
as linear orders, and there is $d \leq_{\mb_2} d_{\mb_2}$ and $\eta \in \mct_{\mb_2}$ with $\max \dom(\eta) = d = d_{\mb_1}$ 
which lists $\langle c \in X_{\mb_1} : c <_{\mb_1} d_{\mb_1} \rangle$ in increasing order.\footnote{So e.g. $\mb_1 \notin \sub(\mb_2)$ when $X_{\mb_1}$ is an 
initial segment of $X_{\mb_2}$ and $d_{\mb_2} < \max(X_{\mb_1})$, since there is not enough room for the list.}  
\end{enumerate}
\end{defn}

\begin{defn}[Hereditary closure]  Let $\cs$ be a CSP or weak CSP. 
\begin{enumerate}
\item Let $\mb_1, \mb_2 \in \psord(\cs)$. We say $(f,g,h)$ is an accurate isomorphism from $\mb_1$ onto $\mb_2$ when it 
respects both the order structure below $d_{\mb_i}$ and the tree structure, that is:
\begin{enumerate}
\item $f$ is a one-to-one mapping, definable in $M^+_1$, from $X_{\mb_1}$ to $X_{\mb_2}$ 
\item $g$ is a one-to-one mapping, definable in $M^+_1$, from $\{ a : a <_{\mb_1} d_{\mb_1} \}$ 
onto $\{ a : a <_{\mb_2} d_{\mb_2} \}$, such that $a<_{\mb_1} c <_{\mb_1} d_{\mb_1} \iff
g(a) <_{\mb_2} g(c) <_{\mb_2} d_{\mb_2}$. 
\item $h$ is a one-to-one function definable in $M^+_1$ from $\mct_{\mb_1}$ onto $\mct_{\mb_2}$ such that 
\[ \eta \tlf_{\mct_{\mb_1}} \nu \iff h(\eta) <_{\mct_{\mb_2}} h(\nu) \] 
and $\eta = \nu^\smallfrown\langle a \rangle$ in $\mb_1$ 
iff $h(\eta) = h(\nu)^\smallfrown\langle f(a) \rangle$ in $\mb_2$.
\end{enumerate}

%\item We say \emph{accurately isomorphic} to indicate existence of a weak triple isomorphism. 

\item Call $\cs$ \emph{hereditary} if $\ord(\cs)$ is closed under $\sub$ and accurate isomorphism. 

\item Say ``$\cs_2$ is the hereditary closure of $\cs_1$'' when $\cs_1, \cs_2$ are the same except that $\ord(\cs_2)$ is the closure of 
$\ord(\cs_1)$ under members of $\sub$ and under accurate isomorphism.  
%In particular, $\cs_2$ does \emph{not} necessarily 
%satisfy ``closure under product,'' $\ref{d:estt}(6)$. 
\end{enumerate}
\end{defn}

\begin{obs}
If $\cs$ is a weak CSP then so is its hereditary closure. 
\end{obs}

\begin{obs} \label{m11a} 
As a consequence of Definition \ref{m10}(3), if $\cs$ is a weak CSP which is hereditarily closed, then for any $\ma \in \ord(\cs)$, 
any $X \subseteq Y_\ma$ and any $d \in X \cap [0_\ma, d_\ma]$, there is $\mb \in \ord(\cs)$ with $X_\mb = X$, $<_\mb = <_\ma \rstr X$, and $d = d_\mb$. 
\end{obs}

%The hypothesis ``weak CSP'' will be used where applicable in the rest of the paper, for greater generality.  
%However, for many theorems we specialize to models of arithmetic. More can be said about weak CSPs; 
%some basic results are collected in the Appendix. 
%

\section{CSPs with exponentiation}

A feature of CSPs is that one can always develop a certain amount of Peano arithmetic internally, as in \cite{MiSh:998} \S 5. 
That observation motivates the main definition of this section, `closure under exponentiation.' In this section we work out that 
CSPs of this kind have the very nice property that $\xp_\cs = \xt_\cs$.  For context, recall that the main theorem of \cite{MiSh:998} that 
$\mc(\cs, \xt_\cs) = \emptyset$ only implies that $\xt_\cs \leq \xp_\cs$. % \emph{is} a $(\xt_\cs, \xt_\cs)$-cut. 
For the other direction, it's natural to try to show that a $\kappa$-indexed path through some $\mct$ with no upper bound translates to a $(\kappa, \kappa)$-cut 
in some $\ma \in \ord(\cs)$. The problem is that the natural translation produces a definable, discrete linear order 
which contains a $(\kappa, \kappa)$-cut but is \emph{not} necessarily a member of $\ord(\cs)$. After explaining the problem, we will propose the solution. 

%Fir
%
%In \cite{MiSh:998} it was shown that in any 
%In this section we introduce a definitions
%
%This section defines and develops the notion of ``closure under exponentation'' for a CSP or weak CSP. Using this we resolve the following. 
%In \cite{MiSh:998}, our proof that $\mc(\cs, \xt_\cs) = \emptyset$ (Theorem \ref{998-mt}, quoted in \S \ref{s:preparation} above) showed that 
%for any cofinality spectrum problem $\cs$, $\xt_\cs \leq \xp_\cs$.  As discussed in \cite{MiSh:998} \S 6, this leaves open the question of whether $\xp_\cs \leq \xt_\cs$. 
%In Theorem \ref{ps-is-ts} below, we will prove the second inequality holds when $\cs$ is closed under exponentiation. 
%This allows us to characterize the first (smallest) cut in $\mc(\cs, \xt_\cs)$ as symmetric, which will be useful in \S \ref{s:peano-arithmetic}.
%

\begin{hyp} \label{b0}
In this section, unless otherwise stated, $\cs$ is a \emph{weak} CSP. 
For transparency, we assume $\cs$ is hereditary, i.e. closed under $\sub$ and full internal isomorphism. 
We retain the notation $Y_\ma$ from $\ref{a2x}$. 
\end{hyp}
%
%The issue in showing that $\xp_\cs \leq \xt_\cs$ arises as follows. 

\begin{defn} \emph{(Flattening the tree, c.f. \cite{MiSh:998} Lemma 6.2)} \label{d:coll}
Let $\ma \in \ord(\cs)$ be given. We define the following linear order.
Fix in advance two distinct elements of $X_\ma$; without loss of generality\footnote{recall that $X_\ma$ is a discrete linear order with first element, so $0_\ma, 1_\ma$ have 
the natural meanings.} 
we use $0_\ma, 1_\ma$, called $0,1$, so $0 <_\ma 1$.
Let $\os_\ma$ be the set $\mct_\ma \times \{ 0, 1 \}$.
Let $<_{\os_\ma}$ be the linear order on $\os_\ma$ defined as follows:
\begin{itemize}
\item If $c = d$, then $(c, i) <_{\os_\ma} (d,j)$ iff $i<_\ma j$
\item If $c \tlf_\ma d$ and $c \neq d$, then $(c, 0) <_{\os_\ma} (d, 0) <_{\os_\ma} (d, 1) <_{\os_\ma} (c, 1)$
\item If $c, d$ are $\tlf$-incomparable, then let $e \in \mct_\ma$, $n_c, n_d \in {\os_\ma}$ be such that
$e = \cis(c,d)$ is the common initial segment and $e^\smallfrown n_c \tlf_\ma c$ and $e^\smallfrown n_d \tlf_\ma d$. Necessarily $n_c \neq n_d$ by definition
of $e$, so for $s, t \in \{ 0, 1 \}$ we define
\[ (c, s) <_{\os_\ma} (d,t) \iff n_c <_\ma n_d \]
\end{itemize}
For each given $\ma$, we refer to this ordered set $({\os_\ma}, <_{\os_\ma})$ just constructed 
as ``the order given by flattening the tree $\mct_\ma$''.
\end{defn}

\begin{conv} \label{s:conv}
For the remainder of this article, given $\cs$ and $\ma \in \ord(\cs)$, let $({\os_\ma}, <_{\os_\ma})$ denote the order constructed in \ref{d:coll}. 
\end{conv}

\begin{fact} \label{cy} 
Let $\cs$ be a weak CSP. 
If $\ma \in \ord(\cs)$, $\mct_\ma$ witnesses treetops and $({\os_\ma}, <_{\os_\ma})$ is the order given by flattening the tree 
$\mct_\ma$ then $({\os_\ma}, <_{\os_\ma})$ has a $(\xt_\cs, \xt_\cs)$-cut.
\end{fact}

\begin{proof}
The proof when $\cs$ is a CSP is \cite{MiSh:998} Lemma 6.2. 
No assumptions are made about closure under Cartesian products in that proof, so the identical 
result holds for weak CSPs as well. 
\end{proof}

We cannot a priori conclude from Fact \ref{cy} that $(\xt_\cs, \xt_\cs) \in \mc(\cs, \xt_\cs)$, because we have not shown that $\os_\ma$  
belongs to $\ord(\cs)$. However, note that $({\os_\ma}, <_{\os_\ma})$ is a definable discrete linear order with a first and last element, since:

\begin{obs} \label{f16a}
For any weak cofinality spectrum problem $\cs$ and $\ma \in \ord(\cs)$, 
there is an internal order-isomorphism between an initial segment of $(\os_\ma, <_{\os_\ma})$ and $Y_\ma$. %, i.e. $(X_\ma \rstr _{< d_\ma}, <_\ma)$. 
\end{obs}

\begin{proof}
Let $\langle c_\alpha : \alpha < d_\ma \rangle$ be the sequence of elements of $\mct_\ma$ corresponding to functions which are constantly $0_\ma$, 
listed in increasing order: this sequence is linearly ordered by $\tlf_\ma$. 
By the definition of $\os_\ma$, we have that $(\{ (c_\alpha, 0): \alpha < d_\ma \}, <_\ma)$ is an initial segment of $(\os_\ma, <_{\os_\ma})$.  
Moreover, it is isomorphic to $(Y_\ma, <_\ma)$ via the internal map $\lgn_\ma$ from Definition \ref{d:estt}(7). 
\end{proof}

%We now give a key definition of the paper. 

%\begin{proof} 
%Clear.
%\end{proof}

%Recalling \ref{s:conv} and \ref{d:trees-ok}, we now give a key definition for the paper.
%Note that ``closed under exponentiation'' means informally that $|X_\ma|^{d_\ma}$ is well defined inside $M^+_{1,\cs}$, and not $2^{|X_\ma|}$.

\begin{defn}[Closed under exponentiation] \label{d:nx}
Let $\cs$ be a weak CSP. We say $\cs$ is closed under $($simple$)$ exponentiation %\emph{simply closed under exponentiation}, or has simple exponentiation, 
when for every nontrivial $\ma \in \ord(\cs)$ there is a nontrivial 
$\mb \in \ord(\cs)$ such that $(X_\mb, <_\mb)$ and $(\os_\ma, <_{\os_\ma})$ are internally isomorphic.  
\end{defn}

\begin{defn} \label{d:nx2} For $\cs$ a weak CSP, it will also be useful to define:
\begin{enumerate}
\item[(a)] $\cs$ is \emph{strongly closed under exponentiation}, or has strong exponentiation, when: for every $\ma \in \ord(\cs)$ there is 
$\mb \in \ord(\cs)$ such that $(\os_\ma, <_{\os_\ma})$ is accurately isomorphic to $( X_\mb \rstr {d_\mb}, <_\mb )$, following 
Definition $\ref{m10}$. \footnote{In Definition \ref{m10}, accurate isomorphism of two linear orders $\mb_0, \mb_1$ involves
an isomorphism of sets plus an order-isomorphism below the bounds $d_{\mb_i}$. For these purposes, consider the $d$ for $S_\ma$ to be $\max S_\ma$.} 

\item[(b)] $\cs$ is \emph{weakly closed under exponentiation}, or has weak exponentiation, when: for every $\ma \in \ord(\cs)$ there is 
$\mb \in \ord(\cs)$ such that $(X_\mb, <_\mb)$ and $(\os_\ma, <_{\os_\ma})$ are internally isomorphic and 
 $(Y_\mb, <_\mb )$, $(Y_\ma, <_\ma )$ are internally isomorphic. %, Definition \ref{d:342}. 

%\item ``$\cs$ has exponentiation'' will mean: satisfying one of the above. 
%\item ``Closed under exponentiation'' means: satisfying one of the above. 
\end{enumerate}
\end{defn}

\begin{disc}
The phrase ``$\cs$ has exponentiation'' clearly covers $\ref{d:nx}$, $\ref{d:nx2}(a)$ and $\ref{d:nx2}(b)$. One of the issues raised by $\ref{d:nx2}$ is 
whether $(\os_\ma, <_{\os_\ma})$ may map onto $X_\mb$ in such a way that $d_\mb$ is not above the range of the map. 
To distinguish $\ref{d:nx}$ from $\ref{d:nx2}$, we will say ``simple exponentiation.''
%If so, one has a priori less power. 
\end{disc}

%In Definition \ref{d:nx}, a key distinction is between having a

\begin{cor} \label{z45} \label{sym6}
Let $\cs$ be a weak cofinality spectrum problem which has exponentiation, $\ma \in \ord(\cs)$. 
Then: 
\begin{enumerate}
\item $(\os_\ma, <_{\os_\ma})$ is a discrete linear order in which every nonempty definable subset has a first and last element.
\item There is $\mb \in \ord(\cs)$ so that $(X_\mb, \leq_\mb)$ has a $(\xt_\cs, \xt_\cs)$-cut. 
\end{enumerate}
\end{cor}

\begin{proof}
(1) This is inherited from the order-isomorphism to an element of $\ord(\cs)$. 

(2) Let $\mct_\ma$ witness treetops. By Fact \ref{cy}, the order $(\os_\ma, <_{\os_\ma})$ has a $(\xt_\cs, \xt_\cs)$-cut. 
Given $\mb \in \ord(\cs)$ such that $(\os_\ma, <_{\os_\ma})$ is order-isomorphic to $X_\mb$, 
clearly $(X_\mb, \leq_\mb)$ has a $(\xt_\cs, \xt_\cs)$-cut.
\end{proof}

%\newpage

We arrive at a fact which will be useful throughout the paper: if $\cs$ has exponentiation then $\xp_\cs = \xt_\cs$ 
and `the first cut is symmetric'. 

\begin{theorem} \label{ps-is-ts} \emph{ }
\begin{enumerate}
\item Let $\cs$ be a weak CSP with exponentiation. Then $\xp_\cs \leq \xt_\cs$.
\item Let $\cs$ be a CSP with exponentiation. Then $\xp_\cs = \xt_\cs$.
\item Let $\cs$ be a CSP with exponentiation. Then
\[ \xt_\cs = \min \{ \kappa :  (\kappa, \kappa) \in \cts \} \] 
and the first cut in $\cts$ is necessarily symmetric, that is, if 
\[  \mu = \min \{ \kappa + \lambda : (\kappa, \lambda) \in \cts \} \]
then $(\mu, \mu) \in \cts$.
\end{enumerate}
\end{theorem}

\begin{proof}
First we prove (1). Corollary \ref{sym6}(2) shows that {if} $\xt_\cs = \kappa$, then $(\kappa, \kappa) \in \cts$, so $\xp_\cs \leq \kappa + \kappa = \kappa$. 
Thus $\xp_\cs \leq \xt_\cs$. 

When in addition $\cs$ is a cofinality spectrum problem  
the analysis of \cite{MiSh:998} applies. By 
Theorem \ref{998-mt}, \S \ref{s:preparation} above 
we have that $\mc(\cs, \xt_\cs) = \emptyset$, thus $\xt_\cs \leq \xp_\cs$, proving (2).  

For (3), the proof of \ref{ps-is-ts} shows that $\xp_\cs = \min \{ \kappa : (\kappa, \kappa) \in \cts \}$. Since $\xp_\cs = \xt_\cs$, 
this is sufficient. 
\end{proof}

%\begin{concl} \label{c5x} 
%\end{concl}

%\begin{proof}
%\end{proof}

However, as we will see in Theorem \ref{t6}, 
the situation for the local versions of these cardinals, Definition \ref{d:local}, is more subtle. 
It would be interesting to explore this further.

\section{On bounded arithmetic} 

In this section, working towards our first main application in \S \ref{s:peano-arithmetic}, we set up CSPs and weak CSPs arising from models of PA or BPA and check when they are 
closed under exponentiation in the sense just described. 
%Working towards the first application of CSPs, to Peano arithmetic in the next section, the present section establishes relevant canonical ways of 
%finding CSPs from such models. 

\begin{defn} 
A formula is called bounded if all of its quantifiers are bounded. 
By BPA we mean bounded Peano arithmetic, that is, the restriction of the Peano axioms containing induction only for bounded formulas. 
\end{defn}

When working with models of PA or BPA, we will use the notation $x^y$ in accordance with: 

\begin{fact} \label{fact:exp} \emph{(Gaifman and Dimitracopoulos \cite{GD}, see \cite{paris-wilkie} \S 1.1)}
Let $I\Delta_0$ denote basic arithmetic with bounded induction. %, e.g. as in \cite{p-w} p. $261$.
There exists a $\Delta_0$ formula $\vp(x,y,z)$, which we denote by $x^y=z$, which can be shown 
in $I\Delta_0$ to have all the usual properties of the graph of exponentiation except for the sentence 
\[ \forall x \forall y \exists z  (x^y = z). \]  
\end{fact}

%The next definition lists some canonical weak CSPs obtainable from such models. 
%We will keep track of a set $D$ of elements for which powers exist. 

Regarding Definition \ref{m23}, we will focus on $\cs_3[N]$ for $N$ a model of PA or BPA. 
%, but we include natural variations. 

\begin{defn}[Some canonical weak CSPs from models]  \label{m23}
Let $N \models BPA$ or $N \models PA$, and $\ell = 2, 3$. 
We define $\cs = \cs_\ell[N, D]$. Omitting $D$ means $D = N$.  
\begin{enumerate}
\item[(0)] 
$D$ is an initial segment of $N$ closed under addition. $D$ is \emph{nontrivial} if it has a nonstandard member. 
We require that: 
\begin{enumerate}
%\item If $\ell = 1$, then $D = N$. 
\item If $\ell = 2$, then $d \in D$ implies that $N \models (\forall x)(x^d \mbox{ exists })$. 
\item The case $\ell = 3$ is covered in (2)(c). 
%\item The case $\ell = 4$ is covered in (2)(d). 
\end{enumerate}
\item $\Delta = \Delta_\cs$ is either: 
\begin{enumerate}
\item in the case of BPA, the set of all bounded formulas, or 
\item in the case of PA, the set of all formulas %$($all formulas in the full case$)$ 
$\vp(x,y,z)$ with $\ell(x) = \ell(y) = 1$,
\end{enumerate} 
which satisfy $\ref{d:estt}(1), (2), (4), (7)$. 
\item $\ma \in \ord(\cs)$ when the data of  
$(X_\ma, <_\ma)$ is internally isomorphic\footnote{recall that as the set is listed with its order, this means: internally order-isomorphic.} 
to some $(X_\ma, \leq_\ma, d_\ma, \mct_\ma)$ which satisfies: 
\begin{enumerate}
\item the set of elements of $X_\ma$ is a $\Delta$-definable bounded subset of $N$. 
\item $<_\ma$ is a $\Delta$-definable linear order on $X_\ma$ 
\item $d_\ma \in X_\ma$, and $( \{ d : d < d_\ma \}, <_\ma )$ is an initial segment of $N$ with the usual order. 
\begin{itemize}
%\item If $\ell = 1$, $N = D$ so $d_\ma \in D$ naturally and there are no further requirements.  
\item If $\ell = 2$, we require that $d_\ma \in D$. %\footnote{i.e. we add all such possible}
\item If $\ell = 3$, we require that $d_\ma \in D$ and $|X_\ma|^{d_\ma}$ exists. % \footnote{any nontriviality?}
%\item If $\ell = 4$, we require that $d_\ma \in D$ and $|X_\ma|^{d_\ma}$ exists, and that for every  
%$n \in N$ there is $\ma \in \ord(\cs)$ with $[0,n] \subseteq Y_\ma$. 
\end{itemize}
\item $\mct_\ma$ is defined and satisfies the conditions from $\ref{d:estt}$, and its defining formulas 
depend uniformly on the the formula defining $X_\ma$.
\end{enumerate}

%\br
%\item Unless otherwise stated, if $N \models PA$, $\Delta$ is from $(1)(b)$ and $D = N$. If $N \models BPA$, $\Delta$ is from $(1)(a)$.  
%
%%$\cs[N,D]$ is understood to be full $($and $D = N$$)$, and 
%%if just $N \models BPA$ then $\cs[N,D]$ is understood to be derived, or 
%%$D = \{ b \in N : N \models (\forall x)(x^b \mbox{ exists } ) \}$. 
%
%\item Writing $\cs[N]$ means $\cs[N,D]$ and $D = N$. 
%
\item Write $\cs^+[N, D]$ to indicate that we close the set of orders of $\cs[N,D]$ under taking Cartesian products, where $d_{\ma \times \mb}$ is 
understood to be nontrivial if $d_\ma, d_\mb$ are, 
and that the order on at least one of the pairs is given by the G\"odel pairing function. 
\end{enumerate}
\end{defn}

%\begin{rmk}
%In the present paper we concentrate on $\cs_3[N]$. 
%\end{rmk}

For use in later papers, we record here: 
\begin{defn} In the context of $\ref{m23}$, 
suppose $f$ is a nondecreasing function with $\dom(f) = N$ 
and $\rn(f) \subseteq \{ I : I $ an initial segment of $N$ $\}$, such that: $d \in f(a) $ implies that $a^d$ exists in $N$ 
and whenever $a$ is nonstandard, $f(a)$ contains some nonstandard $d$. 
Define $\cs[N, f]$ by requiring that for each $\ma$,
\begin{enumerate}
\item[(a)] $d_\ma \leq f(\max(X_\ma) + 1)$, and 
\item[(b)] $f(\max(X_\ma)+1) = I$ implies $d_\ma \in I$. 
\end{enumerate}
\end{defn}

Let us also give a name to a recurrent assumption. 
\begin{defn}[Reasonable models]
We say $N \models BPA$ is reasonable when $a \in N$ implies that $\{ a^n : n $ a standard integer $ \}$ is a bounded subset of $N$. 
\end{defn}
%\item $D$ is \emph{weak} in $N$ if $D$ is an initial segment of $N$ and $D$ is \emph{strong} in $N$ if it is closed under products. 

%\item We say $D$ is a reasonable initial segment for $N \models BPA$ when $d \in N \land x \in N \implies (x^d \mbox{ exists })$. 

%\item $D$ is nontrivial when it has a nonstandard member. 
%\end{enumerate} 
%\end{defn}

\begin{claim} \label{2.15A} Let $N \models BPA$.
\begin{enumerate}
\item If $N$ is $\aleph_1$-saturated or just recursively saturated, then $N$ is reasonable. 

\item $N$ is reasonable iff $N$ is a candidate in the sense of \ref{d:r19} below. 

\item If $D_1$ is a reasonable initial segment of $N$ then 
$D_2 = \cl(D_1, N) := \{ a \in N : $ for some $d \in D_1$ and standard $n$, $N \models$ ``$a < d^n$'' $\}$ is also reasonable. 

\item If $D$ is an initial segment of $N$ then $\cl(D, N)$ is closed under products. % in $N$. 
\end{enumerate}
\end{claim}

\begin{proof}
(1),(2),(4) are immediate. 

(3): Let $a _2 \in D_2$ and $n \in N$ be standard. Let $d_1 \in D_2$ and $m$ standard be such that $N \models d_2 < (d_1)^m$. 
We define $x_\ell$ by induction on $\ell \leq n\cdot m$ as follows:  $x_0 = x$, $x_{\ell} = (x_\ell)^{d_1}$. 
Note that $(x_\ell)^{d_\ell}$ exists by the assumption on $D_1$. So $x_{nm}$ is well defined. Since 
\[ x^{d_2} \leq x^{{d_1}^m} = (x^{d_1})^{d_1}\cdots (x^{d_1})^{d_1} \leq x_{nm} \]
where the dots $\cdots$ indicate the multiplication has $m$ terms, $x^{d_2}$ exists. 
\end{proof}

\begin{obs} \label{cond-str}
If $N \models BPA$ and $\cs = \cs_\ell[N, D]$ for $\ell \in \{ 2, 3 \}$ then 
$\cs$ is a hereditary weak CSP.  
\end{obs}

%\begin{obs} The hypotheses of Conclusion \ref{c2b} are satisfied 
%when $M_{1,\cs} = N \models BPA$ and $\cs = \cs_\ell[N,D]$ for $\ell \in \{ 2, 3, 4 \}$ and $D$ is cofinal in $N$. 
%\end{obs}

\begin{claim} \label{c:exp-a}
Let $N \models BPA$ and suppose that for $n$, $d \in N$ we have that $n^d$ exists. 
Then the tree of functions from $d$ to $n$ is definable in $N$, by bounded formulas. 
Specifically, the operations from $\ref{d:estt}(7)$ for the tree of sequences of length $d$ into  
$([0, n], <^N)$ are all definable by bounded formulas. 
\end{claim} 

\begin{proof}
For compatibility with \ref{d:estt}, we denote $[0, n]$ by $X_\ma$ and $d$ by $d_\ma$. 
Let $\mct$ denote the sequences of elements of $X_\ma$ of length $\leq d_\ma$. 
First we show that for each $\eta \in \mct$ there is $c \in N$ which is a code for $\eta$ 
(and that this is uniformly and boundedly definable). 
This applies the long known fact that G\"odel coding may be carried out in BPA; we sketch 
a proof for completeness, following the method of Wilkie-Paris \cite{paris-wilkie}.

We use $B$-adic coding (for $B = 2$), representing each element of $N$ as a word in the finite alphabet $\{ 0, 1 \}$.  
(This ignores the empty word, which could be accommodated by using a different finite alphabet or $B = 3$.)  
Since for each $\mct_\ma$ or $\os_\ma$ we will be coding sequences of elements of uniformly bounded length, 
we don't need a separate symbol to indicate a transition between codes for distinct elements. 
Let $\lgb(x)$ denote $\max \{ \ell : 2^\ell \leq x \}$, the dyadic length. 

By \ref{m23}(2), we may assume every element of $\ord(\cs)$ is isomorphic to a canonical one, i.e. to 
$\ma \in \ord(\cs)$ where: 
\begin{itemize}
\item $X_\ma = [0, \max(X_\ma) ]$
\item $\leq_\ma$ agrees with the order of $N$ restricted to $X_\ma$
\item $d_\ma$ is such that $(\max X_\ma)^{d_\ma}$ exists, 
\\ which is equivalent to: 
$(\exists n)(|X_\ma| \leq 2^n \leq 2|X_\ma|) \land$ `` $2^{n\cdot d_\ma}$ exists''.
\item $(\exists n, m)(d_\ma \leq 2^m \leq 2d_\ma \land |X_\ma| \leq 2^n \leq 2|X_\ma|$ and $2^{n\cdot 2^m}$ exists$)$
\item so w.l.o.g. $d_\ma = 2^m$, $|X_\ma| = 2^n$.  
\end{itemize}

%Assume first that $a = 2^n$, $d = 2^m$. 
For an element $a$ of $X_\ma$, let $\rep(a)$ denote the $B$-adic representation of $a$ of length exactly $n+1$, 
padded with zeros if necessary (this is possible by the choice of $n$). 
In the expression below let ``$2^m + \rep(a)$'' mean in base 2, so this will 
effectively move $\rep(a)$ (which is a sequence of length $\leq n$) over $m$ spaces.  

Let $\vp_{code}(x,i, b)$ mean: 
%\footnote{Alternative without the $(m+1)$: if $x < 2^{n \cdot 2^m}$ let $x = \sum_{i < n} b_i 2^{m_i}$. Let $\lg(x) = \max \{ i : b_i \neq 0 \}$, 
%and require that $b_{\lg(x)} = 1$ and $x$ represents $\langle b_i : i < \lg(x) \rangle$.} 

\begin{itemize}
\item $x < 2^{n \cdot {2^{m+1}}}$ 
\item $i < d_\ma$
\item $b < \max X_\ma$ 
\item $(\exists x_1 x_2)({x_1} ^\smallfrown x_2 \tlf x \land \lgb(x_1) = (n+1)i \land \lgb(x_2) = n+1 \land x_2 = 2^{n+1} + \rep(b) )$
\end{itemize}

%Recall that any element of $X_\ma$ has a dyadic representation of length $\leq n$ (or exactly $n$ if we pad with final zeros). 
Informally, $\vp_{code}$ asserts that $x$ is the code for a sequence, thought of as consisting of no more than $d_\ma$ consecutive blocks of length $n+1$ 
(leaving one extra space for the coding of $\os_\ma$ in the proof of Claim \ref{c:exp}), 
the $(i+1)$st of which is $\rep(b)$.  

As written several values of $x$ may code the same sequence; we may avoid this by restricting to 
$x$ such that no $y < x$ codes the same sequence. 

Then for any given $\ma \in \ord(\cs)$, for our fixed values of $n, m$,  we may naturally represent $\mct_\ma$ by 
\begin{align*}
\mct_\ma & =  \{ x : x < 2^{n \cdot 2^{m+1}}, (\exists i < n)\left( \lgb(x) = (m+1)\cdot i \right) \\
 & \mbox{ and }(\forall j \leq i) (\exists x_0, x_1 < x)({x_0} ^\smallfrown x_1 \tlf x \mbox{ and } 2^m \leq x_0 < 2^m + 2^n) \} 
\end{align*}
Since we have fixed the values of $X_\ma$ (thus, of $\max (X_\ma)$) and $d_\ma$, we can easily 
build on $\psi_{code}$ to find bounded formulas defining: the partial order on elements of 
$\mct$ by initial segment, length, concatenation, and value of the function at a given element of its domain. 
\end{proof}

\begin{claim} \label{c:exp}
Let $N \models BPA$ be reasonable and suppose that $\cs = \cs_3[N, D]$. 
Then $\cs$ has simple exponentiation. %IF ... weak exponentiation. 
\end{claim}

\begin{proof}
Recall that to say $\cs$ has simple exponentiation means that for every $\ma \in \ord(\cs)$ there is 
$\mb \in \ord(\cs)$ such that $(X_\mb, <_\mb)$ and $(\os_\ma, <_{\os_\ma})$ are internally isomorphic. 
Weak exponentiation adds that also $(Y_\mb, <_\mb )$, $(Y_\ma, <_\ma )$ are internally isomorphic. 

Let $\ma \in \ord(\cs)$ be nontrivial. By construction of $\cs_3$, we may identify $\ma$ with its 
``canonical'' isomorphic image, so assume $X_\ma$ is an initial segment of $N$ with the usual order. 
We will show that the hypothesis that $|X_\ma|^{d_\ma}$ exists, i.e. that $(\max(X_a))^{d_\ma}$ exists, 
already ensures weak exponentiation. 

As $\cs = \cs_3(N)$, there is $\mb \in \ord(\cs)$ so that $(X_\mb, \leq_\mb)$ is internally isomorphic to 
$([0, n_*], <^N)$ and this isomorphism takes $d_\mb$ to $d_*$ and $\mct_\mb$ to $\mct_*$. 
Then $\mct_*$ is boundedly definable by Claim \ref{c:exp-a}. Continuing in the notation of that proof, 
to code the tree consisting of $\os_\ma = \mct_\ma \times \{ 0, 1 \}$, define $f$ as follows:  
\begin{quotation}
if $x = \left( \sum_{i < \lg(x)} (b_i + 2^m)\cdot 2^{(m+1)i}, \ii \right)$, $\ii \in \{ 0, 1 \}$, 
\\ then $f(x) = \sum_{i < \lg(x)} (b_i + 2^m) \cdot 2^{(m+1)i} +  \ii$. 
\end{quotation}
recalling that we have left one unit of space by arranging our coding into blocks of size $n+1$. 
Let $\psi^+_{code} = \rn(f)$, which is definable by a bounded formula. 
If $x_1, x_2 \in \os_\ma$ then $x_1 <_{\os_\ma} x_2$ iff $f(x_1) < f(x_2)$. 
This proves that $(\os_\ma, <_{\os_\ma})$ is order-isomorphic to an initial segment of $N$. 

Let $n_* = \max \rn(f)$. 
Since $N$ is reasonable, there is some nonstandard $d_* \in N$ such that $(n_*)^{d_*}$ exists. 
Applying Claim \ref{c:exp-a} once more, we have a tree $\mct_*$ of functions from $d_*$ to $[0, n_*]$ which 
is definable by bounded formulas. By the definition of $\cs_3(N)$, there is $\mb \in \ord(\cs)$ 
with $X_\mb$ internally order-isomorphic to $[0, n_*]$ under the usual order so that the 
image of $d_\mb$ is $d_*$.  

This proves $\cs$ has simple exponentiation. 

If in addition $(n_*)^{d_\ma}$ exists, let $d_\mb = d_\ma$. Then the same internal order-isomorphism takes 
$Y_\mb$ to $Y_\ma$, and $\cs$ has weak exponentiation. This completes the proof.  
\end{proof}

\begin{claim} \label{c43}
Suppose $N \models BPA$, $N$ is reasonable and $\cs = \cs_3[N,D]$. Then $\cs$ is a CSP. 
\end{claim}

\begin{proof}
As $\cs$ is a hereditary weak c.s.p. there are two potentially missing conditions, $\ref{d:estt}(5)$-$(6)$. 

Let nontrivial $\ma, \mb$ be given. By definition of $\cs_3$, we may assume that $X_\ma$, $X_\mb$ are internally order-
isomorphic to initial segments $[0, n_\ma]$, $[0, n_\mb]$ of $N$ respectively (so in what follows we identify 
$X_\ma, X_\mb$ with these images). Without loss of generality, $n_\ma \leq n_\mb$.  
Let $\pr: N \times N \rightarrow N$ denote the pairing function 
$(x,y) \mapsto (x + y + 1)^2 + x$. 
Consider the set $X_{\ma \times \mb} = \{ \pr(x,y) : x \in X_\ma, y \in X_\mb \} \subseteq [0, (n_\ma + n_\mb + 1)^2]$.  
Then $\pr$ is an isomorphism from $X_\ma \times X_\mb$ onto an initial segment of $N$ which we call $X_\bc$. 
Let $\leq_\bc = \leq^N$ be the usual order.  
Let the order $\leq_{\ma \times \mb}$ be such that $\pr$ is an order-isomorphism from $(X_{\ma \times \mb}, \leq_{\ma \times \mb})$ 
onto $(X_\bc, \leq_\bc)$.  
Note that\footnote{An independent proof that \ref{d:estt}(6) will be satisfied on some pair, assuming only that the c.s.p. is weak 
and that cardinality grows [which will be true in any such $N$] is given in the next section.}
this pairing function satisfies \ref{d:estt}(6). % for any choice of $d_{\ma \times \mb}$ whose image in $X_\bc$ is nonstandard. 

Let $n_* = \max (X_\bc)$. As we assumed $N$ is reasonable, there is some nonstandard $d_*$ such that $(n_*)^{d_*}$ exists. 
Let $d_\bc = d_*$. Now existence of the tree $\mct_\bc$ is by Claim \ref{c:exp-a}.  
Recalling the closure under isomorphism from \ref{m23}(2), the product $\ma \times \mb$ is indeed a nontrivial element of $\cs$.  
Thus \ref{d:estt}(5) holds, which completes the proof. 
\end{proof}

\begin{concl} \label{concl15}
Assume $N \models BPA$, $N$ is reasonable and $\cs = \cs_3(N, D)$. 
Then $\cs$ is a cofinality spectrum problem with exponentiation, and so Theorem $\ref{ps-is-ts}$ applies: $\xp_\cs = \xt_\cs$. 
\end{concl}

%\newpage

\br 
\br

\section{Saturated models of Peano arithmetic} \label{s:peano-arithmetic}

We now apply the above analysis to cuts in models of Peano arithmetic. 

\begin{thm-lit} \emph{(Pabion, 1982 \cite{pabion})} \label{t:pabion}
Let $\lambda > \aleph_0$ and $M$ be a model of PA. Then $M$ is $\lambda$-saturated iff $(M, <)$ is $\lambda$-saturated. 
\end{thm-lit}

\begin{thm-lit} \emph{(Shelah, 1978 \cite{Sh:a} Cor. 2.7 pps. 337-341)} \label{s-order}
\begin{enumerate}
\item If $\de$ is a regular ultrafilter on $\lambda$ such that $M^\lambda/\de$ is $\lambda^+$-saturated 
for $M$ some model of linear order, then $N^\lambda/\de$ is 1-atomically-$\lambda^+$-saturated for any $N$ in a countable language, i.e. $\lambda^+$-saturated 
for types consisting of atomic formulas e.g. $x<a$, $b<x$. 
\item Hence if $T$ has the strict order property it is maximal in Keisler's order. 
\end{enumerate}
\end{thm-lit}

Pabion's theorem may be derived from the proof of Shelah's theorem just quoted. 
It was known that Keisler's order had a maximum class, so to show maximality of linear order, it sufficed to show that realizing types in the language 
of linear order (in some ultrapower) ensured realization of types in some given maximal theory, such as PA. 
The proof of Theorem \ref{s-order} shows that in some language expanding linear order, 
the types of the given theory may be coded in such a way that omission of a type corresponds to a cut in a linear order; 
the point as regards Pabion's theorem is that this coding argument does not rely on ultrapowers except insofar as 
ultrapowers commute with reducts, and so may be carried out in any model of Peano arithmetic.

\begin{obs} \label{obs15}
For $M$ a model of linear order $(~<~)$ and uncountable $\lambda$, $M$ is $1$-atomically $\lambda$-saturated 
iff $\theta_1 + \theta_2 \geq \lambda$ whenever $M$ has a cut of cofinality $(\theta_1, \theta_2)$ and $\theta_1 + \theta_2 \geq \aleph_0$ 
$($equivalently, $>2$, since we may have a $(1, 1)$-cut$)$. Usually, we can omit the ``$1$-atomically,'' as e.g. in the order reduct of a model of PA or in $Th(\mathbb{Q}, <)$. 
\end{obs}

In Observation \ref{obs15}, note that if we use $x \leq y$, then $\min \{ \theta_1, \theta_2 \} \geq \aleph_0$ but then an $(\theta_1, 1)$-cut does not give incompactness. 

The present results are an improvement in two respects. First, we can restrict to the case $\theta_1 = \theta_2$, i.e. symmetric cuts. 
Second, our results are for bounded Peano arithmetic, not just PA. 

\begin{defn} \label{d:r19}
Call $N$ a \emph{candidate model} when $N$ is a model of bounded Peano arithmetic $($with no last element$)$ which is reasonable, i.e. 
such that for any $a \in N$ there is a nonstandard $d$ such that $a^d$ exists. 
We call $\cs$ a \emph{candidate c.s.p.} when $\cs = \cs_3[N]$ for a reasonable $N$. 
\end{defn}

To connect to saturation, we bring in a definition from \cite{MiSh:998}. 

\begin{defn} \label{x15} \emph{(\cite{MiSh:998} Definition 3.41)}
Let $\cs$ be a cofinality spectrum problem and $\lambda$ a regular cardinal.
Let $p = p(x_0, \dots x_{n-1})$ be a consistent partial type with parameters in $M^+_1$.
We say that $p$  is a \emph{$\ord$-type over $M^+_1$} if: $p$ is a consistent partial type in $M^+_1$ and
%$p \subseteq \{ \vp_\ell(x,\overline{a}) : \ell < \ell_* < \omega,
%\overline{a} \subseteq M^+_1 \}$, i.e. $p$ is local, and 
for some $\ma_0, \dots \ma_{n-1} \in \ord(\cs)$, we have that
\[ p \vdash \bigwedge_{i < n}  \text{``}x_i \in X_{\ma_i}\text{''} \]
and $p$ is finitely satisfiable in $X_{\ma_0} \times \cdots \times X_{\ma_{n-1}}$.
We say simply that $M^+_1$ is \emph{$\lambda$-$\ord$-saturated} if every $\ord$-type over $M^+_1$
over a set of size $<\lambda$ is realized in $M^+_1$. Finally, we say that $\cs$ is {$\lambda$-$\ord$-saturated} if $M^+_1$ is.
\end{defn}

\begin{claim} \label{x16} \emph{(\cite{MiSh:998} Claim 3.43)}
Let $\cs$ be a cofinality spectrum problem. If $\kappa < \min \{ \xp_\cs, \xt_\cs \}$ then $\cs$ is 
$\kappa^+$-$\ord$-saturated.
\end{claim}

\begin{rmk}
Since by our definition any $\ma \in \ord(\cs)$ has a maximum element, $\ord$-saturation does not a priori guarantee that the cofinality of the model is large. 
\end{rmk}

\begin{obs} \label{c5}
For a model $N$ of $PA$, the following are equivalent: 
\begin{enumerate}
\item $N$ is $\lambda$-saturated. 
\item $\cf(N) \geq \lambda$ and $N$ is boundedly $\lambda$-saturated, that is, $N \rstr \leq a$ is $\lambda$-saturated for every $a \in N$.
\end{enumerate}
\end{obs}

\begin{proof}
It suffices to prove (2) implies (1). Given a type $p(x)$ of cardinality $<\lambda$, write $p(x) = \{ \vp_i(x,\bar{a}_i) : i < \lambda \}$. 
Since $\cf(N) \geq \lambda$ there is some $a_* \in N$ such that $p(x) \cup \{ x < a_* \}$ is finitely satisfiable.  
%Let $p^\prime(x) = \{ \vp_i(x,\bar{a}_i) \land x<a_* : i < \lambda \}$. Then $p^\prime$ is consistent. 
Let $b_* = 2^{a_*}$. Now for each $i$ there are $c_i < b_*$ and $\vp^\prime_i(x,y, b_*)$ which is $b_*$-bounded [meaning 
that all quantifiers are of the form $(\exists z < b_*)$ or $(\forall z < b_*)$] and 
$N \models (\forall x)(\vp_i(x,\bar{a}_i) \equiv \vp^\prime_i(x,c_i, b_*))$. Let 
$p^{\prime\prime}(x) = \{ \vp^\prime_i(x,c_i, b_*) \land x < a_* : i < \lambda \}$. This is a finitely satisfiable type in $N \rstr_{\leq b_*}$. 
Its realization implies realization of $p$, and it is realized by hypothesis (2). 
\end{proof}

\begin{theorem} \label{t:pa}
Let $N$ be a model of Peano arithmetic and $\lambda$ an uncountable cardinal. 
If the reduct of $N$ to the language of order has cofinality $\geq \lambda$ and no $(\kappa, \kappa)$-cuts 
for $\kappa < \lambda$, then $N$ is $\lambda$-saturated. 
\end{theorem}

\begin{proof} 
We may assume $N$ is a nonstandard model. Hence it follows that $N$ is reasonable in the sense of \ref{m23}, i.e. 
$a \in N$ implies that $\{ a^n : n $ a standard integer $ \}$ is a bounded subset of $N$. (If not, 
$\cf(N) < \lambda$.) 

Assume $\cf(N) \geq \lambda$.  
Let $\cs = \cs_3(N)$ be from \ref{m23}. By Conclusion \ref{concl15}, $\cs$ is a csp with exponentiation.  
Thus Theorem \ref{ps-is-ts} applies and $\xp_\cs = \xt_\cs$.  
By Claim \ref{x16} and Observation \ref{c5}, $\cs$ is $\min \{ \xp_\cs, \xt_\cs, \cf(N, <) \}$-saturated. 
If $\xt_\cs \geq \lambda$, 
we finish, so assume that $\xt_\cs < \lambda$.  
By Theorem \ref{ps-is-ts}, there is $\ma \in \ord(\cs)$ whose $X_\ma$ contains a $(\xt_\cs, \xt_\cs)$-cut. 
Recall from Definition \ref{m23} that since $\cs = \cs_3(N)$, for each $\ma \in \ord(\cs)$, 
$(X_\ma, \leq_\ma)$ is internally order-isomorphic to an initial segment of $N$ with the usual order. 
Then $N$ has a $(\xt_\cs, \xt_\cs)$-cut, which completes the proof. 
\end{proof}

%Note that in fact the proof of Theorem \ref{t:pa} shows:
%
%\begin{claim} \label{c:nt}
%Let $\cs$ be a cofinality spectrum problem [or a strong csp] and suppose $\ma \in \ord(\cs)$ is nontrivial. 
%Then $X_\ma$ has a 
%$(\xt_\cs, \xt_\cs)$-cut. 
%\end{claim}

\begin{theorem} \label{c3} \label{t:bpa}
Let $N$ be a model of BPA which is reasonable, i.e. for every $a \in N$ the set 
$\{ a^n : n \in N $ finite $\}$ is bounded. Then the following are equivalent:
\begin{enumerate}
\item for every $n_* \in N$, the model $N_{<n_*} = N \rstr \{ a: N \models a < n_* \}$ is $\lambda$-saturated.
\item for every $n_* \in N$, the model $N_{\leq n_*}$ considered as a linear order has no $(\kappa, \kappa)$-cuts for $\kappa = \cf(\kappa) < \lambda$. 
\end{enumerate}
\end{theorem}

\begin{proof}
(1) implies (2) is obvious, so assume (2) holds. 
Let $\cs = \cs_3(N)$ be from \ref{m23}. Then by \ref{concl15} $\cs$ is a cofinality spectrum problem with exponentiation, 
and $\xp_\cs = \xt_\cs$.

Let $\kappa = \xp_\cs = \xt_\cs$, so $\kappa$ is regular. By \ref{sym6}, some  $\ma \in \ord(\cs)$ has a $(\kappa, \kappa)$-cut. 
By the definition of $\cs_3(N)$, any $X_\ma$ is internally order-isomorphic to a bounded initial segment of $N$ with the usual order. 
Thus, some bounded initial segment of $N$ has a $(\kappa, \kappa)$-cut. As we've assumed (2), 
it must be that $\lambda \leq \kappa = \xp_\cs = \xt_\cs$. 

Let $n_* \in N$ be given. 
As we assumed $N$ is reasonable, there is a nonstandard $d_*$ such that $N \models$ ``${n_*}^{d_*}$ exists''. 
Recalling \ref{m23}(2), there is a nontrivial $\ma \in \ord(\cs)$ 
such that $(X_\ma, \leq_\ma)$ isomorphic to an initial segment of $N$, containing $[0, n_*]$, 
with the usual order. Thus, to prove (1), it will suffice to show that every $\ma \in \ord(\cs)$ is $\gamma^+$-$\ord$-saturated 
for every $\gamma < \lambda$. By \ref{x16}, if $\gamma < \min \{ \xp_\cs, \xt_\cs \}$ then $\cs$ is 
$\gamma^+$-$\ord$-saturated.
Since $\lambda \leq \xp_\cs = \xt_\cs$, this completes the proof. 
\end{proof}

Note that Theorems \ref{t:pa} and \ref{t:bpa} show that the situation in models of Peano arithmetic is very different from that in real closed fields, 
as shown by the next quoted theorem. 
By ``asymmetric cut'' we mean a cut in which the infinite cofinalities of each side are not equal.   

\begin{thm-lit}[Theorem 1.1 of Shelah \cite{Sh:757}]
Let $K$ be an arbitrary ordered field. Then there is a symmetrically complete\footnote{This means that any decreasing sequences of 
closed bounded intervals, of any ordinal length, has nonempty intersection.} real closed field $K^+$ containing $K$ such that 
any asymmetric cut of $K$ is not filled. So if $K$ is not embeddable into $\mathbb{R}$, then $K^+$ and $K$ necessarily have an asymmetric cut. 
\end{thm-lit}

\section{On the local cardinals $\xp_{\cs, \ma}$ and $\xt_{\cs, \ma}$} \label{s:local}

Returning to CSPs generally, in this section we prove Theorem \ref{t6}, a complementary result to Theorem \ref{ps-is-ts}. The theorem shows that the local cardinals 
$\xp_{\cs, \ma}$ and $\xt_{\cs, \ma}$ from Definition \ref{d:local} need not normally agree, even in CSPs arising from models of Peano arithmetic, if the 
underlying $M^+_1$ is not uniformly saturated. 

\begin{defn} \label{d:local}
Let $\cs$ be a CSP or weak CSP, and $\ma \in \ord(\cs)$.  
\begin{enumerate} 
\item Let $\xp_{\cs, \ma}$ be $\min \{ \kappa :$ there are regular $\kappa_1, \kappa_2$ such that 
$\kappa_1+\kappa_2 = \kappa$ and $X_\ma$ has a $(\kappa_1, \kappa_2)$-cut $\}$. 
\item Let $\xt_{\cs, \ma}$ be 
$\min \{ \kappa ~:~ \kappa \geq \aleph_0$ and there is in the tree $\mct_\ma$ 
 a strictly increasing sequence of cofinality $\kappa$ with no upper bound $\}$. 
\end{enumerate}
\end{defn}

We need a preliminary lemma. 

\begin{lemma} \label{c:tr-cut}
Let $\cs$ be a cofinality spectrum problem, $\ma \in \ord(\cs)$. 
Suppose $(\os_\cs, <_{\os_\cs})$ has a $(\kappa, \kappa)$-cut, for $\kappa \leq \xt_\cs$.  %$($thus, $= \xt_\cs$$)$.  
Then either $X_\ma$ has an $(\kappa, \kappa)$-cut or else $\mct_\ma$ has a branch of cofinality $\kappa$ with no upper bound.
\end{lemma}

\begin{proof} 
Let $(\overline{a}, \overline{b}) = (\langle (a_i, t_i) : i < \kappa \rangle, \langle (b_i, s_i) : i < \kappa \rangle)$ witness the cut in $\os_\ma$, 
with $a_i, b_i \in \mct_\ma$ and $t_i, s_i \in \{ 0_\ma, 1_\ma \}$ for each $i < \kappa$. 
By definition of cut, we may assume $\kappa$ is regular. 

\step{Step 1: Simplifying the presentation of intervals.}
By the pigeonhole principle, we may assume the sequences $\langle t_i : i < \kappa \rangle$ 
and $\langle s_i : i < \kappa \rangle$ are constant. For each $i< \kappa$, write $A_i$ for the closed interval in the linear order $S_{\ma}$ whose 
endpoints are given by $(a_i, t_i)$ and $(a_i, |t_i - 1|)$, and likewise for $B_i$. 
By the construction of $S_{\ma}$, any two intervals of this form are either concentric or disjoint. 

Let $\overline{A} = \langle A_i : i < \kappa \rangle$ and $\overline{B} = \langle B_i : i < \kappa \rangle$. 

The task of this step will be to prove that without loss of generality, $\bar{A}$ consists of either pairwise concentric or pairwise disjoint intervals,  
and likewise for $\bar{B}$. Here ``concentric'' means either either concentric decreasing: $j < i < \kappa \implies C_j \supsetneq C_i$ [for $C = A$ or $B$]
or concentric increasing: $j < i < \kappa \implies C_i \supsetneq C_j$. 

If $\kappa = \aleph_0$, then by Ramsey's theorem\footnote{This also holds if $\kappa$ is weakly compact.}, 
we may assume that $\overline{A}$ consists either of concentric intervals  
or disjoint intervals moving right, meaning $j < i < \kappa$ implies $A_j \cap A_i = \emptyset$ and $(\forall x \in A_j)(\forall y \in A_i)(x < y)$. Likewise, 
we may assume that $\overline{B}$ consists either of concentric intervals or of disjoint intervals moving left, meaning
 $j < i < \kappa$ implies $B_j \cap B_i = \emptyset$ and $(\forall x \in B_j)(\forall y \in B_i)(y < x)$.  

If $\kappa > \aleph_0$, let $a \land b$ denote the maximal common initial segment of $a, b \in \mct_\ma$. 
For this argument, we use $c$ to denote either $a$ or $b$. 
For each $i$, 
the sequence $\langle \lgn_\ma(c_i \land c_j) : j \in [i, \kappa) \rangle$ is a sequence of elements of $Y_\ma \subseteq X_\ma$ 
[recall Definition \ref{a2x}] bounded by $\lgn_\ma(c_i)$. 
Thus, for some club $E_i$ of $\kappa$ with $\min E_i > i$, we have that $\bar{\ell}_i := \langle \lgn_\ma(c_i \land c_j) : j \in E_i \rangle$
is either constant or $<_\ma$-decreasing. Let $S \subseteq \kappa$ be a stationary set of $i$ on which we get the same outcome (either always constant 
or always decreasing). 

Let $E = \bigcap \{ \epsilon < \kappa : \epsilon$ a limit ordinal and $\epsilon \in \bigcap_{i < \epsilon} E_i \}$, so $E$ is a club of $\kappa$. 
There are several cases:

\begin{enumerate}
\item First case: for all $i \in S$, $\bar{\ell}_i$ is constant. 
\begin{enumerate}
\item If $i \in S$, $j \in E_i$ implies $c_i \land c_j = c_i$, then  
$\langle c_i : i \in S \cap E \rangle$ is a $\tlf_{\mct_\ma}$-increasing sequence so we are in the concentric decreasing case. 
\item 
If $i \in S$, $j \in E_i$ implies $\lgn_\ma(c_i \land c_j) < \lgn(c_i)$ then by Fodor's lemma there is $\gamma \in \kappa$ and a stationary subset $X$ of $\kappa$ such that 
$i \in X$ and $j \in E_i$ implies $\lgn_\ma(c_i \land c_j) = \gamma$, so we are in the pairwise disjoint case.  
\end{enumerate}

\item Second case: for all $i \in S$, $\bar{\ell}_i$ is $<_\ma$-decreasing. (Remember that the branches of $\mct_\ma$ are internally pseudofinite, but not necessarily well-ordered from an external point of view.) Let $X = S \cap E$. Then $\langle c_i : i \in X \rangle$ is a $\tlf_{\mct_\ma}$-decreasing sequence, so we are in the concentric increasing case. 
\end{enumerate}

%CHECK that it is really legitimate to do a and b together. I think so.

\step{Step 2: The concentric cases.} Suppose both $\bar{A}$ and $\bar{B}$ are concentric. Then $\langle a_i : i < \kappa \rangle$ 
and $\langle b_i : i < \kappa \rangle$ are both $\tlf_\ma$-linearly ordered sequences in $\mct_\ma$. There are four cases depending on whether 
each of these sequences is $\tlf_\ma$-increasing or decreasing. 

(a) Suppose both $\langle a_i : i < \kappa \rangle$ and $\langle b_i : i < \kappa \rangle$ are $\tlf_\ma$-increasing. 
If they lie along eventually different branches, the original sequence $(\overline{a}, \overline{b})$ could describe only a pre-cut 
and not a cut according to the definition of $\os_\ma$, so we get a contradiction. If they lie along the same branch, then it must be that $\mct_\ma$ has a branch of 
cofinality $\kappa$, as desired. 

(b) If $\langle a_i : i < \kappa \rangle$ is $\tlf_\ma$-increasing while $\langle b_i : i < \kappa \rangle$ is $\tlf_\ma$-decreasing, these 
and form a cut $(*)$ in the linearly ordered set
\[ (\{ c \in \mct_\ma : c \tlf b_0 \}, \tlf) \]
Then the projections $(\langle \lgn(a_i) : i < \kappa \rangle, \langle \lgn(b_j) : j < \kappa \rangle$ form a pre-cut $(**)$ in $X_\ma$. 
If this pre-cut $(**)$ were realized, say by $t$, then $b_0 \rstr t$ realizes the cut $(*)$, contradiction. 
This shows that $X_\ma$ has a $(\kappa, \kappa)$-cut. 

(c) If $\langle a_i : i < \kappa \rangle$ and $\langle b_i : i < \kappa \rangle$ are both $\tlf_{\mct_\ma}$-decreasing, the original sequence
$(\overline{a}, \overline{b})$ will not describe a cut, so we ignore this case. 

(d) If $\langle b_i : i < \kappa \rangle$ is $\tlf_{\mct_\ma}$-increasing while $\langle a_i : i < \kappa \rangle$ is $\tlf_{\mct_\ma}$-decreasing, 
the argument is parallel to case (b). 

\br

\step{Step 3: Not both concentric.} Again, there are several possibilities. 

Suppose first that neither $\overline{A}$ nor $\overline{B}$ is concentric, so $\overline{A}$ is a sequence of disjoint intervals moving right and 
$\overline{B}$ is a sequence of disjoint intervals moving left.

Consider the sequence $\langle c_i : i < \kappa \rangle$ where $c_i := \operatorname{lub} \{ a_i, b_i \}$ in the tree $\mct_\ma$. By definition of 
$\overline{a}$ and $\overline{b}$, this sequence will be either eventually constant or a path through the tree $\mct_\ma$. 

If the sequence is a path through the tree, then $\langle (c_i, 0) : i < \kappa \rangle$ is cofinal in $\overline{a}$ and 
$\langle (c_i, 1) : i < \kappa \rangle$ is cofinal in $\overline{b}$. So the path $\langle c_i : i < \kappa \rangle$ cannot have an upper bound, 
as given any such upper bound $d_*$, by definition of $(\os_\ma, <_{\os_\ma})$, we would have that $(d_*, 0)$ and $(d_*, 1)$ both realize 
the original cut, contradiction. 

So in this case, there is a path through $\mct_\ma$ of length $\kappa$ with no upper bound.

If the sequence is eventually constant, then there is $i_* < \kappa$ such that all $\{ a_i, b_i : i_* < i < \kappa \}$ are immediate successors of the same node, say $a_*$, in $\mct_\ma$. 
So $\lgn(a_*) < d_\ma$. 
By definition of the order $(\os_\ma, <_{\os_\ma})$ and the case we are in, this means 
\[ (\langle a_i(\lgn(a_i)-1) : i_* < i < \kappa \rangle, \langle b_i(\lgn(b_i)-1) : i_* < i < \kappa \rangle) \] 
is a pre-cut in $X_\ma$. 
Suppose for a contradiction it were realized by $x$; then $c_* := {a_*}^\smallfrown\langle x \rangle$ would exist since $\lgn(a_*) < d_\ma$. 
Then in $(\os_\ma, <_{\os_\ma})$, $(c_*, 0)$ and $(c_*, 1)$ would both realize the original cut $(\overline{a}, \overline{b})$, contradiction. 

So in this case, $X_\ma$ has a $(\kappa, \kappa)$-cut.

\br
Otherwise, precisely one of $\overline{A}$, $\overline{B}$ is not concentric. The cases are parallel, so assume 
the non-concentric side is $\overline{B}$. Define $d_i$ for $i<\kappa$ by $d_i := \operatorname{lub} \{ b_i, b_0 \}$ in the tree $\mct_\ma$.  
Writing $C_i$ for the interval $((c_i, 0), (c_i, 1))$ and $D_i$ for the interval $((d_i, 0), (d_i, 1))$ in $\os_\ma$, we have that 
$\langle C_i : i < \kappa \rangle$ and both $\langle D_i : i < \kappa \rangle$ are concentric sequences of intervals with 
$D_i \subseteq C_j$ for all $i >> j$. 

This reduces the problem to Step 2. 

\step{Step 4: Finish.} We have shown that in each case either $X_\ma$ has a $(\kappa, \kappa)$-cut or else $\mct_\ma$ has a strictly increasing path of length $\kappa$ 
with no upper bound, so this completes the proof. 
\end{proof}

\br

\newpage

\begin{theorem} \label{t6} 
Let $\kappa$ be a regular uncountable cardinal. 

\br
\begin{enumerate}

\item Suppose we are given $M$ a model of PA which is $\kappa$-saturated, and $a_* \in M$ nonstandard. 
Then we can find a countable set $X \subseteq M$ such that letting $N$ be the Skolem hull of $\{ a \in M : M \models a \leq a_* \} \cup X$, 
we have that the reduct $(N, <)$ to the language of order has an $(\aleph_0, \aleph_0)$-cut.

\br
\item
There is a cofinality spectrum problem $\cs$ with $M^\cs_1 = N$ and $\ma \in \ord(\cs)$ such that 
\[ \xt_{\cs, \ma} < \xp_{\cs, \ma} \]
in fact, $\xt_{\cs, \ma} = \aleph_0$ while $\xp_{\cs, \ma} \geq \kappa$. 
\br
\item If in the conditions above, we have just that $M \models BPA$ and for some $r_* < a_*$ nonstandard  
$M \models$ ``${a_*}^{r_*}$ exists'', this is enough.  
\end{enumerate}
\end{theorem}

%\begin{lemma} \label{t6a}
%\end{lemma}

\begin{proof}
First we prove (1). 
Let $A = M \rstr a_*$. 

Let $\langle F_n : n < \omega \rangle$ list the Skolem functions of $M$, each appearing infinitely often (for transparency). 
Let $k_n$ be the arity of $F_n$, and without loss of generality $k_n \leq n$. 

Let $p$ be the type in the variables $z$, $x_i$ ($i<\omega$), $y_i$ ($i<\omega$) and parameter $a_*$ expressing: 

\begin{enumerate}[label=\emph{(\alph*)}]
\item $z < a_*$ is nonstandard
\item $x_0 = a_*$
\item $y_0 = (a_*)^z$
\item $m < n < \omega$ implies  $a_* = x_0 \leq x_m < x_n < y_n < y_m \leq y_0$
%\item for all $1 \leq \ell < \omega$, $(a_*)^\ell < c_n$
%\item for all $\ell, n < \omega$, $(x_n)^\ell < y_n$  
\item letting $B_n = \{ F_n(e_0, \dots, e_{k_n-1}) : e_\ell \in A \cup \{ x_\ell, y_\ell : \ell < n \} \}$, we have that 
$B_n \cap (x_n, y_n) = \emptyset$. 
\end{enumerate}

Let us check that $p$ is consistent. By the $\kappa$-saturation of $M$, this will suffice to show it is realized. 

Fix $n<\omega$ and consider a finite fragment of $p \rstr z, x_0, \dots, x_n, y_0, \dots, y_n$. 
Let $\ell$ be the maximal exponent appearing in conditions of the form $(e)$. 
We are looking for $r, a_0, \dots, a_n, b_0, \dots, b_n$ such that 
\begin{itemize}
\item $r$ is nonstandard, or simply above some given natural number 
\item $a_* = a_0 < \dots < a_n < b_n < \dots < b_0 = (a_*)^r$ 
\item $m\leq n$ implies $B_m \cap (a_m, b_m) = \emptyset$
\end{itemize}

For each $n$, the set $B_n$ is definable in $M$ (since it only involves one function $F_n$, of arity $k_n$) 
and of power $c_n$, where $c_n \approx (|A| + 2n)^{k_n} < {a_*}^{k_n+1}$.    
Let $B_n(w, v_0, \dots, v_n, v^\prime_0, \dots, v^\prime_n )$ denote the set 
\[ \{ F_n(e_0, \dots, e_{k_n-1}) : e_\ell \leq w \lor e_\ell \in \{ v_0, \dots, v_n, v^\prime_0, \dots, v^\prime_n \} \} \] 
which will likewise have size $\leq w^{k_n}\cdot (2n+2)^{k_n} < w^{k_n+1}$. 
Let $\vp_n(w,v_0, \dots, v_n, v^\prime_0, \dots, v^\prime_n)$ assert that 
\[(\forall t)(\exists s \leq x^{k_n+1} \cdot t)(\exists a \exists b)(a < b \leq s ~\land ~b-a > t~ \land ~(a,b) \cap B_n(x) = \emptyset)\]
Then clearly for all $m \leq n$, $M \models \forall w \forall\bar{v}\forall\bar{v}^\prime \vp_m(w, \bar{v}, \bar{v}^\prime)$, 
recalling that $k_n$ is fixed so the exponential notation abbreviates multiplication. 

Choosing $r$ such that ${(k_0+2)(k_1+2)\cdots(k_{n}+2)} \leq r < a_*$ and $b_* = b_0 = (a_*)^r$ we may then choose $a_1, b_1, \dots, a_n, b_n$
by induction on $\ell \leq n$ such that:

\begin{itemize}
\item $k < \ell \implies a_k < a_\ell < b_\ell < b_k$
\item $(a_0, b_0) = (a_*, (a_*)^r)$
\item $b_\ell - a_\ell > (a_*)^{n \cdot(n+1-\ell)}$
\end{itemize} 

For $\ell+1$, we have $b_\ell - a_\ell > (a_*)^{(n+2)(n+1-\ell)}$ and $B(a_*, a_0, \dots, a_\ell, b_0, \dots, b_\ell)$ is internally a set with 
$\leq (a_*)^{k_\ell+1} \leq (a_*)^{n+1} < (a_*)^{n+2} $ elements, so there is room. 

This completes the verification that $p$ is consistent, therefore (by saturation) realized. For the remainder of the proof, fix 
realizations $r$, $a_i (i<\omega)$, $b_i (i<\omega)$ of the type $p$. 

[With a little more care, using $(a_*)^{d_* \cdot(d_*+1 -\ell)}$ for $\ell < \omega$, 
we could alternately have chosen the entire countable sequence by induction, avoiding the appeal to the type and 
$\aleph_1$-saturation.]
\br

Let $N$ be the Skolem hull of $M \rstr a_* \cup \{ a_n, b_n : n < \omega \}$, so $N \models PA$. 
Let $\cs = \cs^+_1[N]$ be the canonical CSP from Definition \ref{m23}.   
Then:

\begin{enumerate}[label=\emph{(\alph*)}]
\item $N \preceq M$. 
\item $M_{\leq a_*} = N_{\leq a_*}$ so $(N, <^N)$ is saturated below $a_*$ by definition of $M$.
\item $(\langle a_n : n < \omega \rangle, \langle b_n : n < \omega \rangle)$ is an $(\aleph_0, \aleph_0)$-cut, because by 
construction the Skolem functions do not fill it. 
\item Thus, if $\ma$ is such that $X_\ma = M_{\leq a_*} =  N_{\leq a_*}$, and, say, $d_\ma = \max X_\ma$, 
then $\xp_{\cs, \ma} \geq \kappa$ whereas if $\mb$ is such that 
$X_\mb = N_{\leq b_*}$, and, say, $d_\mb = \max X_\mb$, then $\xp_{\cs, \mb} = \aleph_0$. 
\end{enumerate}

This completes the proof of (1). We continue the argument to prove (2). (3) will follow from the proof. 

Let $\ma$ be such that $X_\ma = N_{\leq a_*}$. Recalling the small nonstandard exponent $r$, 
let $d_* = r+1$ and consider the definable subtree $\mct \subseteq \mct_\ma$ consisting of 
sequences of length $< d_*$ of numbers $< a_*$. $\mct$ has cardinality $a_*^r$
[the tree $\mct_\ma$ will be at least as large]. 

Recalling that $N \models PA$, $\cs$ is closed under strong exponentiation, 
so there is an injection from $\os_\ma$ into some $X_\ma$. Composing with the G\"odel pairing function if needed, 
we may assume there is a definable injection of $\os_\ma$ (as a set) into some $X_\mb \subseteq N$. Applying Conclusion \ref{c2b}, we obtain a definable 
order-isomorphism between $(\os_\ma, <_{\os_\ma})$ onto an initial segment of $M^+_1 = N$.  
Since $\mct$ is the set of all sequences from $X_\ma$ to itself of length $\leq d_{\ma} = a_* >> r$, the cardinality of $\mct$ is $\geq {a_*}^r = b_*$, 
so its image must contain the $(\aleph_0, \aleph_0)$-cut. 
Because $X_\mb$ has an $(\aleph_0, \aleph_0)$-cut, necessarily $(\os, <_{\os})$ has such a cut. 

By Lemma \ref{c:tr-cut}, either $X_\ma$ has an $(\aleph_0, \aleph_0)$-cut 
or else $\mct_\ma$ has a branch of cofinality $\aleph_0$ with no upper bound. 
Since $X_\ma$ is $\kappa$-saturated by assumption, we must be in the second case. 
This shows that $\xt_{\cs, \ma} = \aleph_0$. On the other hand, $\xp_{\cs, \ma} \geq \kappa$, by the hypothesis of saturation.  
\end{proof}

%\begin{theorem} \label{t6}
%For any regular uncountable $\kappa$ there is a cofinality spectrum problem $\cs$ and $\ma \in \ord(\cs)$ such that 
%\[ \xt_{\cs, \ma} < \xp_{\cs, \ma} \]
%in fact, $\xt_{\cs, \ma} = \aleph_0$ while $\xp_{\cs, \ma} \geq \kappa$. 
%\end{theorem}
%
%\begin{proof}
%Apply Lemma \ref{t6a} using the given $\kappa$. That is, 
Thus, the local cardinals $\xp_{\cs, \ma}$ and $\xt_{\cs, \ma}$ need not be equal, 
even in well behaved cofinality spectrum problems, if the model $M^+_1$ is not uniformly 
saturated. 
%\end{proof}

\section{Characterizing the $\tlf^*$-maximal class} \label{s:tlf}

In this section we give the first real evidence that $SOP_2$ is a dividing line by proving that, under instances of GCH, $SOP_2$ characterizes maximality in the 
interpretability order $\tlf^*$ which will be defined below. 
%xxxxxxxxx check comment
The proof uses cofinality spectrum problems. This answers an open question and also gives evidence for a recent conjecture, 
as we now explain. The use of GCH comes only from Fact \ref{x1} below. 
Recall:

\begin{defn} \emph{($SOP_2$, c.f. \cite{ShUs:844} 1.5)} \label{d:sop2}
$T$ has $SOP_2$ if there is a formula $\vp(\overline{x}, \overline{y})$ which does, meaning that in $\mathfrak{C}_T$ there are parameters
$\{ \overline{a}_\eta : \eta \in {^{\omega > } 2} \rangle$, $\ell(\overline{a}_\eta) = \ell(y)$, such that: 
\begin{enumerate}
\item For each $\eta \in {^{\omega}2}$, the set $\{ \vp(\overline{x}, \overline{a}_{\eta \rstr \ell}) : \ell < \omega \}$ is consistent.
\item For any two incomparable $\eta, \nu \in {^{\omega > }2}$, the set $\{ \vp(\overline{x}, \overline{a}_{\eta}), \vp(\overline{x}, \overline{a}_{\nu}) \}$ is inconsistent. 
\end{enumerate}
\end{defn}

Shelah in \cite{Sh:500} had defined an order on theories,  a natural weakening of Keisler's order:  
$T_1 \tlf_* T_2$, which holds, roughly speaking, if there is a third theory $T_*$ which interprets both 
$T_1$ and $T_2$ and whose models $M_*$ have the property that if the reduct to $\tau(T_2)$ is saturated, so is the reduct to $\tau(T_1)$. 
[See Definition \ref{tstar-defn} below.] This was studied and developed more extensively 
by D\v{z}amonja-Shelah \cite{DzSh:692} and Shelah-Usvyatsov \cite{ShUs:844}. 
In the Shelah-Usvyatsov paper, building on work of D\v{z}amonja and Shelah, it was shown that: 

\begin{fact} \emph{(Shelah and Usvyatsov \cite{ShUs:844} 3.15(2), under GCH\footnote{This hypothesis is missing from the statement in \cite{ShUs:844}, but that proof 
quotes \cite{DzSh:692} 3.2, which assumes relevant instances of GCH.})}  \label{x1} 
If $T$ is $NSOP_2$ then $T$ is necessarily non-maximal in $\tlf^*$.
\end{fact}
 
In Shelah and Usvyatsov \cite{ShUs:844} and 
D\v{z}amonja and Shelah \cite{DzSh:692} it was asked:

\begin{qst}[Question 1.8 of \cite{ShUs:844}]
Does $\tlf^*$-maximality characterize either $SOP_3$ or $SOP_2$, maybe both?
\end{qst}

\begin{qst} \emph{(\cite{DzSh:692} Question 3.1)}
Does $SOP_2$ imply $\tlf^*$-maximality?
\end{qst}

Here we settle the question, giving a positive answer: it characterizes $SOP_2$. 
Now let us explain the connection to Keisler's order which motivates this work and our solution. 
$SOP_2$ is a property which is not yet well understood and was not known, prior to the present paper, to be a dividing line. Recently, however, 
we proved the following theorem:

\begin{thm-lit} \emph{(Malliaris and Shelah \cite{MiSh:998} Theorem 11.11)} \label{998-sop2-max}
Any theory with $SOP_2$ is maximal in Keisler's order. 
\end{thm-lit}

We conjecture there that $SOP_2$ characterizes the maximum Keisler class. In the current section, we give strong evidence for this 
conjecture by proving the result for the order $\tlf^*$, which refines Keisler's order. 

We now state the main result of this section: 
we prove that $T$ is $\tlf^*$-maximal if and only if it has $SOP_2$ (Theorem \ref{tstar-sop2-max} below). 
In light of Fact \ref{x1}, it suffices to prove that any theory with $SOP_2$ is $\tlf^*$-maximal.
Since $\tlf^*$ refines Keisler's order, even though it is not known whether it is a strict refinement, it is not sufficient to quote Theorem \ref{998-sop2-max}; rather, we 
use the technology of cofinality spectrum problems developed for the proof of Theorem \ref{998-sop2-max}. 

Given this result, it is natural to try to say more about the dividing line at $SOP_2$, which we do in \S \ref{s:1198}. 

\begin{conv}
Throughout this section $T$ denotes a complete countable first-order theory. 
\end{conv}

We now review the ``interpretability order'' $\tlf^*$, introduced in Shelah \cite{Sh:500} Definition 2.10
as a natural weakening of Keisler's order.
We first need a definition of ``interpretation.''

\begin{defn} \emph{(Interpretations, c.f. \cite{DzSh:692} 1.1)} \label{star1} 
Let $T_0$ and $T_*$ be complete first-order theories. Suppose that
\[ \overline{\vp} = \langle \vp_R(\overline{x}_R) : ~\mbox{$R$ a predicate or function symbol of $\tau(T_0)$, or $=$} \rangle \]
is such that each $\vp_R(\overline{x}_r) \in \tau(T_*)$. 

\begin{enumerate}
\item For any model $M_* \models T_*$, we define the model 
$N = {M_*}^{[\overline{\vp}]}$ as follows:

\begin{itemize}
\item $N$ is a $\tau(T_0)$-structure
\item $\dom(N) = \{ a : M_* \models \vp_{=}(a,a) \} \subseteq M_*$
\item for each predicate symbol $R$ of $\tau(T_0)$, $R^N = \{ \overline{a} : M_* \models \vp_R[\overline{a}] \}$
\item for each function symbol $f$ of $\tau(T_0)$ and each $b \in N$, 
$N \models$ ``$f(\overline{a}) = b$'' iff $M_* \models \vp_f(\overline{a}, b)$, and $M_* \models$ ``$\vp_f(\overline{a}, b) \land 
\vp_f(\overline{a}, c) \implies b = c$''. 
\end{itemize}
Note that by the last clause, we may restrict to vocabularies with only predicate symbols.

\item Say that $\overline{\vp}$ is an \emph{interpretation} of $T_0$ in $T_*$ if:
\begin{itemize}
\item each $\vp_R(\overline{x}_r) \in \tau(T_*)$ 
\item for any model $M_* \models T_*$, we have that ${M_*}^{[\overline{\vp}]} \models T_0$
\end{itemize}

\item If there exists $\overline{\vp}$ which is an interpretation of $T_0$ in $T_*$, say simply that 
``$T_*$ interprets $T_0$.'' 
\end{enumerate}
\end{defn}

\begin{defn} \emph{(The interpretability order $\tlf^*$, c.f \cite{DzSh:692} 1.2)} \label{tstar-defn}
\begin{enumerate}
\item Let $T_0, T_1$ be complete first-order theories and $\lambda$ an infinite regular cardinal. 
We say that $T_0 \tlf^*_\lambda T_1$ if there exists a theory $T_*$ such that: 

\begin{enumerate}
\item $T_*$ interprets $T_0$, witnessed by $\overline{\vp_0}$, and $T_1$, witnessed by $\overline{\vp_1}$, where
``interprets'' is in the sense of $\ref{star1}$
\item For every model $M_* \models T_*$, \emph{if} ${M_*}^{[\overline{\vp_0}]}$ is $\lambda$-saturated, then 
${M_*}^{[\overline{\vp_1}]}$ is $\lambda$-saturated
\end{enumerate} 

\item We say that $T_0 \tlf^* T_1$ if $T_0 \tlf^*_\lambda T_1$ for all large enough regular $\lambda$. 
\end{enumerate}
\end{defn}

\begin{disc} \label{rmk-sing1}
Definition $\ref{tstar-defn}$ is stated for regular cardinals, but it also makes sense for singular cardinals. 
%The main proof of this section, Lemma $\ref{x5}$, holds also for $\lambda$ singular. 
\end{disc}

%\begin{disc}
As ultrapowers commute with reducts and the choice of index models is irrelevant\footnote{Keisler \cite{keisler} proved that if 
$\de$ is a regular ultrafilter on $\lambda$ and $M \equiv N$ in a countable language then $M^\lambda/\de$ is $\lambda^+$-saturated 
iff $N^\lambda/\de$ is $\lambda^+$-saturated.} clearly $\tlf^*$-equivalence
refines $\tlf$-equivalence [equivalence in Keisler's order]. 
A priori, one would expect that $\tlf^*$ is much weaker but this is not known to be the case. In fact, the known Keisler classes
$($i.e. of stable theories$)$ coincide with those for $\tlf^*$ by \cite{Sh:500} $2.11$. 
%\end{disc}

We will use a result parallel to that familiar from Keisler's order:

\begin{fact} \emph{(Shelah \cite{Sh:500} 2.11 p. 23)}  
Any theory with the strict order property is $\tlf^*$-maximal $[$i.e. -maximum$]$. 
\end{fact}

\begin{claim} \label{x4}
Let $\cs$ be a csp %or strong csp, 
so $\xp_\cs = \xt_\cs$. 
%\begin{enumerate}
%\item 
Suppose $\ma \in \ord(\cs)$ has a $(\xt_\cs, \xt_\cs)$-cut. Then $\mct_\ma$ witnesses treetops. 
%Let $N$ be a model whose domain is a pseudofinite linear order. 
%Let $\cs$ be a strong G\"odel CSP such that $\ord(\cs)$ consists of the set of initial segments of 
%$\dom(N)$, closed under Cartesian product and initial segment. Call $\ma \in \ord(\cs)$ ``one-dimensional'' if 
%it is an initial segment of $\dom(N)$. 
%
%\item There is $\ma \in \ord(\cs)$ such that $X_\ma \subseteq M^+_1$ and $\mct_\ma$ witnesses treetops. 
%\end{enumerate}
\end{claim}

\begin{proof}
(Included for completeness, this argument simply adapts the proof of \cite{MiSh:998} Lemma 6.1, which proved that in this case 
$\mct_{\ma \times \ma}$ witnesses treetops, to show that $\mct_\ma$ witnesses treetops.)
Choose a sequence $(\langle a_\alpha : \alpha < \xt_\cs \rangle, \langle b_\alpha : \alpha < \xt_\cs \rangle)$ witnessing the cut. 
By induction on $\alpha < \xt_\cs = \xp_\cs$ let us choose a path through $\mct_\ma$ satisfying the following. 
\begin{itemize}
\item for each $\alpha$, $c_\alpha$ belongs\footnote{Alternately, rather than asking that $c_\alpha$ code a pair, ask that if $t_1 < t_2 < t_3$ are 
successive elements of $\lgn(c_\alpha)$ then either $c_\alpha(t_1) < c_\alpha(t_3) < c_\alpha(t_2)$ or else $c_\alpha(t_2) < c_\alpha(t_3) < c_\alpha(t_1)$. 
Then at limit steps, the condition is that the interval defined by $(c_\alpha(n), c_\alpha(n+1))$ includes $a_\alpha$.}
 to the definable subtree of $\mct_\ma$ consisting of elements $x$ such that: for each $n < \max\dom(x)$,
$x(n)$ codes a pair [i.e. is of the form $(a+b)^2+a$ for elements $a <_\ma  b \in X_\ma$], 
and if $m < n < \max\dom(x)$ and $x(m) = \langle a_1, b_1 \rangle$ and $x(n) = \langle a_2, b_2 \rangle$ then $a_1 <_\ma a_2 <_\ma < b_2 <_\ma < b_1$. 
\item $\beta < \alpha \implies c_\beta \tlf c_\alpha$
\item for each $\alpha$, $n_\alpha := \max \dom(c_\alpha)$
\item for each $\alpha$, $c_\alpha(n_\alpha) = \langle a_\alpha, b_\alpha \rangle$. 
\end{itemize}
The construction of this tree follows the template of \cite{MiSh:998}. At successor steps, we concatenate. At limit steps $\alpha < \xt_\cs$, we first choose an upper bound 
$c_*$ for the sequence built so far, by definition of $\xt_\cs$. Let $n_* = \max \dom(c_*)$. Then since $X_\ma$ is pseudofinite, the nonempty, bounded set 
$\{ n < n_* : c_*(n) = \langle a, b \rangle$ and $a <_\ma a_\alpha$ and $b_\alpha <_\ma b \}$ has a maximal element $n_{**}$. 
Then let $c_\alpha = {c_* \rstr_{n_{**}}}^\smallfrown\langle a_\alpha, b_\alpha \rangle$. 

Having completed the construction of the sequence, notice that it is unbounded in $\mct_\ma$, because if it were to have an upper bound $c_\star$ 
then either of the elements in the pair coded by $c_\star ( \max\dom(c_*) )$ would realize our original cut, contradiction. This completes the proof. 
%If $\mct_\ma$ witnesses treetops, then by Fact \ref{cy} the order $(\os_\ma, <_{\os_\ma})$ has a $(\xt_\cs, \xt_\cs)$-cut. 
%By closure under exponentiation, there is $\mb \in \ord(\cs)$ which is order-isomorphic to $\os_\ma$. Then 
%$\mb$ has a $(\xt_\cs, \xt_\cs)$-cut. If $\mb$ is one-dimensional, we finish. Otherwise, it is a finite Cartesian product of the form 
%$\mb = \ma_1 \times \cdots \times \ma_n$ for some $n < \omega$ and each $\ma_i \in \ord(\cs)$ ($i \leq n$) is one-dimensional. 
%Let $(\bar{a}, \bar{b})$ witness the cut. 
%Since the order $<_\mb$ is given by the G\"odel pairing function (applied $n-1$ times), the projection of this sequence 
%onto each coordinate will be a pre-cut, and in at least one case it must be a cut, otherwise we could construct coordinatewise  
%a realization to the original cut. Let $\ma_i$ be such a one-dimensional order with a $(\xt_\cs, \xt_\cs)$-cut. 
%FILL [6.1?] Then $\mct_\ma$ witnesses treetops. 
\end{proof}

We will use one further fact about trees in CSPs, which  % $\ord(\cs)$ is closed under sub-orders; what of sub-trees of $\tr(\cs)$? 
explains that $\tr(\cs)$ is quite robust. 

\begin{fact}[\cite{MiSh:998} Claim 2.13] \label{c9}
If $(\mct, \tlf_\ma)$ is a definable subtree of $(\mct_\ma, \tlf_\ma)$ %with an infinite path
and $\langle c_\alpha : \alpha < \kappa \rangle$ is a $\tlf_\ma$-increasing sequence of elements of $\mct$, 
then $\langle c_\alpha : \alpha < \kappa \rangle$ has an upper bound in $\mct$ if and only if it has 
an upper bound in $\mct_\ma$. 
\end{fact}

\begin{defn} \label{cx}
Say that a tree $\mct_\ma \in \tr(\cs)$ 
\emph{witnesses treetops} if there is in $\mct_\ma$ a strictly increasing
sequence of cofinality $\xt_\cs$ with no upper bound, i.e. the infinite minimum is attained in $\mct_\ma$. 
\end{defn}

%\begin{rmk}
Fact \ref{c9} shows that if $\mct^\prime$ is a definable subtree of $\mct_\ma$, $\mct_\ma$ witnesses treetops \emph{and} there is an infinite 
increasing sequence in $\mct^\prime$ with no upper bound, then $\mct^\prime$ witnesses treetops. But there may not 
be such a sequence, as e.g. in the trivial case when $\mct^\prime$ consists just of the root. 
%\end{rmk}

\begin{lemma}  \label{x5}
Let $T_1$ be any theory with $SOP_2$. Then $T_1$ is maximal in $\tlf^*$.
\end{lemma}

\begin{proof}
Let $T_0$ be the theory $Th(\omega + \omega^*)$ of an infinite discrete linear order with a first and last element. 
As just explained, it will suffice to show that there is a theory $T_*$ with the
following properties. 

\begin{quotation}
For any $M_* \models T_*$, 
\begin{itemize}
\item there exists an interpretation of $T_1$ in $M_*$, denoted ${M_*}^{[\bar{\vp_1}]}$
\item there exists an interpretation of $T_0$ in $M_*$, denoted ${M_*}^{[\bar{\vp_0}]}$
\item for any infinite regular cardinal $\lambda$, if ${M_*}^{[\vp_1]}$ is $\lambda$-saturated then 
${M_*}^{[\vp_0]}$ is $\lambda$-saturated.
\end{itemize}
\end{quotation}

%\step{Step 0: Construction of $T_*$.}
We build the theory $T_*$ in several steps. 
Let $\vp(x,y)$ be a formula of $T_1$ which has $SOP_2$. Let $n_\vp = \ell(y)$; here and elsewhere, we may omit overlines. 
Without loss of generality, $\tau(T_1)$ has only predicates. To begin, fix: 
\begin{itemize} 
\item a model $M_1 \models T_1$ and a set of parameters $\langle a_\eta : \eta \in {^{\omega >} \omega} \rangle$ witnessing $SOP_2$ 
for $\vp$. [Recall that by compactness, we may assume our tree witnessing $SOP_2$ is infinitely branching, with paths consistent and incomparable nodes 
inconsistent.] 
\item a model $M_2 \models (\mch(\aleph_1), \in)$. Without loss of generality, the domains of $M_1$ and $M_2$ are disjoint, and the 
relation $\epsilon$ does not appear in $\tau(M_1)$. 
\end{itemize}
Let $M$ be the model which is the disjoint union of $M_1$ and $M_2$, expanded by the additional relations and functions $\{ N, P, F \}$ as follows 
and no additional structure. (Without loss of generality, these symbols do not occur in $\tau(M_1)$ or $\tau(M_2)$.)
\begin{itemize}
\item $N^M$ names $\omega$ in $M_2$.  
\item $P^M$ names $\{ a_\eta : \eta \in {^{n>} n} \} \subseteq \dom(M_1)$. %, the $SOP_2$-tree parameters fixed above. 
\item $F^M$ 
is a unary function with domain ${^{\omega>} \omega} \subseteq M_{2}$ and range $P^{M} \subseteq M_1$,
such that $F^{M}(\eta) = a_\eta$. 
\end{itemize}
For each $n <\omega$, let $M_n$ be the model $M$ expanded further by the constant $c$ where $c^{M_n} = n \in \omega = N^{M} = N^{M_n}$. 
Let $\de$ be a nonprincipal ultrafilter on $\omega$. Let $N_* = \prod_n M_{n}/\de$ and let $T_* = Th(N_*)$. 
This completes the construction of $T_*$, the theory of the ultraproduct as just defined.  
Note that $\tau(T_*) = \tau(N_*) = \tau(M_1) \cup \tau(M_2) \cup \{ N, P, F, c \}$.   
In any model $M$ of $T_*$, the constant $c^M$ is a nonstandard element of $N^M$, i.e. of the (nonstandard) copy of $\omega$.  
 
%Let $T_*$ be the theory of the ultraproduct, as just defined. 

%We now prove that saturation behaves as desired. 
In addition to $T_*$ and $N_*$ as just defined, let 
$M_*$ be an arbitrary but fixed model of $T_*$ which satisfies: $M_* \rstr_{\tau(T_1)}$ is $\lambda$-saturated. 
Note that $T_*$ interprets $T_0$ and $T_1$, that is, there exist $\overline{\vp}_0$, $\overline{\vp}_1$ witnessing $\ref{star1}$.
This is because $M_* \rstr_{\tau(T_1)} \models T_1$ by construction, while 
$N_0 := N^{M_*} \rstr c^{M_*} = (\{ a \in M_* : M_* \models a \in c \}, \in^{M_*}) \models T_0$. 
\br

%\step{Step 2: The cofinality spectrum problem $\cs$.}
To analyze saturation, we define a cofinality spectrum problem as follows.  %EXPLAIN WHAT THE DOMAIN OF THE MODEL IS.    
%Clearly $\mathbb{N}$ is definable in each $M_{1,n}$. Then $\mathbb{N}^{M_*}$ is the definable set of nonstandard integers, 
%and $c^{M_*}$ belongs to this set and is nonstandard.  
First, let $\ord^{0}(\cs)$ be the smallest set containing the set of all finite nonempty initial segments of ${N}^{M_*}$, in each case 
letting $d_\ma = \max X_\ma$.  [We may use any non-empty order in $\mch(\aleph_0)$.]
For each $\ma, \mb \in \ord^{0}(\cs)$, let $\ma \times \mb$ be the definable, pseudofinite linear order on $X_\ma \times X_\mb$ 
given by the G\"odel pairing function, and let $d_{\ma \times \mb} = (d_\ma, d_\mb)$. Let $\ord(\cs)$ be the closure as just described 
of $\ord^{0}(\cs)$ under Cartesian product, as well as initial segment. [If $\mb$ is an initial segment of $\ma$, 
define $d_\mb = \min \{ \max X_\mb, d_\ma \}$.] 

Suppose for a moment that $M_* = N_*$, i.e. suppose we are really in the case of the ultrapower. 
Then for each $\ma \in \ord(\cs)$, %[recall these are formulas over the empty set] 
$\vp_\ma$ is a formula $\vp(x,y,z)$ over the empty set. Because we are in $N_*$, for each suitable parameter $\bar{a}$, 
the set $\vp(x,x,\bar{a})$ is contained in the ultraproduct of finite sets. 
Let $\psi(w,z)$ be such that for each suitable $\bar{a} \in {^{\ell(\bar{z})}M_{n}}$ 
[though by construction, we can restrict to the case where $\bar{a} \subseteq \mathbb{N}^{M_{n}}$]  
we have that $\mct_{\vp, \bar{a}} = \{ \eta : M \models \psi(\eta, \bar{a}) \}$ is the set of finite sequences of members of the finite set 
$X_{\vp, \bar{a}} = \{ b : M \models \vp(b,b,\bar{a}) \}$ 
of length $< \max X_{\vp, \bar{a}}$. Let $\tlf = \{ (\eta, \nu) : \eta, \nu \in \mct_{\vp, \bar{a}}$ and $\eta$ is an initial segment of $\nu$ $\}$. 
These are definable in $M_{n}$ by the choice of its theory. 
The length and evaluation functions $\lgn$ and $\xr$ can be defined likewise. By \lost theorem, these formulas will define the appropriate trees in the ultrapower.  
Thus, we get a cofinality spectrum problem. Moreover, for each given $\vp_\ma$, each $\bar{a}$ of length $\ell(z)$, and each $t <\omega$, 
we have that in each index model, the tree $\mct = \mct_{\vp, \bar{a}[t]}$ is finite, so its flattening $(\os, <_\os)$ [in the sense of \ref{d:coll} above] 
can be injectively mapped into $N^{M_{n}}$. That is, $M_{n} \models$ ``there exists 
an injective order-preserving map of $(\os, <_\os)$ into $N$''. As the flattening is uniformly definable from $\psi$, and we are in a model with 
sufficient set theory, such an injection exists also in the ultrapower and so $\cs$ has (strong) exponentiation. 

Returning to the case of arbitrary $M_* \equiv N_*$, as we are in a model with sufficient set theory, note that the existence of such associated trees 
and injections are elementary properties of each given $\vp$ and $\psi$. Thus, also in this more general case, 
$\cs$ is a cofinality spectrum problem with exponentiation.

%
%
%\begin{quotation}
%[\emph{disc}. Recall that ultraproducts (or powers) commute with reducts. We could, indeed, 
%first expand each model $M_{1,n}$ to the a model of the disjoint union $\{ M_{1,n} : n <\omega \}$, each named by a distinguished predicate, and \emph{thus} 
%treat the ultraproduct as an ultrapower. So for the purposes of producing a cofinality spectrum problem, 
%one could, if desired, take $T_*$ to be the theory of 
%the ultraproduct in some large expanded vocabulary in which the closure of the set of orders under 
%Cartesian products and all other closure conditions requested in \cite{MiSh:998} Definition 2.1 are explicitly given.  
%Alternately, one could simply expand each index model to model sufficient set theory or Peano arithmetic. 
%
%In either case, because of the ultraproduct, the condition that every nonempty definable linearly ordered 
%subset of $X_\ma$ and of $\mct_\ma$ have a first and last element is guaranteed by their being pseudofinite in $T_*$. 
%Note also that because of the ultraproduct, necessarily $d_\ma$ %(i.e. the length of the longest sequence of $0$s in the tree $P^{M_*}$) 
%is nonstandard, so $\ma$ is nontrivial.]
%\end{quotation}
\br

%Let $N_* \models T_*$ be such that $N \rstr_{\tau(T_1)}$ is $\lambda$-saturated. 

%\step{Step 3. $\xt_\cs \geq \lambda$.}

Having defined $\cs$, we now prove that $\xt_\cs \geq \lambda$. 
First we verify that it is sufficient to look at one-dimensional trees, i.e. that treetops are witnessed by some $\mct_\ma$ where $X_\ma \subseteq N^{M_*}$. 
By Theorem \ref{ps-is-ts}, there is a $(\xt_\cs, \xt_\cs)$-cut in some $X_\ma$, $\ma \in \ord(\cs)$.  If $X_\ma$ is one-dimensional, then by 
Claim \ref{x4}, we finish. If not, $X_\ma$ is contained in some finite Cartesian power of $N^{M_*}$. Let $f: X_\ma \rightarrow N^{M_*}$ be a definable injection, 
given by applying the G\"odel pairing function finitely many times. By the definition of $\cs$, this is an order preserving map. (This appeal 
to the definition is not necessary: we could simply use that such a map is an injection of sets, so by Conclusion \ref{c2b} of the next section, 
there is an injective order-preserving map into some one-dimensional $X_\mb$.) Necessarily $X_\mb$ will have a $(\xt_\cs, \xt_\cs)$-cut, 
so again by Claim \ref{x4}, we finish.  

So to show $\xt_\cs \geq \lambda$, it will suffice to show that for any one-dimensional $\ma$, and any $\kappa < \lambda$, 
any $\kappa$-indexed strictly increasing sequence $\langle \eta_\alpha : \alpha < \kappa \rangle$ in $\mct_\ma$ has an upper bound. 
Recall the constant $c$ from the signature of $N_*$ and $M_*$. 
Since $M_* \equiv N_*$, any one-dimensional $X_\ma$ is contained in a nonstandard initial segment, i.e. for some $m$, 
$\mathbb{N}^{M_*} \models$ ``$c \leq m$ and $X_\ma \subseteq M_* \rstr m$''. By Fact \ref{c9}, without loss of generality $X_\ma$ is the full initial segment below $m$
and $\mct_\ma$ is the full tree ${^{m>}m}$. In $(\mch(\aleph_0), \epsilon)^{M_*}$
we define $\mct_\ma = ( {^{m>}m}, \tlf)$ and let 
\[ \psi(y, m) := (\exists \eta \in \mct_{\ma} )(F(\eta) = y)) \]
define the subset of $P^{M_*}$ corresponding to the image $F(\mct_\ma)$. 
Now if $\overline{\eta} = \langle \eta_\alpha : \alpha < \kappa \rangle$ is increasing in $\mct_\ma$, then recalling the $SOP_2$-formula $\vp$ from the 
beginning of the proof, 
\[ p_{\overline{\eta}} = \{ \vp(x, F(\eta_\alpha) ) : \alpha < \kappa \} \] 
is a consistent partial type\footnote{The fact that comparability of elements in the domain of $F$ is reflected in the consistency of their images 
will carry over from the models $M_n$ to $N_*$ by \lost theorem and from there to $M_*$ by elementary equivalence.}
in $M_* \rstr \tau(T_1)$. By the assumption of $\lambda$-saturation in that signature, some $d \in M_*$ realizes $p$, hence 
\[ \{ \eta \in \mct_\ma : \vp[d,F(\eta)] \} \] 
is a subset of a branch of $\mct_\ma$ by definition of $SOP_2$ and of $F$, and is definable in $M_*$. Then the set 
\[ \{ a \in X_\ma : a < d_\ma \land (\exists \eta)(\eta \in \mct_\ma \land \vp[d,F(\eta)] \land a \in \dom(\eta) ) \} \]
of lengths of such elements 
is a definable, nonempty, bounded subset of $X_\ma$, so contains a greatest element $a_*$. 
Any $\eta$ along the distinguished branch whose domain contains $a_*$ will be an upper bound for $\overline{\eta}$. This completes the proof that 
$\xt_\cs \geq \lambda$. 

\br
Recall that we had set $N_0 := N^{M_*} \rstr c^{M_*}$ % = (\{ a \in M_* : M_* \models a \in c\}, \in^{M_*}) \models T_0$ 
as the domain for our interpretation of $T_0$ in $M_*$. 
%Towards saturation $N_0$ has no $(\kappa_1, \kappa_2)$-cuts for regular $\kappa_1, \kappa_2 < \xt_\cs$.
Since $\cs$ is a CSP with exponentiation, $\xp_\cs = \xt_\cs = \lambda$ by Theorem \ref{ps-is-ts}. 
%Recall $N_0$ from Step 1.
%Since we are in the case of a csp with exponentiation, $\xp_\cs = \xt_\cs$, by Theorem \ref{ps-is-ts}. 
By definition of $\xp_\cs$, there are no $(\kappa_1, \kappa_2)$-cuts 
in any $X_\ma$, for $\ma \in \ord(\cs)$ and $\kappa_1, \kappa_2 < \xt_\cs$. 
This in particular is true for $\ma$ with $X_\ma = \{ a : M_* \models a < c \}$ and $d_\ma = \max X_\ma$. 
In fact, for arbitrarily large $a \in N^{M_*}$, $\ord(\cs)$ contains $\ma$ such that $X_\ma \supseteq N \rstr_{\leq a}$,  
in which, therefore, we have no $(\kappa_1, \kappa_2)$-cuts.  
%so this means there are no $(\kappa, \kappa)$-cuts for $\kappa < \xt_\cs$ in $N_0$, either. 

Finally, let us prove that ${M_*}^{[\vp_0]}$, i.e. the $\tau_0$-submodel whose domain is $N_0$, is $\lambda$-saturated. 
We know that for models of $T_0$, every formula is a Boolean combination of the formulas $x = y$, $x < y$, and $\vp_k(x,z) = (\exists^{!k} y)(x<y<z)$. 
Thus $N_0 \models T_0$ is $\lambda$-saturated iff every cut $(C_1, C_2)$ of $N_0 \subseteq N^{M_*}$ of cofinality $(\kappa_1, \kappa_2)$ is filled, where 
$\kappa_1, \kappa_2$ are regular and $\kappa_1 + \kappa_2 \leq \lambda$. This was proved in the previous paragraph.  

This completes the proof. 
\end{proof}

%Note that in the construction just given, we may remove the assumption that $\lambda$ is regular. 

%Recalling Remark \ref{rmk-sing1}, notice: 
%
%\begin{cor} 
%Let $T_1$ be any theory with $SOP_2$, let $T_0$ be the theory of a discrete linear order with endpoints, and 
%let $T_*$ be the theory built in Lemma \ref{x5}.  Then: 
%
%\begin{quotation}
%For any $M_* \models T_*$, 
%\begin{itemize}
%\item there exists an interpretation of $T_1$ in $M_*$, denoted ${M_*}^{[\bar{\vp_1}]}$
%\item there exists an interpretation of $T_0$ in $M_*$, denoted ${M_*}^{[\bar{\vp_0}]}$
%\item for \emph{any} infinite cardinal $\kappa$, if ${M_*}^{[\vp_1]}$ is $\kappa^+$-saturated then 
%${M_*}^{[\vp_0]}$ is $\kappa^+$-saturated.
%\end{itemize}
%\end{cor}
%
%\begin{proof}
%It suffices to prove this 
%{rmk-sing1}
%\end{proof}

\begin{theorem}[GCH] \label{tstar-sop2-max}
$T$ is $\tlf^*$-maximal if and only if it has $SOP_2$.
\end{theorem}

\begin{proof}
By Fact \ref{x1} (which assumes relevant instances of GCH) and Lemma \ref{x5}. 
\end{proof}

%Theorem \ref{tstar-sop2-max} gives the first real evidence for $SOP_2$ as a dividing line.

\br

\section{Useful tools and additional definitions} \label{s:tools}

Before turning to a structure theory for $NSOP_2$, we prove several additional facts about CSPs: 
Claim \ref{c2} and Conclusion \ref{c2b}, which show that from a suitable bijection of sets, 
we can recover an order-isomorphism.  It follows that the assumption that the order on all pairs 
was given by the G\"odel pairing function in the CSP constructed in Lemma \ref{x5} could be weakened, 
as mentioned in that proof.  Tying up loose ends, we show that one of the main consequences of \ref{d:estt}(6) can be 
recovered in weak hereditary CSPs, Claim \ref{c:weak-double-star}, and discuss the barrier to fully 
recovering \cite{MiSh:998} for such CSPs in \ref{uniformity}.  
Finally, we include a definition of ``strong'' CSPs, natural when the underlying model $M^+_1$ 
is totally ordered.

\begin{defn} \label{d:smaller} Let $\mb \in \ord(\cs)$. We say that $Z \subseteq X_\mb$ is \emph{small} 
in $X_\mb$ if: 
\begin{enumerate}
\item[(a)] there is some definable $V$, $Z \subseteq V \subset X_\mb$ such that 
$M^+_1 \models$ ``there does not exist $x \in \mct_\mb$ such that $x$ is a bijection from $X_\mb$ into $V$.''
\item[(b)] if $Z \neq \emptyset$, then $\max (Z) < d_\mb$
\end{enumerate}
When $Z \subseteq X_\mb$ is small in $X_\mb$ and is an initial segment, we call it a \emph{small initial segment}. 
\end{defn}

\begin{rmk}
Condition \ref{d:smaller}(b) allows for concatenation. 
\end{rmk}

\begin{claim} \label{c2}
Suppose we are given a weak cofinality spectrum problem $\cs$, $\mb \in \ord(\cs)$, $h \in M^+_1$, $(W, <_W)$ a definable and pseudofinite linear order, and $Z$ 
a small initial segment of $Y_\mb$, such that: 
\[ M^+_1 \models \mbox{`` $h$ is a partial bijection with $W \subseteq \dom(h)$, $\rn(h)\subseteq Z$''} \] 
Then $({W}, <_{W})$ is internally order-isomorphic to an initial segment of $Y_\mb$. 
\end{claim}

\begin{proof} 
Let $({W^*}, <_{W^*})$ denote the definable set $h(W) \subseteq X_\mb$ with the definable linear order given by 
$ x <_{W^*} y ~\iff~ h^{-1}(x) <_{W} h^{-1}(y) $.  
This allows us to identify $({W^*}, <_{W^*})$ definably with $({W}, <_{W})$, and now we 
prove that  $({W^*}, <_{W^*})$ 
is internally order-isomorphic to an initial segment of $(X_\mb, \leq_\mb)$. 

Let $\mct$ be the definable subtree of $\mct_\mb$ given by $\vp(x)$, which says: ``$x \in \mct_{\mb}$ is a one to one 
order preserving function from an initial segment of $(X_{\mb}, \leq_{\mb})$ onto an initial segment 
of $({W^*}, <_{W^*})$''. Thus if $a <_\mb b$ are in $\dom(x)$ then $x(a) <_{W^*} x(b)$.

First, note that the set $\eff = \{ x \in \mct : \vp(x) \}$ is linearly ordered. This is because
any two elements $x, x^\prime \in \mct$ must agree on the first element of their domain; and if 
neither of $x, x^\prime$ extends the other, then the set in their common domain on which they agree is nonempty and definable. 
But if the last element of this set is not $\max \dom(x)$, we get a contradiction since the order is pseudofinite. 

Consider the subset of $X_{\mb}$ given by:
\[ \{ a \in X_{\mb} : \mbox{there is $x$ such that $\vp(x)$ and $a \in \dom(x)$} \} \]
As this subset is nonempty and definable, by assumption it has a last element $a_* \in X_{\mb}$.   
Let $x_* \in \mct$ be a function witnessing this, i.e. such that $a_* \in \dom(x_*)$. 
Necessarily $x_* = \max \eff$. 

There are three cases.

\step{Case 1.} The desired case: $\rn(x_*) = W^*$.

\step{Case 2.} Not case 1, but $\dom(x_*) = X_{\mb}$. Then 
\[ M^+_1 \models \mbox{ ``$x_* \in \mct_\mb$ is an injection of $X_{\mb}$ into $Z$''} \] 
contradicting Definition \ref{d:smaller}.

\step{Case 3.} Not case 1 or 2, so $\dom(x_*) \subsetneq X_{\mb}$ and $\rn(x_*) \subsetneq W^*$. 
By Definition \ref{d:smaller} and the hypotheses of the Claim, ``not case 2'' implies $\lgn(x_*) < d_{\mb}$. 
Writing $S$ for successor, note that as the orders are pseudofinite, the function ${x_*}^\smallfrown \langle S(f(a_*)) \rangle$ is well defined 
(meaning the function extending $x_*$ by the additional condition $S(a_*) \mapsto S(f(a_*))$ is well defined).
We may concatenate, so this new function belongs to $\mct$, contradicting the choice of $a_*$. 

As Cases 2 and 3 are contradictory, we are necessarily in Case 1, which completes the proof. 
\end{proof}

\begin{concl} \label{c2b} 
Suppose that $\cs$ is a weak cofinality spectrum problem, and:
\begin{enumerate}
\item $N \subseteq M_{1,\cs}$ is definable and linearly ordered. 
\item Arbitrarily large initial segments of $N$ are orders for $\cs$, more precisely, 
there is $\psi \in \Delta$ so that for any $a \in N$ 
there is $\ma \in \ord(\cs)$ with: \\ $\{ b \in N ~:~  \models b \leq a \} \subseteq \{ b \in M_{1,s} : \models \psi(b,b,a) \}=  X_\ma \subseteq N$ and $a < d_\ma$. 
\item Cardinality of initial segments of $N$ grows internally, more precisely, letting $\theta = \theta(w,y)$ define the tree associated to $\psi(x_1, x_2, y)$,  
we have that $M_{1,\cs} \models$ ``$(\forall y_0 \in N)(\exists y_1 \in N)($there does not exist $z \in \theta(w,y_1)$ such that $z$ is a bijection from $\psi(x,x, y_1)$ into $\psi(x,x,y_0))$''.
\end{enumerate}
Let $(W, <_W)$ be a definable, pseudofinite linear order in the model. If there is an internal injection $f$ of sets 
from $W$ into some initial segment of $N$, then for some $\mb \in \ord(\cs)$ with $X_\mb \subseteq N$ there exists an internal order-isomorphism $g$ from  
$(W, <_W)$ onto an initial segment of $Y_\mb$. 
\end{concl}

\begin{proof}
By the hypothesis, %$D = N$ in the notation of Definition \ref{m23}, 
we can find $\mb_0$ such that $X_{\mb_0} \supseteq f(W)$, 
and $d_{\mb_0} > \max(f(W))$. 
Let $\psi(z)$ be the formula in $\Delta$ defining the pseudofinite linear order corresponding to 
the initial segment below $z$, and let $\theta(z)$ define its associated tree. 
Then $M_{1,\cs} \models$ ``$(\exists z)($there does not exist $x \in \theta(z)$ such that $x$ is a bijection from $\psi(z)$ into $X_{\mb_0})$''.
Let $c$ be any such $z$. Again by hypothesis, there is $\mb \in \ord(\cs)$ such that $X_\mb \supseteq \{ x : x \leq c \}$ and $d_\mb > c$. 
Now apply Claim \ref{c2}. 
\end{proof}

\br

\begin{claim} \label{c:weak-double-star}
If $\cs$ is a nontrivial weak c.s.p. and is hereditarily closed, then there exist nontrivial $\ma, \ma^\prime \in \ord(\cs)$ 
which together satisfy $(**)$ of $\ref{on-pairing}$.
\end{claim}

\begin{proof} 
Let $\mb$ (so also $\mct_\mb$) be nontrivial.  
In section 5 of \cite{MiSh:998}, we showed that for any $\mb \in \ord(\cs)$, 
it is possible to define addition, multiplication, and exponentiation on any element of $\ord(\cs)$, %e.g. on $X_\mb$ 
(that is, to define relations on $X_\mb$ which have all the same properties as the graphs 
of these functions, except that they are possibly not total). This does not require any assumptions on pairing. 
%In \cite{MiSh:998} Lemma 5.3, we exhibit a formula $\psi(y)$ asserting that G\"odel codes exist for all functions from $X_\mb \rstr y$ to itself, 
%which holds on all standard $n \in X_\mb$.\footnote{
This is done in 
the proof of \cite{MiSh:998} Lemma 5.3, essentially as follows. Addition is given by: 
$\vp_+(x,y,z) = (\exists \eta \in \mct_\mb)(\lgn(\eta) = y \land \eta(0) = x \land \eta(y-1) = z \land (\forall i)(i <\lgn(\eta) \implies \eta(S(i)) = S(\eta(i)) )$. 
[We omit the parameter $\bar{c}_\mb$ for readability.]
To obtain multiplication, $\vp_\times(x,y,z)$, substitute
``$(i < \lgn(\eta) \implies \eta(S(i)) = \eta(i)+x )$'' as necessary, and define exponentation 
$\vp_{exp}(x,y,z)$ by substituting in the appropriate place 
``$(i < \lgn(\eta) \implies \eta(S(i)) = \eta(i) \times x )$,'' i.e. requiring that the sequence increment by a factor of $x$.  
These are graphs of partial functions, which need not be total. We can therefore define ``$x$ is a prime.'' 
Let $\vp_4(x,y)$ assert that $x$ is the $y$th prime by saying: 
$y>0 \land (\exists \eta \in \mct_\mb)(\lgn(\eta) = y \land \eta(0) = 2 \land \eta(y-1) = x \land (\forall i)(i < \lgn(\eta) \implies 
\mbox{``}\eta(S(i)) \mbox{ is the $\leq_\mb$-least prime number strictly greater than $\eta(i)$} \mbox{''}))$.
Let $\vp_5(x,n,m)$ assert that $x$ is divisible by the $n$th prime precisely $m$ times, by asserting the existence of $\eta \in \mct_\mb$ of length $m$ whose first element is 
$x$, whose subsequent elements decrease by a factor of the $n$th prime and whose last element has no more such factors. 
Let the formula $\vp_6(x,\eta)$ assert that $x \in X_\mb$ is a G\"odel code for $\eta$ by stating: 
``$\eta \neq \emptyset$, $\eta \in \mct_\mb$, $x > 2$ and for all $i < \lgn(\eta)$,  
writing $m = \eta(i)$, we have that $x$ is divisible by the $i$th prime precisely $m+1$ times''. 
Let $\theta(x)$ assert that $(\forall y < x)(\exists z) \vp_{+}(x,y,z) \land (\forall y < x)(\exists z) \vp_{\times}(x,y,z) \land (\forall y < x)(\exists z) \vp_{exp}(x,y,z)$.
Let $\psi(y)$ be the formula: 
\[ (\forall x < y)\theta(x) \land (\forall \eta \in \mct_\mb)\left((\lgn(\eta) < y \land (\forall i < \lgn(\eta))(\eta(i) < y)) \implies (\exists x)(\vp_6(x,\eta))\right) \] 
which asserts that G\"odel codes exist for all functions from $X_\mb \rstr_y$ to itself. 

For our present case, apply this as follows. Recall that $X_\mb$ is pseudofinite. 

First step: Find nonstandard $n_* \leq d_{\mb}$ so that the definable set $Z$ of codes for pairs of elements of $[0, n_*]_{\mb}$ is contained in $Y_\mb$. 
%Second step: Using hereditary closure, let $\ma_0 \in \ord(\cs)$ be such that $X_{\ma_0}$ is internally isomorphic to $Z$. 
Second step: Find $n_{**} \in X_\mb$, $n_{**} \leq n_*$ still nonstandard, so that $X_\mb$ contains all codes for functions from $[0, n_{**}]_\mb$ to itself. 
Third step: Let $m_*$ be maximal $\leq_\mb n_{**}$ such that $\langle m_*, m_* \rangle \leq n_{**}$, so necessarily $m_*$ is nonstandard too. 
Let $\ma_1$ be such that $X_{\ma_1} = [0, n_{**}]_\mb$ and $d_{\ma_1} = m_*$. 

Now define $\ma$ so that $X_\ma = [0, m_*]_\mb$ and $d_{\ma} = m_*$. 
What about the desired tree $\mct$ of functions from $X_\ma$ to $X_\ma \times X_\ma$? Recalling that $d_{\ma_1} = m_*$, 
this tree is naturally isomorphic to the definable sub-tree $\mct \subseteq \mct_{\ma_1}$ whose elements are functions whose range 
consists only of codes for pairs of elements each of which are $\leq m_*$. 
By construction, the codes for elements of $\mct_{\ma_1}$ and therefore for elements of $\mct$ form a definable subset of $X_\mb$.  
Finally, let $\ma^\prime = \mb$. 

%So, in our present case, modify $\psi(y)$ in a straightforward way to obtain $\psi_*(y)$ which asserts that G\"odel codes below $d_\mb$ 
%exist for all functions from $X_\mb \rstr_y$ to the subset $Z$ of $X_\mb \rstr_y$ consisting of G\"odel codes below $d_\mb$ of pairs 
%of elements $(a_1, a_2)$ of $X_\mb$ with $a_1 \leq_\mb y$ and $a_2 \leq_\mb y$.  
%
%The formula $\psi_*(n)$ holds on all standard $n \in X_\mb$, so it also holds on some nonstandard $n_* < d_\mb$.  Recalling 
%\ref{m11a} and the assumption of hereditary closure, let $\ma \in \ord(\cs)$ be given by $X_\ma = X_\mb \rstr n_*$, $\leq_\ma = \leq_\mb \rstr n_*$, 
%and $d_\ma = n_*$. 
%Let $\ma^\prime = \mb$. 
This completes the proof. 
\end{proof}

\begin{disc} \label{uniformity}
By a similar argument, in any weak CSP which is hereditarily closed, for some nontrivial $\ma \in \ord(\cs)$ we have available nontrivial 
elements of $\ord(\cs)$ which can be thought of as canonically representing any given one of the finite Cartesian powers $~\ma \times \cdots \times \ma~$ with the desired 
ordering, e.g. one derived from repeated applications of the G\"odel pairing function. However, as noted in Discussion \ref{on-pairing}, we don't a priori 
have Cartesian products of distinct $\ma, \mb \in \ord(\cs)$. This prevents us from obtaining the uniformity of functions such as 
$\lcf$ across \emph{all} $X_\ma$ which was necessary for the main theorems of \cite{MiSh:998} to go through. 
\end{disc}

\begin{defn}[Strong CSPs] \label{d:strong-csp}
Call $\cs$ a \emph{strong CSP, or lexicographic CSP} if the demands on Cartesian products from $\ref{d:estt}(5)$ are replaced by:\footnote{So here $\Delta$ is retained.} 
\begin{enumerate}
\item if $\ma, \mb \in \ord(\cs)$ then internally either $|Y_\ma| \leq |Y_\mb|$ or $|Y_\mb| \leq |Y_\ma|$. %, see \ref{a4} below.

\item if $\ma \in \ord(\cs)$ and $d\leq_\ma d_\ma$ then there is $\mb \in \ord(\cs)$ such that:
\begin{enumerate}
\item $X_\mb = X_\ma$, $<_\mb = <_\ma$, $d_\mb = d$
\item $\mct_\mb = \mct_\ma \rstr \{ \eta : (\lgn_\ma(\eta) \leq d) \land (\forall n < \max \dom(\eta))(\eta(n) \leq d) \}$
\end{enumerate}

\item if $\ma, \mb \in \ord(\cs)$ and $h$ is a definable\footnote{Not necessarily via $\Delta$.} 
isomorphism from $(Y_\ma, <_\ma)$ onto $(Y_\mb, <_\mb)$ then there is $\bc = \ma \times \mb$ 
such that $X_\bc = X_\ma \times X_\mb$, $<_\bc$ is the lexicographic order, $d_\bc = (d_\ma, 0_\mb)$, and $\mct_\bc$ is naturally defined.  
\end{enumerate}
\end{defn}

\begin{obs} \label{s-w-gen}
Any strong CSP is a weak CSP, and the hereditary closure of a strong CSP is a strong CSP. 
\end{obs}

%In this case we have enough for the main results of \cite{MiSh:998}. 

\begin{claim} \label{c42}
If $\cs$ is a strong CSP which is hereditarily closed, then: 
\begin{enumerate}
\item[(a)] $\cs$ satisfies Property $(**)$ of Discussion $\ref{on-pairing}$. 
\item[(b)] Moreover, for any nontrivial $\ma, \mb \in \ord(\cs)$, there is a nontrivial $\bc \in \ord(\cs)$ with either  
$X_\bc = X_\ma \times Y_\mb$ or $X_\bc = X_\mb \times Y_\ma$. 
\item[(c)] Suppose that for every $\ma \in \ord(\cs)$ there is $\ma^\prime \in \ord(\cs)$ such that 
$X_\ma$ is internally isomorphic to a subset of $X_{\ma^\prime}$ and $d_{\ma^\prime} = \max(X_{\ma^\prime})$. 
Then $\cs$ is closed under Cartesian products.  
\end{enumerate}
\end{claim}

\begin{proof}
%By Claim \ref{c:weak-double-star}, it suffices to check the Cartesian products, condition \ref{d:estt}(5). 

(1) By Claim \ref{c:weak-double-star} and Observation \ref{s-w-gen}. 

(2) Fix nontrivial $\ma, \mb \in \ord(\cs)$. 
By condition \ref{d:strong-csp}(1), without loss of generality, $|Y_\ma| \leq |Y_\mb|$ 
witnessed by an internal partial isomorphism $h$. Since Condition \ref{d:strong-csp}(3) requires $h$ to be surjective, 
let $d = h(d_\ma) \in Y_\mb$. Since $\cs$ is hereditarily closed and $d \leq d_\mb$, apply Observation \ref{m11a} to find 
$\mb^{\prime\prime} \in \ord(\cs)$ such that $X_{\mb^{\prime\prime}} = Y_\mb$ and $d_{\mb^{\prime\prime}} = d$. 
Now $h: Y_\ma \rightarrow Y_{\mb^{\prime\prime}}$ is onto,  
so by condition \ref{a2x}(c) there is $\bc = \ma \times \mb^{\prime\prime}$ with 
$X_\bc = X_\ma \times X_{\mb^{\prime\prime}} = X_\ma \times Y_{\mb}$ 
and $d_\bc$ is nonstandard, thus $\bc$ is nontrivial. 

(3) As $d_{\ma^\prime} = \max(X_{\ma^\prime})$, $d_{\mb^\prime} = \max(X_{\mb^\prime})$, we have that $Y_{\ma^\prime} = X_{\ma^\prime}$, $Y_{\mb^\prime} = X_{\mb^\prime}$, and  
the previous condition (2) shows that their Cartesian product exists. 
The hypothesis of (3) allows us to find Cartesian products for any two elements of $\ord(\cs)$ by 
first isomorphically embedding them in suitable larger elements $\ma^\prime, \mb^\prime$, finding 
$\bc^\prime = \ma^\prime \times \mb^\prime$ and then applying hereditary closure. 
\end{proof}

\section{Towards a structure theory for $NSOP_2$} \label{s:1198}

In Section \ref{s:tlf}, we gave the first real evidence that the strong tree property $SOP_2$ is a dividing line. 
Motivated by this result, we now look for the beginnings of a structure 
theory for $NSOP_2$. The key objects are so-called higher formulas, defined using ultrafilters.  
The main results are first, Theorem \ref{a13}, which characterizes $NSOP_2$ in terms of few higher formulas; 
second, the Symmetry Lemma \ref{sym-lemma}, which characterizes $NSOP_3$ in terms of symmetric inconsistency for higher formulas; 
and third, Theorem \ref{ess-sop2}, which proves that $SOP_2$ is sufficient for a certain kind of exact saturation to fail. 
%and third, Theorem \ref{b2}, which characterizes $NSOP_2$ by counting in a way reminiscent of stability. 
%, which 
%deals with an analogue of nonforking extensions. %xxxx CHECK. 

\begin{conv}
Throughout this section, $T$ is a complete first order theory and $\mathfrak{C} = \mathfrak{C}_T$ is a monster model for $T$. 
\end{conv}

$SOP_2$ was defined in \ref{d:sop2} above. $SOP_3$ was first defined in Shelah \cite{Sh:500} as a weakening of the strict order property; note that 
in Definition \ref{d:sop3}, the case where $\psi = \neg \vp$ is the strict order property. 

\begin{defn}[ \cite{Sh:500} 2.20, \cite{ShUs:844} 1.3] \label{d:sop3}
$T$ has $SOP_3$ if there is an indiscernible sequence $\langle \overline{a}_i : i < \omega \rangle$ and formulas $\vp(\overline{x}, \overline{y})$, 
$\psi(\overline{x}, \overline{y})$ such that:
\begin{enumerate}
\item $\{ \vp(\overline{x}, \overline{y}), \psi(\overline{x}, \overline{y}) \}$ is contradictory.
\item for each $k<\omega$, the following is a consistent partial type:
\[  \{ \psi(\overline{x}, \overline{a}_j) : j \leq k \} \cup \{ \vp(\overline{x}, \overline{a}_i) : i > k \} \]
\item for $j<i$, the set $\{ \vp(\overline{x}, \overline{a}_i), \psi(\overline{x}, \overline{a}_j) \}$ is contradictory.
\end{enumerate}
\end{defn}

It is known that $SOP_3$ implies $SOP_2$ but it is open whether, on the level of theories, the converse is true; so it is possible that the 
Symmetry Lemma below will also characterize $SOP_2$. (Still, for pairs $(T, \Delta)$ the converse fails since $SOP_2$, $SOP_3$ are known to be 
distinct at the level of formulas.)

We first look for a useful way to capture the asymmetry of $SOP_3$. This approach relates to the 
idea of ``semi-definability'' from \cite{Sh:a} VII.4.

\begin{defn} %\emph{(c.f \cite{Sh:c} VII.4)}
If $\de$ is an ultrafilter on $A \subseteq \mathfrak{C}^m$, then for any set $B \subseteq \mathfrak{C}$ we define:  % note: not $^m\mathfrak{C}$! 
\[  \Av(\de, B) = \{ \psi(\overline{x};\overline{b}) : \overline{b} \in B, \psi \in [\ml], \{ \overline{a} \in A : \models \psi(\overline{a},\overline{b}) \} \in \de \} \]
so this is an element of $\st^m(B)$. %Write $m(\de)$ for the arity $m$. 
\end{defn}

%\begin{expl} In \cite{Sh:c} VII.4 there were related definitions:
%\begin{enumerate}
%\item We say that the type $p \in \st(B)$ is \emph{semi-definable over $A$} when there is an ultrafilter $\de$ on $A$ such that 
%\[ p \subseteq \Av(\de, \dom(p))  \]
%\item We say that the $m$-type $p \in \st_m(B)$ is semi-definable over $A$ when it is semi-definable over ${A}^m$.  
%\end{enumerate}
%Every type over a model $M$ is semi-definable over $M$, because it is finitely realized in $M$.  
%\end{expl}
%%(note: notation changes)

\begin{defn} \label{d:24}
Let $\de$ be an ultrafilter on ${^mA}$. $($We sometimes write $m = \mathbf{m}(\de)$ for this arity.$)$  
\begin{enumerate}
\item We say that the infinite indiscernible sequence 
$\overline{b} = \langle \overline{b}_s : s \in I \rangle$ is \emph{based on $\de$} when 
\[ \tp(\bar{b_s},A + \bar{\mathbf b}_{>s}) = \Av(D,A + \bar{\mathbf b}_{>s})) \]
where $\bar{\mathbf b}_{>s} = \bigcup \{ {\bar b}_t ~:~ t \in I, ~ s <_I t \}$.  
\item For each $\de$, let $\ob(\de) = \ob(\de, A)$ be the set of such $\overline{\mathbf b}$, i.e. the set of all infinite indiscernible sequences based on $\de$ 
$($assuming the monster model $\mathfrak{C}$ is well defined$)$. 
\end{enumerate}
\end{defn}

In Definition $\ref{d:24}$, the elements approach $A$; of course we could have inverted the order.  
Given such an ultrafilter $\de$ and an infinte indiscernible sequence built from it, we may 
naturally ask when a given formula instantiated along this sequence is consistent. 
%We will be interested in consistency and inconsistency of formulas instantiated along these special indiscernible sequences. 

\begin{defn} \label{higher-k}
Let $A \subseteq \mathfrak{C}$ and $\vp = \vp(\overline{x}, \overline{y}) = \vp(\overline{x}, \overline{y}, \overline{c})$ for $\overline{c} \in \mathfrak{C}$. 
\footnote{Usually, $\overline{c}$ is empty, and in any case, we can just incorporate it into the parameters $\overline{y}$.}
\begin{enumerate}
\item  Let $\uf_\vp(A)$ be the set of ultrafilters $\de$ on ${^{\lgn(\bar{y})}A}$ such that \\ if $\overline{\mathbf b} = \langle \overline{b}_s : s \in I \rangle \in \ob(\de)$, then
\[ \{ \vp(\overline{x},\overline{b}_s, \overline{c}) : s \in I \}      \]
is a consistent partial type.
\item For each $k<\omega$, let $\uf_{\vp, k}(A)$ be the set of ultrafilters $\de$ on ${^{\lgn(y)}A}$ such that 
%if $\overline{\mathbf b} = \langle \overline{b}_s : s \in I \rangle$ is an 
%infinite indiscernible sequence based on $\de$ 
if $\overline{\mathbf b} = \langle \overline{b}_s : s \in I \rangle \in \ob(\de)$
and $s_0 <_I \cdots <_I s_{k-1}$, then 
\[ \{ \vp(\overline{x},\overline{b}_{s_\ell}, \overline{c}) : \ell < k  \}      \]
is a consistent partial type. So when $k=\infty$, we may omit it. 
\end{enumerate} 
\end{defn}

%We will aim to compare the behavior of formulas along sequences determined by ultrafilters. Note that given a formula $\vp(\bar{x})$, 
%$\ell(\bar{x}) = m$,  
%in order to apply \ref{higher-k} we need to choose a partition of the variables. 

We arrive at a key definition of the section: higher formulas $(\vp, A, \de)$, triples such that $\vp$ is indeed consistent when instantiated along 
any $\bar{b} \in \ob(\de, A)$. Two subsequent theorems of the section will 
characterize $SOP_2$ and $SOP_3$ in terms of the interaction of these higher formulas. 

\begin{defn}[Higher formulas] Let $\bar{m} = (m_0, m_1)$. Writing $m$ instead of $\bar{m}$ means $m_0 = 1$ and $m = m_1$, 
or that $m_0$ is clear from the context. %, see \ref{d:811}. 
\label{a7}
Let $\HF^{\bar{m}}_k$ be the set of triples $\rho = (\varphi,A,\de)$
   where $\varphi = \varphi(\bar x,\bar y)$ with $\ell(\bar{x}) = m_0$, $\ell(\bar{y}) = m_1$, and no more parameters, 
   $\de$ is an ultrafilter on ${}^{m_1} A$ and
   \[ \de \in \uf_{\varphi,k}(A).\] 
\end{defn}

\begin{defn} In the context of $\ref{a7}$,   
\begin{enumerate}
\item Let $\HF^{\bar{m}} = \HF^{\bar{m}}_\infty = \bigcap_{k} \HF^{\bar{m}}_k$. 

\item We may wish to consider higher formulas over a fixed set $A$, or a using a fixed formula $\vp$, in which case 
our notation will be: 
\begin{enumerate}
\item Given $A$, we may write $\HF^{\bar{m}}_k(A)$, or ``$(\vp, \de) \in \HF^{\bar{m}}_k(A)$''.

\item Let $\HF^{\bar{m}}(A) = \HF^{\bar{m}}_\infty(A) = \bigcap_{k} \HF^{\bar{m}}_k(A)$.

\item Given $\varphi = \vp(\bar{x}, \bar{y})$, with $\ell(\bar{x}) = m_0$ and $\ell(\bar{y}) = m_1$, 
we may write $\HF^\varphi = \HF^{\bar{m}}_\varphi$, $\HF^{\bar{m}}_{\varphi,k}(A)$, etc, where the subscript $\vp$ 
means we restrict to triples whose first element is $\vp$ with the given partition of variables. 

\item Call the elements of $\HF_\vp(A)$ ``higher $\vp$-formulas over $A$.''
\end{enumerate}
\end{enumerate}
\end{defn}

\begin{conv}
In Defintion \ref{a7}$(1)$, we may say ``$\de$ is an ultrafilter \emph{over} $A$'' without mentioning $m_1$ when it is clear from context. 
\end{conv}

We would like to study pairwise consistency or inconsistency of higher formulas as follows. Suppose we are given $(\vp_0, A_0, \de_0)$, $(\vp_1, A_1, \de_1)$, 
$\langle b_{0,s} : s \in I_0 \rangle \in \ob(A_0, \de_0)$ and $\langle b_{0,t} : t \in I_1 \rangle \in \ob(A_1, \de_1)$. If we choose $s \in I_0$ and $t \in I_1$, 
will $\vp_0(\bar{x},\bar{b}_{0,s})$ and $\vp_1(\bar{x},\bar{b}_{1,t})$ be consistent?  What if we choose finitely many instances from each list? 
The specter of $SOP_3$ suggests that we should first fix an interpolation of $I_0$ and $I_1$ into a single linear order $I$
and pay attention to the relative position of the indices $s$ and $t$. 
The notation we now introduce in $\ref{d:k-p-sigma}$-$\ref{d:mixed}$ is one way to handle this (most of the time we use $\sigma = 2$).

\begin{defn} \emph{(Partitions of linear orders, [Sh:950] Definition 1.39 )} \label{d:k-p-sigma}
\\ Let $K_{p,\sigma}$ be the class of triples 
\[ (I, <_I, (P^I_i)_{i<\sigma} ) \]
where $I$ is linearly ordered by $<_I$ and $\langle P^I_i : i < \sigma \rangle$ is a partition of $I$. 
\end{defn}

%With a partition in hand, there is a natural notion of mixed objects; introducing this notation will simplify the 
%definition of inconsistent higher formulas later on. 

\newpage

\begin{defn} %\emph{(Mixed objects, $\ob(\overline{\de}, \overline{A}$))} 
\label{d:mixed} 
Fix $\sigma$ and suppose $\langle (\vp_i, A_i, D_i) : i < \sigma \rangle$ is a sequence of higher formulas. 
Let $\bar{A} = \langle A_i : i < \sigma \rangle$ and let $\bar{D} = \langle D_i : i < \sigma \rangle$.  
Define
\[ \ob(\bar D) = \ob(\bar D,\bar A) \] 
to be the set of $\bar{\mathbf b} = \langle \bar b_s:s \in I\rangle$ such that: 
\begin{enumerate}
\item $I \in K_{p,\sigma}$ 
\item if $s \in P^I_i$ then:
\begin{itemize}
\item $\ell g(\bar b_s) = \mathbf m(D_i)$
\item $\tp(\bar b_s,\bigcup\limits_{j<\sigma} A_j \cup
\bar{\mathbf b}_{>s}) \subseteq \Av(D_i) := \Av(D_i,\mathfrak{C}_\tau)$. 
\end{itemize} 
\end{enumerate} 
We may write $A$ instead of $\bar{A}$ when all the $A_i$ are the same. 
We may write $i = \mathbf{i}(s) = \mathbf{i}(s, I)$ and we may write $I_{\bar{\mathbf b}} = I[\bar{\mathbf b}]$ for $I$. 
\end{defn}

Note that in Definition \ref{d:mixed}, we do not require that $\bar{\mathbf{b}}$ be indiscernible; in fact, 
it may consist of sequences of differing lengths, if the $\mathbf{m}(D_i)$ differ.

\begin{defn}[$n$-inconsistent higher formulas] \label{d:inc}
Assume that for $\ell = 0, 1$, 
\[ \rho_\ell = (\varphi_\ell,A_\ell,D_\ell) \in \HF^{\bar{m_\ell}}_{k_\ell}. \]  
\begin{enumerate}
\item We say $(\varphi_0,A_0,D_0)$ is
   $n$-contradictory to $(\varphi_1,A_1,D_1)$ when: \\ for every $\bar{\mathbf b}
   \in \ob(\langle D_0,D_1\rangle,\langle A_0,A_1\rangle)$ and every $s_0 <
\ldots < s_{2n-1}$ with $s_\ell \in P^I_0{[\bar{\mathbf b}]}$ for
$\ell < n$ and $s_\ell \in P^I_1[\bar{\mathbf b]}$ for $\ell \in
   [n,2n)$,  we have that 
   \[\{\varphi_0(\bar x,\bar b_{s_\ell}):\ell <
   n\} \cup \{\varphi_1(\bar x,\bar b_{s_\ell}):\ell \in [n,2n)\} \]
   is contradictory.

\item In ``$n$-contradictory,'' if $n=1$ we may omit it and writing $n = \infty$ means ``for some $n$".
Of course, ``$n$-consistent'' is the negation. 
\item We say that $(\varphi_0,A_0,D_0)$ and $(\varphi_1,A_1,D_1)$ are
\emph{mutually $n$-contradictory}
when $(\varphi_\ell,A_\ell,D_\ell)$ is $n$-contradictory 
to $(\varphi_{1-\ell},A_{1-\ell},D_{1-\ell})$ for $\ell=0,1$.  On the symmetry of this notion, 
see Lemma \ref{sym-lemma} below. 

\item ``The set of $\{ \rho_i = (\varphi_i,A_i,D_i) : i \in S \}$ are pairwise $n$-contradictory'' will mean that each pair is mutually $n$-contradictory. 
\end{enumerate}
\end{defn}

\begin{disc}
\emph{ }
\begin{enumerate}
\item 
In $\ref{d:inc}(1)$ we could have allowed the choice of elements from the 
two partitions to alternate. 
However, we will see this is immaterial for $NSOP_3$, and moreover is a little less natural when replacing $n$-contradictory higher formulas 
by $1$-contradictory derived formulas arising as $n$-fold conjunctions. 

\item 
This definition enforces an order between the elements $s_i$ in different partitions; 
one could give a different definition requiring only that $s_0 < \dots < s_{n-1}$ and $s_n < \dots < s_{2n-1}$. 
%This would be equivalent to the current definition if the partitions alternate infinitely often.
%
%\item  
%We know by the definition of $\ob$ that instantiating $\vp_\ell$
%along the sub-sequence indexed by $s \in P_\ell$ is consistent, for $\ell = 0, 1$. The definition of contradictory points out when 
%merging these two lists of instances becomes uniformly inconsistent.  
%The point is that for $NSOP_3$ or $NSOP_2$, the picture is nicer. 
\end{enumerate}
\end{disc}

\begin{defn} \label{d:811}
Say that $T$ has \emph{symmetric inconsistency} if, fixing $m = \ell(\bar{x})$, for any\footnote{The notation $\vp_\ell(\bar{x}_{[m]}, \bar{y}_{[m_\ell]})$ 
means that $\ell(\bar{x}) = m$ and $\ell(\bar{y}) = m_\ell$.} $m_1, m_2 <\omega$, 
$\vp_\ell = \vp_\ell(\bar{x}_{[m]}, \bar{y}_{[m_\ell]})$ and $(\vp_\ell, A_\ell, \de_\ell) \in \HF^{m_\ell}$ for $\ell = 1,2$,
we have that $(\vp_1, A_1, \de_1)$ is $\infty$-contradictory to $(\vp_2, A_2, \de_2)$ iff 
 $(\vp_2, A_2, \de_2)$ is $\infty$-contradictory to  $(\vp_1, A_1, \de_1)$.  
\end{defn}

As desired, this definition picks up on the asymmetry of $SOP_3$:

\begin{claim} \label{o:sym}
If $T$ has $SOP_3$, then $1$-inconsistency is not symmetric, i.e. $T$ has symmetric inconsistency in the sense of $\ref{d:811}$.  
\end{claim} 

\begin{proof}
Let the sequence $\langle \bar{a}_n : n < \omega \rangle$ and the formulas $\vp(\bar{x}, \bar{y})$, $\psi(\bar{x}, \bar{y})$ witness $SOP_3$, 
see \ref{d:sop3}.  
Let $A = \{ \bar{a}_n : n < \omega \}$. Let $\de$ be an ultrafilter on ${^{\lgn(\bar{y})}A}$ such that $\{ \bar{a}_n : n > k \} \in \de$ for every $k<\omega$. 
Let $\rho_1 = (\vp, A, \de)$, $\rho_2 = (\psi, A, \de)$. To see that $\rho_2$ is $1$-consistent with $\rho_1$ but $\rho_1$ is not $1$-consistent with 
$\rho_2$, let $I$ be the linear order $\omega + \omega^*$. Let $(P_0, P_1)$ be any partition of $I$ into 
two infinite sets, and let $\langle \bar{b}_s : s \in I \rangle$ be as in \ref{d:mixed} for $((\de, \de), (A, A))$. 
\end{proof}

%
%First, since $T$ has $SOP_3$, we can find an indiscernible sequence $\langle \overline{a}_i : i < \omega \rangle$ and formulas $\vp(\overline{x}, \overline{y})$, 
%$\psi(\overline{x}, \overline{y})$ satisfying \ref{d:sop3}. 
%
% such that:
%\begin{enumerate}
%\item $\{ \vp(\overline{x}, \overline{y}), \psi(\overline{x}, \overline{y}) \}$ is contradictory.
%\item for each $k<\omega$, the following is a consistent partial type:
%\[  \{ \psi(\overline{x}, \overline{a}_j) : j \leq k \} \cup \{ \psi(\overline{x}, \overline{a}_i) : i > k \} \]
%\item for $j<i$, the set $\{ \vp(\overline{x}, \overline{a}_i), \psi(\overline{x}, \overline{a}_j) \}$ is contradictory.

\begin{lemma}[Symmetry lemma] \label{sym-lemma}
\emph{ }
%\begin{enumerate}
%\item 
For $T$ complete the following are equivalent.
\begin{enumerate}
\item $T$ is $SOP_3$.
\item $T$ has symmetric inconsistency. 
\end{enumerate}
%\end{enumerate}
\end{lemma}

\begin{proof}
The direction (1) implies (2) is given by Claim \ref{o:sym}. 

For the other direction, we will show how an instance of symmetric inconsistency gives rise to $SOP_3$.  
For notational simplicity, we use $m=\ell(\bar{x}) = 1$, since for $SOP_3$ the arity does not matter. 
%For simplicity, let $\vp_\ell = \vp_\ell(x,\overline{y}_\ell)$, i.e. for transparency we let $\bar{x} = x$.  
Suppose we are given $m_0, m_1<\omega$ and $(\vp_\ell, A_\ell, \de_\ell) \in \HF^{m_\ell}$ for $\ell = 0, 1$. 
  
Let %\footnote{Recall that $A_0, A_1$ are sets, and the ultrafilters $D_0$, $D_1$ are ultrafilters on tuples of elements of $A_0$, $A_1$ respectively.} 
$A = A_0 \cup A_1$, $m = m_0 + m_1$ and 
\[ D = D_0 \times D_1 = \{X \subseteq {^m A} ~:~ \{\bar a_0 \in {}^{m_0}A_0:\{\bar a_1 \in 
{}^{m_1}(A_1): {\bar a_0}{~ ^\smallfrown} \bar a_1 \in X\} \in D_1\} \in
D_0\}. \] 
By construction, $D$ is an ultrafilter on ${}^m A$. Suppose $I$ is a linear order and
%\begin{equation}
\[ \bar{\mathbf{a}} = \langle \bar{a}_s = {{{\overline{a}_{0,s} }}} ^\smallfrown\bar a_{1,s}:s \in I\rangle \]
%\end{equation}
is an indiscernible sequence based on $D$, with $\ell(a_{\ell,s}) = m_\ell$,  
thus $\langle \overline{a}_{\ell,s}:s \in I\rangle$ is also an indiscernible sequence based on $D_\ell$ for $\ell = 0, 1$.
For any $h: I \rightarrow \{ 0, 1 \}$, we may consider the partition given by 
%\begin{equation}
\[ I_h = (I, <_I, (P^h_i)_{i < 2}) \mbox{ where } P^h_i =\{s:h(s)=i\} \]
%\end{equation}
Then $I_h \in K_{p, \{ 0, 1\}}$ and $\bar{\mathbf a}_h := \langle a_{h(s),s}:s \in I\rangle$ is $I_h$-indiscernible based on $(D_0, D_1)$. 

\begin{figure}[ht!]
  \centering
    \includegraphics[width=0.5\textwidth]{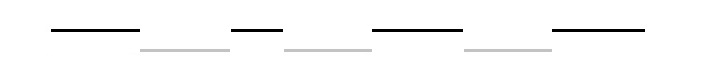}
\caption{\footnotesize{A sample partition of $I$. $I$ indexes a sequence of $m_0+m_1$-tuples which is indiscernible based on $D$. 
Thus, restricting to the first $m_0$ elements 
of tuples with indices in the black regions gives an indiscernible sequence based on $D_0$, whereas restricting to the last $m_1$ elements 
of tuples with indices in the grey regions gives an indiscernible sequence based on $D_1$.}}
\end{figure}
Moreover, if $\langle s_\alpha:\alpha < \omega + \omega\rangle$ is $<_I$-increasing, then:
\begin{enumerate}
\item[(1)$_n$] the following are equivalent:
\begin{enumerate}
\item[(a)] $\{\varphi_0(x,\bar{a}_{0,s_\alpha}) : \alpha < n\} \cup  \{\varphi_1(x,\bar{a}_{1,s_{\omega + \alpha}}) : \alpha < n\}$ is contradictory
\item[(b)] $(\varphi_0,A_0,D_0)$ is $n$-contradictory to $(\varphi_1,A_1,D_1)$ 
\end{enumerate}

\begin{figure}[ht!]
  \centering
    \includegraphics[width=0.5\textwidth]{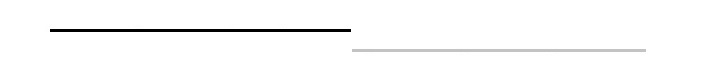}
\caption{\footnotesize{Let the image indicate that we instantiate $\vp_0$ $n$ times along the $D_0$-indiscernible sequence indexed by the 
black region and $\vp_1$ $n$ times along the $D_1$-indiscernible sequence indexed by the grey region.}}
\end{figure}

\item[(2)$_n$] the following are equivalent:
\begin{enumerate}
\item[(d)] $\{\varphi_1(x, \bar{a}_{1,s_\alpha}) : \alpha < n\} \cup  \{\varphi_0(x, \bar{a}_{0,s_{\omega + \alpha}}) : \alpha < n\}$ is contradictory
\item[(e)] $(\varphi_1,A_1,D_1)$ is $n$-contradictory to $(\varphi_0,A_0,D_0)$ 
\end{enumerate}

\begin{figure}[ht!]
  \centering
    \includegraphics[width=0.5\textwidth]{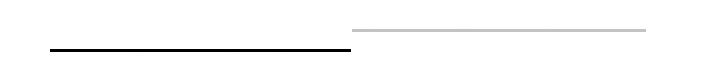}
\caption{\footnotesize{Let the image indicate that we instantiate $\vp_1$ $n$ times along the $D_1$-indiscernible sequence indexed by the 
grey region and $\vp_0$ $n$ times along the $D_0$-indiscernible sequence indexed by the black region.   Then for any given $n$, 
a disparity in consistency between Figures 2 and 3 leads, by taking conjunctions, to an instance of $SOP_3$.}}
\end{figure}

\end{enumerate}

Let $\bar{y} = {\bar{y}_0}^\smallfrown{\bar{y}_1}$, $\varphi^+_0 = \varphi^+_0(x,\bar{y}) = \varphi^+_0(x, \bar{y}_0, \bar{y}_1) =
\varphi_0(x, \bar{y}_0)$, and $\varphi^+_1 = \varphi^+_1(x,\bar{y}) = \varphi^+_1(x, \bar{y}_0, \bar{y}_1) =
\varphi_1(x, \bar{y}_1)$, i.e. these are the given formulas formally considered as having more variables. 
Then for each $n$, to (1)$_n$ above, we may add the equivalent condition: 
\begin{enumerate}
\item[(c)$_n$]$\{\varphi^+_0(x,\bar a_{s_\alpha}) : \alpha < n\} \cup \{\varphi^+_1(x,\bar a_{s_{\omega +\alpha}}) : \alpha < n\}$ is contradictory 
\end{enumerate}
Likewise, for each $n$, to (2)$_n$ above, we may add the equivalent condition: 
\begin{enumerate}
\item[(f)$_n$]$\{\varphi^+_1(x,\bar a_{s_\alpha}) : \alpha < n\} \cup \{\varphi^+_0(x,\bar a_{s_{\omega +\alpha}}) : \alpha < n\}$ is contradictory 
\end{enumerate}

Now if it is not the case that for all $n$ (c)$_n$ iff for all $n$ (f)$_n$, we have a witness to $SOP_3$ for $T$ (given by the conjunctions of 
$n$ copies of $\vp_0$ and of $\vp_1$, respectively). This completes the proof. 
\end{proof}

As a corollary of the proof of Lemma \ref{sym-lemma}, we have: 

\begin{cor}
Assume $T$ is $NSOP_3$. If $p_\ell = (\vp_\ell, A_\ell, \de_\ell)$ for $\ell = 0, 1$ and $p_0$ is $n$-consistent with $p_1$ for every $n$ $($i.e., not $\infty$-contradictory$)$, 
then %all orders are ok, i.e 
if $I \in K^P_\sigma$, $\sigma = 2$ and we have $\langle \bar{b}_s : s \in I \rangle$ as usual, then 
$\{ \vp_\ell(x,\bar{b}_s) : s \in P^I_\ell, \ell < 2 \}$ is consistent. 
\end{cor}

We now work towards Theorem \ref{a13}, using higher formulas to characterize $NSOP_2$. 
First, we show that having $SOP_2$ means many pairwise $1$-contradictory higher formulas. Recall that: 
%$\Ded^+(\lambda)$ is the smallest cardinal $\kappa$ such that no tree with $\lambda$ nodes has 
%$\geq \kappa$ branches. 

\begin{defn}
\label{a23} 
$\Ded^+(\lambda) =  \sup\{|\lim(\cT)|^+:\cT \subseteq {}^{\lambda >}2$ is nonempty, closed under initial segments and has no
$\triangleleft$-maximal members and has cardinality $\le \lambda \}$. 
\end{defn}

\begin{claim} \label{c:converse}
Assume $\varphi(\bar x,\bar y)$ has $\SOP_2$ in $T$ and $\lambda < \mu
< \Ded(\lambda)$.  Then there is $A \subseteq \gC_T$ of cardinality
$\lambda$ and $D_\alpha \in \uf_\varphi(A)$ for $\alpha < \mu$ such
that $\langle (\varphi, D_\alpha):\alpha < \mu)$ are pairwise 1-contradictory.
\end{claim}

\begin{proof}{\label{a21}}
Let the tree $\cT$ witness that $\mu < \Ded(\lambda)$ and let $\nu_\alpha \in
\lim(\cT)$ (i.e. the ``leaves'') for $\alpha < \mu$ be pairwise distinct.  Let $\langle \bar
a_\eta:\eta \in \cT\rangle$ be such that:
   
\begin{enumerate}
\item[$(a)$] $(\varphi(\bar x,\bar a_\eta),\varphi(\bar x,\bar
  a_\nu)$ are contradictory when $\eta \perp \nu$ are from $\cT$
 
\item[$(b)$] $\{\varphi(x,\bar a_{\eta \rstr \alpha}):\alpha \le \ell
  g(\eta)\}$ is a consistent partial type for $\eta \in \cT$.
\end{enumerate}
   
Let $A = \cup\{\bar a_\eta:\eta \in \cT\}$ and for $\alpha <\mu$ let
$D_\alpha$ be an ultrafilter on ${}^{\ell g(\bar y)}A$ concentrating on the  
branch $\nu_\alpha$, i.e. such that:
\[ \mbox{if } \beta < \ell g(\nu_\alpha),\alpha < \mu \mbox{ then }
  \{\bar a_\rho:\nu_\alpha \rstr \beta \trianglelefteq \rho
  \triangleleft \nu_\alpha\} \in D_\alpha \] 
Clearly these ultrafilters are as desired. 
\end{proof}

We will need notation for finitary approximations to $SOP_2$-trees. 

\begin{defn} \label{nsop2n}
We say that $(T,\varphi)$ has $\NSOP_{2,n}$ when
there are no $\bar b_\eta \in {}^{\ell g(\bar y)}\gC$ for $\eta \in {}^{n>}2$ such that
\begin{enumerate}
\item $\eta \perp \nu \Rightarrow \varphi(\bar x,\bar
  b_\eta),\varphi(\bar x,\bar b_\nu)$ are incompatible
\item  for $\eta \in {}^n 2,\{\varphi(\bar x,\bar
  b_{\eta \rstr \ell}):\ell < n\}$ is a type.
\end{enumerate}
\end{defn}

\begin{fact} \label{approx}
$(T, \vp) \models NSOP_2$ iff $\bigvee_n \left( ~(T,\vp) \models NSOP_{2,n}  \right) $.
\end{fact}

We now arrive at the second theorem of the section, which shows how from many pairwise contradictory higher formulas 
we may build an $SOP_2$-tree, complementing \ref{c:converse}. %, i.e. an inverse to \ref{c:converse}. 
Recall that by the Symmetry Lemma \ref{sym-lemma} above, as $NSOP_2$ implies $NSOP_3$, 
being contradictory is a symmetric notion.  

\begin{theorem}  
\label{a13}
For a theory $T$ the following are equivalent:
\begin{enumerate}
\item For every infinite $A$ and formula $\vp$, there are no more than $|A|$ pairwise 1-contradictory higher $\vp$-formulas over $A$.  
%For any infinite $A$ and $\varphi =\varphi(\bar x,\bar y_m) \in \HF^m_\varphi(A)$ we cannot find $|A|^+$ pairwise 1-contradictory members.
\item $T$ has $NSOP_2$. 
\end{enumerate}
\end{theorem}

\begin{proof}
In Claim \ref{c:converse}, it was shown that $SOP_2$ implies many pairwise $1$-contradictory higher formulas. 
So it remains to prove the other direction: many pairwise $1$-contradictory higher formulas imply $SOP_2$. 

\step{Step 0: Setup.}
Let $\lambda = |A|^+$ and by Fact \ref{approx}, let $n$ be such that $(T, \vp)$ has $NSOP_{2,n}$. 

Assume for a contradiction that: 
\begin{enumerate}
\item $(\varphi,D_\alpha) \in \HF^m_\varphi(A)$ for $\alpha < \lambda$ are pairwise 1-contradictory
\footnote{Here ${\HF}^m_{\varphi,n}(A)$ suffices.} and
\item fixing some infinite linear order $I$, let $\bar{\mathbf b}_\alpha = \langle b_{\alpha,s}:s \in I\rangle \in \ob(D_\alpha)$ for $\alpha < \lambda$. 
\end{enumerate}

(Recall that the definition of $1$-contradictory is for any such $\bar{\mathbf b}$.) We will use just that $\lambda = \cf(\lambda) > |A| \geq \aleph_0$. 

\br
\step{Step 1: Approximations.} We define the set $\AP$ of approximations (to a full $SOP_{2,n}$-tree) 
to be the set of $\bx$ consisting of: 
\footnote{Alternately, we could consider: in $\Lambda_x, \frt(\mathbf x)$ is a
  set of pairwise incomparable elements, e.g. the $\tlf$-maximal $\eta \in \Lambda$ of length $< n-1$. 
  This is simpler here, but then the induction step would require two steps: 
  add $\eta^\smallfrown\langle 0 \rangle$ or $\eta^\smallfrown\langle 1
  \rangle$ for some $\eta \in \frt(\mathbf x)$.}

\begin{enumerate}
\item[(1.1)] $\Lambda \subseteq {}^{n>}2$, $\bar a$ where:
\begin{enumerate}
\item $\Lambda$ is non-empty downward closed
\item $\bar a = \langle \bar a_\eta:\eta \in  \Lambda\rangle$ with each $\bar a_\eta \in {^{\lgn(\bar y)}A}$
\item if $\eta \perp \nu$ are from $\Lambda$ then $\varphi(\bar x,\bar a_\eta),\varphi(\bar x, \bar a_\nu)$ are incompatible
\item if $\eta \in \Lambda$ then $\{\varphi(\bar x,\bar a_{\eta \rstr \ell}):\ell \le \ell g(\eta)\}$ is a consistent partial type 
\item if $\nu^\smallfrown\langle 1 \rangle \in \Lambda$ then $\nu^\smallfrown\langle 0 \rangle \in \Lambda$. 
\end{enumerate}
\br
\item[(1.2)] $\bar{\cU} = \langle \cU_\eta:\eta \in \frt(\Lambda)\rangle$, where:
\begin{enumerate}
\item $\frt(\Lambda) := \{\eta \in \Lambda:\ell g(\eta) < n-1$ and $\eta \char 94 \langle 0 \rangle \notin \Lambda$ or
$\eta \char 94 \langle 1 \rangle \notin \Lambda\}$ 
\\ (\emph{the ``frontier'' for our inductive construction of a tree, i.e., the nodes without two immediate successors})
\item each $\cU_\eta \subseteq \lambda$ has %\footnote{When using this proof to check Claim \ref{b18}, add here: $\rk(\cU_\eta) = \infty$} 
 cardinality $\lambda$  %%%xxx
\\ (the intention is a set of indices for the $\bar{\mathbf{b}}_\alpha$ from Step 0)
\item $\cU_\eta \cap \cU_\nu = \emptyset$ for $\eta \ne \nu$ 
\item if $\eta \trianglelefteq \nu$ are from $\frt(\Lambda)$, $k < \omega$, 
$s_0,\dotsc,s_{k-1} \in I,k + \ell g(\eta) \le n$ and $\alpha \in \cU_\nu$ then 
\[ \{\varphi(x,\bar a_{\eta \rstr \ell}):\ell \le \ell g(\eta)\} \cup\{\varphi(x,\bar b_{\alpha,s_\ell}):\ell < k\} \]
is a consistent partial type. 
\item if $\eta \in \frt(\mathbf x),s \in I,\alpha \in \cU_\eta$ and $\nu \in \Lambda,\neg(\nu \trianglelefteq \eta)$, 
\\ then $\varphi(\bar x,\bar b_{\alpha,s}),\varphi(\bar x,a_\nu)$ are incompatible.
\end{enumerate}
\end{enumerate}

\br

Note the role of the two kinds of parameters: the ${\bar b}_{\alpha, s}$ from Step 0, and the parameters ${\bar a}_\eta$ for the tree. 
Informally, the $\uu_\eta$ tell us in which sequences we can expect to continue our consistent partial type while maintaining inconsistency 
elsewhere. 

\br

We define a two-place relation $\le_{\AP}$ on $\AP$ in the natural way: 
$\mathbf x \le_{\AP} \mathbf y$ iff
\begin{enumerate}
\item[$(a)$]  $\mathbf x,\mathbf y \in \AP$
\item[$(b)$]  $\Lambda_{\mathbf x} \subseteq \Lambda_{\mathbf y}$
\item[$(c)$]  $\bar a_{\mathbf x,\eta} = \bar a_{\mathbf y,\eta}$ for
  $\eta \in \Lambda_{\mathbf x}$
\item[$(d)$]  if $\eta \in \frt(\Lambda_{\mathbf x}) \cap
  \frt(\Lambda_{\mathbf y})$ then $\cU_{\mathbf y,\eta} \subseteq
  \cU_{\mathbf x,\eta}$
\item[$(e)$]   if $\eta \in \frt(\Lambda_{\mathbf x}),\nu \in
  \frt(\Lambda_{\mathbf y})$ and $\eta \triangleleft \nu$ but $\nu \rstr
  (\ell g(\eta) +1) \notin \Lambda_{\mathbf x}$,  
  \\ then $\cU_{\mathbf y,\nu} \subseteq \cU_{\mathbf x,\eta}$ .
\end{enumerate}

\br
\step{Step 2: Strategy.} Clearly $\le_{\AP}$ is a partial order on $\AP$.
By choice of $n$ in Step 0, if $\mathbf x \in \AP$ then $\Lambda \subsetneq {}^{n>}2$.
Thus, to obtain a contradiction (and complete the proof) it will suffice to show that: 

\begin{enumerate}
\item[(2.1)] there is $\mathbf z \in \AP$ with $\Lambda_{\mathbf z} = \{<>\}$, i.e. $\AP \neq \emptyset$. 
\item[(2.2)] if $\mathbf x \in \AP$ then there is $\mathbf y \in
  \AP$ such that $|\Lambda_{\mathbf x}| < |\Lambda_{\mathbf y}|$, in fact,
  $\mathbf x <_{\AP} \mathbf y$.
\end{enumerate}

\br
\step{Step 3: Verifying Condition \emph{(2.1)}.}
Let $s_0 <_I \ldots <_I s_n$.  For each $\alpha < \lambda$
  clearly $\gC \models (\exists \bar x) \bigwedge\limits_{\ell \le n}
  \varphi(\bar x,\bar b_{\alpha,s_\ell})$.  So by (1) and (2) of Step 0, 
\[
\cx_\alpha = \{\bar a \in {}^{\ell g(\bar y)}A:\gC \models (\exists
\bar x)[\varphi(\bar x,\bar a) \wedge \bigwedge\limits^{n}_{\ell=1}
  \varphi(x,\bar b_{\alpha,s_\ell})]\} \in D_\alpha.
\]
For each $\alpha < \lambda$, choose $\bar b_\alpha \in \cx_\alpha$.  [So $\bar b_\alpha$ is ``canonically consistent'' with the partial $\vp$-type 
given along the sequence $\bar{\mathbf{b}}_\alpha$ by any $n$ members.] As $|A| < \lambda$, the set
${^{\ell g(\bar y)}A}$ has cardinality $< \lambda = \cf(\lambda)$ so for some $\bar b$ the set 
\[ \cU = \{\alpha < \lambda ~:~ \bar b_\alpha = \bar b\} \] 
has cardinality $\lambda$. 

Now define $\mathbf z$ by:
\begin{enumerate}
\item[$\bullet$]  $\Lambda_{\mathbf z} = \{<>\}$
\item[$\bullet$]  $\bar a_{<>} = \bar b$
\item[$\bullet$]  $\cU_{<>} = \cU$.
\end{enumerate}
and $\mathbf z$ is as required.

\br
\step{Step 4: Proving Condition \emph{(2.2)}.} Let $\Lambda = \Lambda_{\mathbf x}$.

The situation at the inductive step is essentially as follows. We would like the tree to become full, 
so we choose $\varrho$ which is minimal for the property of not having two successors; say, $\varrho^\smallfrown\ii$ is missing (there are minor adjustments at the end 
of this step depending on whether it has a successor at 0 or 1). We try to find a corresponding $\overline{a}_{\varrho^\smallfrown \ii}$ for the missing successor of this node,  
and its set of compatible indices $\uu_{\varrho^\smallfrown \ii}$, subject to the following constraints: 

\begin{enumerate}
\item[(i)] consistency of $\bar{a}_{\varrho^\smallfrown \ii}$ with comparable nodes, (1.1)(d)
\item[(ii)] inconsistency of $\bar{a}_{\varrho^\smallfrown \ii}$ with incomparable nodes, (1.1)(c)
\item[(iii)] consistency of $\bar{a}_{\varrho^\smallfrown \ii}$ with large subsets of $\uu_\nu$, for compatible $\nu$ for (1.2)(d)
\item[(iv)] inconsistency of $\bar{a}_{\varrho^\smallfrown \ii}$ with $\uu_\eta$, for incompatible $\eta$, for (1.2)(e) 
\item[(v)] disjointness of $\uu_{\varrho^\smallfrown \ii}$ from large subsets of $\uu_\nu$, for incompatible $\nu$, for (1.2)(c)
\end{enumerate}

In this informal explanation, ``large'' stands in for the fact that we will also have to refine the other $\uu_\eta$ 
to get actual inconsistency or an actual empty intersection. This completes the description of intent. 

\br

Now to begin, choose $\varrho \in \frt(\Lambda)$ be of minimal %\footnote{was: maximal} 
length, hence necessarily $\ell g(\varrho) < n-1$ and we can choose $\iota < 2$ such
that $\varrho \char 94 \langle \iota \rangle \notin \Lambda$.  Let $k = n - \lgn(\varrho)$. 

Fix for a while a choice of a distinguished element of $\uu_\eta$ for each $\eta \in \frt(\Lambda)$, that is, 
$\bar \alpha = \langle \alpha_\eta:\eta \in \frt(\Lambda)\rangle$ where  
$\alpha_\eta \in \cU_\eta$, and let $s_0 <_I \ldots <_I s_{k-1}$.

By our choice of $\varrho$ and (1.2)(c), for every\footnote{If $\rho$ has precisely one successor in the tree,  
this successor may be from $\frt(\Lambda_{\mathbf{x}})$.}
$\eta \in \frt(\Lambda_{\mathbf x}) \backslash \{\varrho\}$ 
%which is not 
%$\tlf$-comparable to $\varrho$ 
we have that $\alpha_\eta \neq \alpha_\rho$, hence 
the following set belongs to
$D_{\alpha_\varrho}:$
\begin{align*}
 \cx_{\bar\alpha,\eta} =& \cx^1_{\bar\alpha,\eta} \cup \cx^2_{\bar\alpha,\eta} \mbox {where } \\
 \cx^1_{\bar\alpha,\eta} =& \{\bar b \in {}^{\lgn(\bar y)}A \colon: \vp(\bar{x}, \bar{b}), \vp(\bar{x}, \bar{a}_\nu) \mbox{ are incompatible, where }
 \nu \in \Lambda \land \neg(\nu \trianglelefteq \rho) \} \\
 \cx^2_{\bar\alpha,\eta} =& \{\bar b \in {}^{\ell g(\bar y)}A:\varphi(\bar x,\bar b),\varphi(\bar x,\bar b_{\alpha_\eta,s_0}) \mbox{ are incompatible}\} 
\end{align*}
since the $(\vp, \de_\alpha)$ are pairwise $1$-contradictory and as $\langle
\alpha_\ell:\ell \in \frt(\mathbf x)\rangle$ is without repetitions by the assumption (1.2)(c) for $\cx^2_{\bar\alpha,\eta}$ 
and (1.2)(e) for $\cx^1_{\bar\alpha,\eta}$. 

Furthermore, as (1.2)(d) holds for $\mathbf x$, the following set belongs to $D_{\alpha_\varrho}$:
\begin{align*} 
\cZ_{\bar\alpha,\varrho} := \{\bar b \in {}^{\ell g(\bar y)}A : &\{\varphi(\bar x,\bar a_{\eta \rstr \ell}):\ell \le \ell g(\varrho)\} \\
& \cup \{\varphi(\bar x,\bar b)\}  \\
& \cup \{\varphi(\bar x,\bar b_{\alpha_\varrho,s_\ell}):\ell < n-\ell g(\varrho)-1\} \mbox{ is a type} \} 
\end{align*}

Let $\cX_{\bar\alpha} = \cap\{\cX_{\bar\alpha,\eta}:\eta \in
\frt(\Lambda) \backslash \{\varrho\} \mbox{ and not comparable to $\varrho$ }\} \cap \cZ_{\bar\alpha,\varrho}$. Since 
$\Lambda$ is finite and $D_{\alpha_\varrho}$ is a filter, $\cX_{\bar\alpha} \in D_{\alpha_\varrho}$.

By the choice of our sequence of ultrafilters in Step 0, for every $\beta \in \cU_{\mathbf x,\varrho} \backslash
\{\alpha_\varrho\}$, we know that
$(\varphi,D_{\alpha_\varrho}),(\varphi,D_\beta)$ are 1-contradictory. So 
\[ \cY_{\bar\alpha,\beta} = \{\bar b \in {}^{\ell g(\bar y)}A : \varphi(x,\bar b),\varphi(x,\bar b_{\beta,s_0}) \mbox{ are incompatible} \} \] 
belongs to $D_{\alpha_\varrho}$. So we may choose
$\bar b_{\bar\alpha,\beta} \in {}^{\ell g(\bar y)}A$ which belongs to
$\cx_{\bar\alpha} \cap \cY_{\bar\alpha,\beta}$.  As $|\cU_{\mathbf x,\varrho}| =
\lambda = \cf(\lambda) > |A|$, there is $\bar b_{\bar\alpha} \in 
{}^{\ell g(\bar y)}A$ such that 
\[ \cW_{\bar\alpha} = \{\beta \in \cU_{\mathbf x,\varrho}:\bar b_{\bar\alpha,\beta} = \bar b_{\bar\alpha}\} \] has cardinality $\lambda$.  

Recall that all of this is for a fixed $\bar\alpha$, $\langle s_0, \dots, s_{k-1} \rangle$. 

Now continue to fix $s_0 <_I \dots <_I s_{k-1}$. For every $\gamma < \lambda$ we let 
$\bar\alpha_\gamma = \langle \alpha_{\gamma,\eta}:\eta \in \frt(\Lambda)\rangle$
be defined by
\[ \alpha_{\gamma,\eta} = \min(\cU_{\mathbf x,\eta} \backslash \gamma) \]
which is well defined as $\cU_{\mathbf x,\eta}$ is an unbounded subset of
$\lambda$.  So for every $\gamma < \lambda$, we have that $\bar\alpha_\gamma,
\bar b_{\bar\alpha_\gamma}, \cW_{\bar\alpha_\gamma}$ are well defined. 
So for some $\bar b_*$ we have 
\[ \cW_{\bar b_*} = \{\gamma < \lambda:\bar b_{\bar\alpha_\gamma} = \bar b_*\} \]
is an unbounded subset of $\lambda$.
\medskip

\noindent
We can now define the necessary objects, in two cases. For Case 1 below, 
the final definitions depend on whether the node we are dealing with has a sibling or not.  

\underline{Case 1}:   $1 = |\Lambda_\bx \cap\{\varrho \char 94 \langle 0 \rangle,\varrho \char 94 \langle 1 \rangle\}|$

Let $\iota$ be such that $\varrho \char 94 \langle \iota\rangle \notin \Lambda_{\mathbf x}$. 
We define $\mathbf y$ as follows:
\begin{enumerate}
\item[$\oplus$]  $(a) \quad \Lambda_{\mathbf y} = \Lambda_{\mathbf x}
  \cup \{\varrho \char 94 \langle \iota \rangle\}$, hence
$\frt(\Lambda_{\mathbf y}) = \frt(\Lambda) \cup\{\varrho \char 94 \langle
  \iota \rangle\} \backslash \{\varrho\}$
\item[${{}}$]  $(b) \quad \bar b_{\mathbf y,\eta}$ is $\bar b_{\mathbf
  x,\eta}$ if $\eta \in \Lambda_{\mathbf x}$
\item[${{}}$]  $(c) \quad \bar b_{\mathbf y,\eta}$ is $\bar b_*$ if 
$\eta \in \varrho \char 94 \langle \iota\rangle$
\item[${{}}$]  $(d) \quad$ if $\eta \in \frt(\Lambda) \backslash
  \{\varrho\}$ then $\cU_{\mathbf y,\eta}$ is $\{\alpha \in \cU_{\mathbf
  x,\eta}$: for some $\gamma < \lambda,\bar b_{\bar\alpha_\gamma} =$

\hskip25pt  $\bar b_*$ and $\alpha = \alpha_{\gamma,\eta}\}$
\item[${{}}$]  $(e) \quad$ if $\eta = \varrho \char 94 \langle
  \iota\rangle$ then $\cU_{\mathbf y,\eta} = \{\alpha \in 
\cU_{\mathbf x,\varrho}$: for some $\gamma < 
\lambda,b_{\bar\alpha_\gamma} = \bar b_*$ 

\hskip25pt and $\alpha = \alpha_{\gamma,\varrho}\}$.
\end{enumerate}
For clause (e) recall $\bar b_* \in \cx_{\bar\alpha,\varrho} \subseteq \cZ_{\bar\alpha_\gamma,\varrho}$.

As $\varrho \in \frt(\mathbf x)$, we are left with Case 2:

\underline{Case 2}:   $0 = |\Lambda_{\mathbf x} \cap \{\varrho \char
94 \langle 0 \rangle,\varrho \char 94 \langle 1 \rangle\}|$ and let
$\iota=0$.   
The only difference is that 
\begin{enumerate}
\item[$\bullet$]  in clause (a), $\frt(\Lambda_{\mathbf y}) = \Lambda
  \cup \{\varrho \char 94 \langle \iota \rangle\}$
\end{enumerate}
so we have to add
\begin{enumerate}
\item[$(f)$]  if $\eta = \varrho$ then $\cU_{\mathbf y,\eta} =
\cW_{\bar\alpha_\gamma}$ for some $\gamma < \lambda$ such that $\bar
  b_{\bar\alpha_\gamma} = \bar b_*$.
\end{enumerate}

\br
\step{Step 5: Finish.} 
Recalling Step 2, we have shown that the assumption of $|A|^+$-pairwise 1-contradictory higher formulas contradicts $NSOP_2$. This completes the proof. 
\end{proof}

\br
In this theorem, 
replacing 1-contradictory by $n$-contradictory would be straightforward as we
can replace $\varphi$ by a conjunction, defined as follows. 

%$\varphi_{[n]}$

\begin{defn}  \label{b4}
%\emph{ }
%\begin{enumerate}
%\item For $\varphi = \varphi(\bar x,\bar y) \in \bbL(\tau_T)$, let 
%$\mathbf n_{\SOP_2}(\varphi) = \mathbf n_{\SOP_2}(\varphi,T)$ be the minimal
%   $n \le \omega$ such that $n = \omega$ or there are no $\bar a_\eta
%   \in {}^{\ell g(\bar y)}\gC$ for $\eta \in {}^{n>}2$ such that
%   
%\begin{enumerate}
%\item[$(a)$]  $\varphi(\bar x,\bar a_\eta),\varphi(\bar x,\bar a_\nu)$
%  are contradictory when $\eta \perp \nu$ are from ${}^{n>}2$
% 
%\item[$(b)$]  $\gC \models (\exists \bar x) 
%\bigwedge\limits_{\ell \le \ell g(\eta)}
%(\bar x,\bar a_{\eta \rstr \ell})$ for $\eta \in {}^{n>}2$.
%\end{enumerate}
%
%\item  To say that $\varphi = \varphi(\bar x,\bar y)$ has NSOP$_2$ means that 
%$\mathbf n_{\SOP_2}(\varphi)$ is finite.
%\item 
Given $\vp$ and $n$, let 
\begin{align*}
\varphi_{[n]} = &  \varphi_{[n]}(x,\bar{y}_{[n]}) \\ 
 = & \varphi_{[n]}(\bar x,\langle y_i: i < \ell g(\bar y) \cdot n\rangle) \\
  =  & \bigwedge\limits_{\ell < n} \varphi(\bar x,
\langle y_{\ell g(\eta)\cdot \ell +j}:j < \ell g(\bar y)\rangle).
\end{align*}

%\item 
We say $\varphi = \varphi(\bar x,\bar y)$ has NSOP$_2$ \emph{robustly} when 
%the number 
%$\mathbf n_{\SOP_2}(\varphi_{[n]})$ is finite for every $n$, i.e. 
no $\vp_{[n]}$ has $SOP_2$. 
%\end{enumerate}
\end{defn}

\begin{obs}
If $T$ is $NSOP_2$ then every $\vp(\bar{x}, \bar{y}) \in \ml(\tau_T)$ has $NSOP_2$ robustly in $T$. 
\end{obs}

%\newpage

\vspace{5mm} 

In the last main result of this section, we apply Theorem \ref{tstar-sop2-max} characterizing the maximal class in $\trianglelefteq^*$ 
to prove another property of $SOP_2$, related to so-called exact saturation (meaning for which singular $\kappa$ 
a given theory $T$ has a model which is $\kappa$-saturated but not $\kappa^+$-saturated).
Exact saturation was studied in Shelah \cite{Sh:900} \S 2 (pps. 31-37) and 
in a manuscript of Kaplan and Shelah \cite{KpSh:F1473} in preparation, which deal with elementary classes and
$(\mathbf{D}, \kappa)$-sequence homogeneity. 

\begin{defn}[ see \cite{Sh:900}, \cite{KpSh:F1473} ] For any theory $T$, define 
\begin{align*} \truespec(T) =  & \{ \kappa \colon \kappa \geq |T|, \mbox{ $\kappa$ singular } \\ 
&\mbox{and there exists a $\kappa$-saturated not $\kappa^+$-saturated model of $T$} \}. \end{align*}
\end{defn}
The papers \cite{Sh:900} and \cite{KpSh:F1473} find sufficient conditions but not necessary and 
sufficient conditions for exact saturation. 
In the present paper, we define a pseudo-elementary version\footnote{Compare the results of \cite{Sh:a} VI.5 connecting the minimum 
class in Keisler's order to saturation properties of a $PC$-class.} of this 
spectrum, which we will connect to $SOP_2$ via $\trianglelefteq^*$. 
\begin{defn} The exact saturation spectrum for $T$ is defined to be the set:
\begin{align*}
\spec(T) = & \{ (\kappa, \mu)  \colon \kappa \geq \mu \geq |T|, \mbox{ $\kappa$ singular} \mbox{ and for any  
$T_1 \supseteq T$, $|T_1| \leq \mu$, } \\ & \mbox{ there is a $\kappa$-saturated not $\kappa^+$-saturated member of $PC(T_1, T)$} \}. 
\end{align*}
\end{defn} 
In future work, we hope to be able to get necessary and sufficient conditions for $\spec(T)$ to be empty 
(at least restricting ourselves to $\kappa$ strong limit of large enough cofinality), and it seems plausible 
that this may be $SOP_2$.  
Here, using the methods of Section \ref{s:tlf}, we will prove one direction: if $T$ has $SOP_2$ then $\spec(T) = \emptyset$, 
and discuss several open questions.  
First we recall two known examples. 

\begin{fact}[\cite{Sh:900}, \cite{KpSh:F1473}] \label{dlo-fact}
Let $T$ be the theory of dense linear order without endpoints. Then for any singular $\kappa \geq |T|$, if $M \models T$ is $\kappa$ saturated then it is $\kappa^+$-saturated.  
Thus $\truespec(T) = \emptyset$. %, and a fortiori $\spec(T) = \emptyset$.  
\end{fact}

\begin{proof}
By quantifier elimination, it suffices to show that the cofinality and coinitiality of $M$ are at least $\kappa$ and for any regular $\kappa_1, \kappa_2$, 
every $(\kappa_1, \kappa_2)$-pre-cut is filled. Since cofinality and coinitiality of the model and of pre-cuts are necessarily regular cardinals, the result is immediate. 
\end{proof}

\begin{fact}[\cite{Sh:900} Example 2.23] \label{chang-fact}
There is a theory $T$ with the independence property such that: if $T$ has an exactly $\kappa$-saturated model then 
$\kappa$ is regular. (In fact it is necessary and sufficient that $\kappa$ be regular.) 
Thus $\truespec(T) = \emptyset$. %, and a fortiori $\spec(T) = \emptyset$.  
\end{fact}

%Finally, we analyze trees of types, working towards Theorem \ref{t:trees} which 
%say essentially that if $T$ is $NSOP_2$ we can always amalgamate suitably rich trees of copies 
%of a given type.  

Recall the order $\tlf^*$ from Definition \ref{tstar-defn} above.  
%[Question: The order $\tlf^*$ is stated only for regular cardinals. I will continue assuming that it makes sense for any cardinal.]
In the following, countability is not essential.  

\begin{obs} \label{combining}
Suppose $T_0$, $T_1$ are countable theories which are equivalent under the $\trianglelefteq^*$ order. Then there 
exists a theory $T_*$ which interprets both $T_0$ and $T_1$, say via 
$\overline{\vp_0}$ and $\overline{\vp_1}$ respectively, such that: for any model $M \models T_*$, and any uncountable $\kappa$, 
\[ {N}^{[\overline{\vp_0}]} \mbox{ is $\kappa$-saturated if and only if } {N}^{[\overline{\vp_1}]} \mbox{ is $\kappa$-saturated.} \]
\end{obs}

\begin{proof} 
We may assume $T_0$ and $T_1$ have no finite models. 
Let $\mu$ be such that we have $T_a$ witnessing $T_0 \trianglelefteq^* T_1$ 
and $T_b$ witnessing $T_1 \trianglelefteq^* T_0$, 
and we may assume $|T_a|, |T_b| < \mu$.  
Let $\overline{\vp}^a_\ell$ interpret $T_\ell$ in $T_a$ for $\ell = 0,1$, and let
$\overline{\vp}^b_\ell$ interpret $T_\ell$ in $T_b$ for $\ell = 0, 1$.  
Without loss of generality:
\begin{itemize}
\item $\tau(T_0)$ and $\tau(T_1)$ are disjoint. 
\item $\overline{\vp}^a_\ell$ says that for any model $M \models T_a$ the universe of 
${M}^{[\overline{\vp}^a_\ell]}$ is $(P^a_\ell)^M$, for some unary predicate $P^a_\ell$, and $T_a$ implies that 
$P^a_{0}$ and $P^a_1$ are disjoint;
\item similarly for $T_b$. 
\item $\overline{\vp}^a_\ell$ is the identity, where this means: 
\begin{itemize}
\item $\tau(T_a) \supseteq (\tau(T_0) \cup \tau(T_1))$
\item $P \in \tau(T_\ell)$ implies $\vp^a_{\ell, P} = P$, and $T_a \vdash (\forall{\bar{x}})(P(\bar{x}) \implies \forall_{\ell <\lgn(\bar{x})} P_{\ell,a}(x))$
\item similarly for function symbols $F \in \tau(T_\ell)$, adding that they are interpreted as partial functions with 
domain the predicate $P^a_\ell$ $($alternately, we could have assumed without loss of generality that $\tau(T_\ell)$ has only predicates$)$.
\end{itemize}
\item each $\overline{\vp}^b_\ell$ is likewise the identity.
\item $\tau(T_a) \cap \tau(T_b) = (\tau(T_0) \cup \tau(T_1))$. 
\end{itemize}

What remains is to find models $M_a \models T_a$ and $M_b \models T_b$ such that the interpretation of $T_\ell$ in $M_a$ and $M_b$ 
are isomorphic.  Given such models, we may finish by adding additional symbols giving 
a bijection between the respective interpretations of $T_0$ and $T_1$. 

One way to construct such models is to recall that ultrapowers commute with reducts, and that any two elementarily equivalent 
models have isomorphic ultrapowers. Begin with $M_{0,a} \models T_a$ and $M_{0,b} \models T_b$. Choose the ultrafilter $\de_0$ 
so that the $\de_0$-ultrapowers of ${M}^{[\overline{\vp}^a_0]}_{0,a}$ and ${M}^{[\overline{\vp}^a_{0}]}_{0,b}$ are isomorphic.  
Call these ultrapowers $M_{1,a}$ and $M_{1,b}$ respectively. 
(As ultrapowers commute with reducts, we may consider $M_{1,a}$ and $M_{1,b}$ as models of the full $T_a$ and $T_b$ respectively.) 
Consider the model $M_c$ which is the disjoint union of $M_{1,a}$ and $M_{1,b}$ and expand $M_c$ by adding symbols giving the isomorphism between 
${M}^{[\overline{\vp}^a_0]}_{1,a}$ and ${M}^{[\overline{\vp}^a_{0}]}_{1,b}$.  
Let ${M}^{[\overline{\vp}^a_1]}_{c}$ have the obvious meaning of ${M}^{[\overline{\vp}^a_1]}_{1,a}$ considered within this model $M_c$.  
Next, choose the ultrafilter $\de_1$ so that ${M}^{[\overline{\vp}^a_1]}_{c}$ and ${M}^{[\overline{\vp}^a_{1}]}_{c}$ are isomorphic. 
Let $M_*$ be the $\de_1$-ultrapower of $M_c$. In this model we may expand $T_c$ so as to make the interpretations of $T_1$ 
isomorphic. Let $T_*$ be the theory of this expanded model. This completes our construction.

We had fixed an infinite $\mu$ so that $|T_a| + |T_b| < \mu$, and $|T_0| + |T_1| = \aleph_0$ by assumption. The construction gives that also $|T_*| < \mu$.   
\end{proof}

\begin{cor} \label{putting-together} \emph{ }
\begin{enumerate}
\item Let $T_0$ be any theory for which $\truespec(T_0) = \emptyset$. Let $T_1$ be any theory such that 
$T_0$ and $T_1$ are equivalent in the order $\trianglelefteq^*$. Then $\spec(T_1) = \emptyset$. 
\item Suppose $T_0$, $T_1$ are $\tlf^*$-equivalent, as witnessed by $T_*$ with $|T_*| \leq \mu$.  
Suppose $\kappa \geq \mu$ is singular and $(\kappa, \mu) \notin \spec(T_0)$. Then $(\kappa, \mu) \notin \spec(T_1)$. 
\end{enumerate}
\end{cor}

\begin{proof} 

(1)
Let $T_*$ be a theory witnessing their equivalence, such as that given by Observation \ref{combining}. Let 
Let $\overline{\vp_0}$ and $\overline{\vp_1}$ witness the interpretations for $T_0$ and of $T_1$, 
respectively, and let ${N}^{[\overline{\vp}]}$ and ${N}^{[\overline{\vp_*}]}$ denote the respective interpretations 
in a given model $N \models T_*$.  
Now let $\kappa$ be a singular cardinal, $\kappa > |T_1|$ and 
suppose $M \in PC(T_*, T_1)$ is $\kappa$-saturated.  Let $N \models T_*$ with $M = N \rstr \tau(T_1)$ witness that $M \in PC(T_*, T_1)$. 
The fact that $M$ is $\kappa$-saturated says precisely that ${N}^{[\overline{\vp_1}]}$ is $\kappa$-saturated.  
By hypothesis of equivalence, ${N}^{[\overline{\vp_0}]}$ is $\kappa$-saturated as well. By choice of $T_0$, 
${N}^{[\overline{\vp_1}]}$ is $\kappa^+$-saturated. Applying the hypothesis of equivalence in the other direction, 
${N}^{[\overline{\vp_0}]}$ is $\kappa^+$-saturated.  This completes the proof. 

(2) 
As $(\kappa, \mu) \notin \spec(T_0)$ there is a theory $T_{00} \supseteq T_0$, $|T_{00}| \leq \mu$, and a 
$\kappa$-saturated not $\kappa^+$-saturated member of $PC(T_{00}, T_0)$.  In order to make use of this, we 
would first modify the construction of the theory $T_*$ (by means of a better choice of $T_a$, $T_b$) 
so that for any $M_* \models T_*$ we have that 
${M_*}^{[\overline{\vp_0}]} \in PC(T_{00}, T_0)$. Since $|T_{00}| \leq \mu$, we can do this while preserving $|T_*| \leq \mu$. 
Then the proof continues as in part (1). 
\end{proof}

\begin{theorem} \label{ess-sop2}
If $T$ has $SOP_2$ then $\spec(T) = \emptyset$. That is,  
if $T$ has $SOP_2$ then for some $T_1 \supseteq T$ of cardinality $|T|$, for every singular $\kappa > |T|$, 
if $M \in PC(T_1, T)$ is $\kappa$-saturated then it is $\kappa^+$-saturated.
\end{theorem}

\begin{proof}
By Theorem \ref{tstar-sop2-max}, Fact \ref{dlo-fact} and Corollary \ref{putting-together}(1). 
\end{proof}

\br

\begin{qst} 
Does Theorem $\ref{ess-sop2}$ hold for the theory $\trg$ of the random graph?
\end{qst}

In the case of $\truespec(\trg)$ rather than $\spec(\trg)$, note that: 

\begin{claim}
If $\kappa > \cf(\kappa) + |T|$ then $\kappa \in \truespec(\trg)$. 
\end{claim}

\begin{proof}
Let $M$ be a $\kappa^+$-saturated model of $\trg$. Let $\langle a_\alpha : \alpha < \kappa \rangle$ be pairwise distinct members of $M$ 
which form an empty graph, i.e. $\alpha < \beta < \lambda \implies M \models \neg R(a_\alpha, a_\beta)$. Consider the submodel $N \subseteq M$
whose domain is 
\[  \{ b \in M : \left( \exists^{< \kappa} \alpha < \kappa \right) (~ R(b,a_\alpha)~ ) \}. \]
Then $N$ is a model of $\trg$ which is as required: it is $\kappa$-saturated but $\{ R(x,a_\alpha) : \alpha < \kappa \}$ is omitted so it is not 
$\kappa^+$-saturated. 
\end{proof}

%\begin{qst}
%For $T$ with Skolem functions, there is a construction of a $\kappa$-saturated model extending a given $M$ in 
%\cite{Sh:a} VII \S 4, using averages on Sk(B), $|B|<\kappa$.  Without this assumption on $T$ using $\mcp(n)$-diagrams, 
%do we have a parallel result? 
%\end{qst}

Before giving some further evidence, 
we record here that there are other natural directions these investigations could take. 

\begin{disc} 
Instead of asking whether $\kappa \in \spec(T)$ for $\kappa > \cf(\kappa) + |T|$, we may fix a specific way to construct a $\kappa$-saturated model 
and ask if this implies that the constructed model is $\kappa^+$-saturated. For example: 
\begin{itemize}
\item[(a)] we may consider $T$ dependent, $|M| = \{ a_\alpha : \alpha < \alpha_* \}$, $u_\alpha \in [\alpha]^{<\kappa}, \tp(a_\alpha, \{ a_\beta : \beta < \alpha \} )$ 
does not split over $\{ a_\beta : \beta < \alpha \}$, and $M$ is $\kappa$-saturated; see \cite{Sh:715}. 
\item[(b)] as in (a), but we may ask that $T$ has Skolem functions. 
\item[(c)] as in (b), and also $\tp(a_\alpha , \{ a_\beta : \beta < \alpha \})$ is finitely satisfiable in $\operatorname{Sk}(\{ a_\beta : \beta \in u_\alpha \})$, 
see \cite{Sh:a}, VII, section 4. 
\item[(d)] we may consider models built using $\mcp^{-}(n)$-diagrams in some explicit way. 
\end{itemize} 
\end{disc}

\noindent Returning to $\spec$, on the positive side we can settle the case of stable $T$. 

\begin{claim} \emph{ } Let $T$ be a complete theory. 
%\begin{enumerate}
%\item[(a)] If $\kappa$ is a singular cardinal, $\kappa \geq |T|$ and $T$ is $\kappa$-stable, then 
%%for any singular $\kappa \geq |T|$, with $\kappa^{|T|} = \kappa$, 
%$\kappa \in \truespec(T)$.  
%
%\item[(b)] In particular, if $T$ is stable, $\kappa = \kappa^{|T|} > \operatorname{cf}(\kappa)$, then $\kappa \in \truespec(T)$. 
%%$\lambda$ is singular, $\lambda \geq 2^{|T|}$ and $\lambda =\lambda^{<\kappa(T)}$, then $\kappa \in \truespec(T)$. 
%\item[(c)] 
If $T$ is stable, $\kappa = \kappa^{|T|} > \operatorname{cf}(\kappa)$, and $\kappa \geq \mu \geq T$,
then $(\kappa, \mu) \in \spec(T)$. 
%If $T$ is $\kappa$-stable, then for any singular $\kappa$ and any $\mu$ such that $\kappa \geq \mu \geq |T|$, $(\kappa, \mu) \in \spec(T)$. 
%\end{enumerate}
\end{claim}

\begin{proof} 
Recall that under these cardinal hypotheses there is a model of $T$ of cardinality $\kappa$ which is $\kappa$-saturated, 
see \cite{Sh:a} Chapter III Theorem 3.12. Clearly this model is not $\kappa^+$-saturated.  
%Then (b) follows e.g. by \cite{Sh:a} Chapter III Theorem 5.15. 
Given any $T_1 \supseteq T$, $|T_1| \leq \mu$, we may build a model of $T_1$ of cardinality $\kappa$ in the same way 
whose restriction to $T$ is $\kappa$-saturated %(begin with $M \models T_1$ of cardinality $\mu$ and  
but evidently not $\kappa^+$-saturated. 
%
%
%
% For any such $\kappa$, there is a model of $T$ of cardinality $\kappa$ which is saturated. 
%
%Let such $\kappa$ and $\mu$ be given and fix $T_1 \supseteq T$, $|T_1| \leq \mu$. 
%
%%\\ & \mbox{ there is a $\kappa$-saturated not $\kappa^+$-saturated member of $PC(T_1, T)$} \}. 
\end{proof}

\begin{qst}
What about simple theories?
\end{qst}

\begin{prob}
Determine whether the following is true:  for $T$ complete and countable, 
\[ \spec(T) = \emptyset \mbox{ if and only if } T \mbox{ has $SOP_2$.} \]
\end{prob}

\br

These investigations suggest some interesting parallel questions for ultrafilters. The reader 
may recall that if $\lambda$ is a singular cardinal and the ultrafilter is $\lambda$-good then it is $\lambda^+$-good 
(since linear order is in the maximal Keisler class, \cite{Sh:a}.VI.2, a proof similar to that of 
Fact \ref{dlo-fact} goes through). 

\begin{prob}
Suppose $\lambda$ is singular. Does there exist a regular ultrafilter which $\lambda$-saturates 
but does not $\lambda^+$-saturate ultrapowers of the random graph? 
\end{prob}

%A harder question is the following 
The property of an ultrafilter being ``$(\lambda, \aleph_0)$-perfect'' was defined in \cite{MiSh:1050}.  
%If the answer to the first part of Problem \ref{p:perfect} is no, this would 

\begin{prob} \label{p:perfect}
Suppose $\lambda$ is singular. 
If $\de$ is a regular ultrafilter which is $(\kappa, \aleph_0)$-perfect for every $\kappa < \lambda$, 
is $\de$ also $(\lambda^+, \aleph_0)$-perfect?  If not, must it at least produce $\lambda^+$-saturated 
ultrapowers of the random graph? 
\end{prob}

\newpage

\newpage

\end{document}